\documentclass{article}
\usepackage{latexsym,amsmath,amssymb}
\usepackage{amsthm,upref}

\def\Z{\mathbb{Z}}

\def\sphere{\mathbb{S}}
\def\disk{\mathbb{D}}

\def\qdl{\mathrm{Q}}
\def\rck{\mathrm{R}}
\def\dgn{\mathrm{D}}

\def\im{\mathrm{im}\,}
\def\sh{{\mathrm{sh}\,}}

\def\ini{\mathrm{ini}}
\def\ter{\mathrm{ter}}

\def\blhd{\blacktriangleleft}
\def\lhd{\vartriangleleft}

\newtheorem{theorem}{Theorem}[section]

\newtheorem{lemma}[theorem]{Lemma}

\newtheorem{corollary}[theorem]{Corollary}

\theoremstyle{definition}
\newtheorem{remark}{Remark}[section]

\theoremstyle{definition}
\newtheorem{example}{Example}[section]

\begin{document}

\title{A Diagrammatic Construction of \linebreak[4]
Third Homology Classes of Knot Quandles}
\author{Yasto Kimura}
\date{}

\maketitle

\begin{abstract}
We construct elements of the third quandle homology groups 
of knot quandles, which are called the shadow fundamental classes. 
They play the same roles for the shadow quandle cocycle 
invariants of knots as the fundamental classes of knot 
quandles does for the quandle cocycle invariants. 
As an application of the shadow fundamental classes, 
we show the relation between 
the shadow quandle cocycle invariants and 
the based shadow quandle cocycle invariants. 

Moreover,
we will show, for a prime knot, 
that any third quandle homology classes 
are considered as images of the shadow fundamental classes 
of some links. 
\end{abstract}

\section{Introduction.}
\label{CHAP:introduction}

Knot quandles, introduced by Joyce [J], 
are algebraic concepts, which are related to 
the peripheral systems of knots. 
Their (co)homology theories are 
defined in [CJKLS], from which we can obtain 
many invariants of classical knots and 
of surface knots. 
The cocycle invariants in that paper and the shadow cocycle 
invariants in [CKS1] are examples of those invariants. 

Eisermann [E] determined the second (co)homology groups 
of knot quandles.
Furthermore, as a consequence of his result, 
we have homological interpretation of 
the cocycle invariants. 
Our purpose comes directly from these results, that is, 
we will try to determine the third (co)homology groups 
of knot quandles, and to give a homological explanation 
to the shadow cocycle invariants.\\

The arcs of a diagram of a knot \(K\) 
can be considered as generators of the knot quandle 
\(Q(K)\) of \(K\). 
Thus we can regard the diagram as coloured by \(Q(K)\). 
Since any coloured diagram on \(\sphere^2\) is 
shadow colourable, 
the knot diagram has a \(Q(K)\)-shadow colouring, 
which gives a third homology class of \(Q(K)\). 
We call such a homology class a shadow diagram class 
as an element of the rack homology group of \(Q(K)\), 
and also call it a shadow fundamental class 
as that of the quandle homology group. 
As for these homology classes, 
we have: 

\begin{theorem}[Theorem \ref{TH:construction theorem}]
\label{TH:main theorem a}
Let \(L\) be a non-trivial \(n\)-component link. 
\begin{itemize}
\item[a)] The shadow diagram classes are non-zero 
elements of the third rack homology group \(H^\rck_3(Q(L); \Z)\) 
for all diagrams of \(L\). 
\item[b)] There exist \(n\) shadow fundamental classes of \(L\) 
in the third quandle homology group \(H_3^\qdl(Q(L); \Z)\). 
\item[c)] The third quandle homology group 
\(H_3^\qdl(Q(L); \Z)\) splits into the direct sum 
\begin{quote}
\(\bigl(\bigoplus \Z [L_i]\bigr) 
\oplus \Bigl(H_3^\qdl(Q(L); \Z)/\bigl(\bigoplus \Z [L_i]\bigr)\Bigr)\), 
\end{quote}
where \([L_1]\), \ldots, \([L_n]\) are 
distinct shadow fundamental classes of \(L\). 
\end{itemize}
\end{theorem}

Also we show the relation of the shadow fundamental 
classes to the shadow cocycle invariants, 
and apply it to generalise a result in [S]. 
In the proof of the result, 
Satoh used the fact that the set of 
\(\Z_p\)-shadow colourings becomes a module. 
We show that Satoh's result derives essentially from 
the connected-ness of \(\Z_p\) 
by using the shadow fundamental classes, 
or, in other words, we prove the following theorem in general.

\begin{theorem}[Theorem \ref{TH:invariants relation}]
\label{TH:main theorem b}
Let \(X\) be a connected finite quandle 
and \(\phi\) a quandle \(3\)-cocycle of \(X\).
For any link \(L\), there holds an equation
\begin{quote}
\(\Phi_\phi(L) = |X| \cdot \Phi^\ast_\phi(L)\)
\end{quote}
between the shadow cocycle invariant \(\Phi_\phi(L)\)
and the based shadow cocycle invariant \(\Phi^\ast_\phi(L)\).
\end{theorem}

In respect of the first motivation, 
by the method in [CKS1] we obtain 
a generalised knot diagram on a closed manifold 
which represents the given third homology class. 
Considering the possible surgeries on 
these diagrams, we can obtain a knot diagram 
in usual sense as a representative of the given 
homology class when a knot \(K\) is prime, 
that is, we conclude:

\begin{theorem}[Theorem \ref{TH:topological realisation}]
\label{TH:main theorem c}
Let \(K\) be a prime knot. 
For any quandle \(3\)-cycle \(c \in C^\qdl_3(Q(K); \Z)\), 
there exists a pair of a link \(L\) 
and a homomorphism \(f\): \(Q(L) \to Q(K)\) 
such that \([c] = f_\ast[L_\sh]\), 
where \([L_\sh]\) is one of the shadow fundamental 
classes of \(L\).
\end{theorem}

This theorem indicates one way to represent 
elements of the third quandle homology groups. 
We expect that the diagrammatic representation is 
useful for complete determination of the 
third homology groups of quandles.

\section{Preliminaries.}
\label{CHAP:preliminaries}

In this section, 
we will define the concepts of 
racks and quandles, 
their (co)ho\-mol\-o\-gy theories, 
diagrams and their (shadow) colourings, 
and knot quandles. 

\subsection{Racks and Quandles.}
\label{SEC:rack quandle}

Quandles are introduced by Joyce [J]
for characterising unframed knots.
He proved that the quandles associated with knots
completely determine the knot types up to orientation,
and explained Fox's tricolourability of knots
in the viewpoint of quandles.
Racks introduced by Fenn-Rourke [FR]
are proved to be complete invariants of framed knots.
Both quandles and racks obey
the axioms of right invertibility 
and of right distributivity,
but only quandles satisfy the idempotency.

A \textbf{rack} \(R\) is a non-empty set
with two binary operations \(\lhd\) and \(\blhd\)
satisfying two axioms
\begin{quote}
(R1) \quad 
\((a \lhd b) \blhd b = a = (a \blhd b) \lhd b\) and \\
(R2) \quad 
\((a \lhd b) \lhd c = (a \lhd c) \lhd (b \lhd c)\)
\end{quote}
for any \(a\), \(b\) and \(c \in R\).
In addition, if a rack \(Q\) satisfies an axiom 
\begin{quote}
(Q) \quad 
\(a \lhd a = a = a  \blhd a\)
\end{quote}
for any \(a \in Q\), 
it is called a \textbf{quandle}. 
For two racks \(R\) and \(R'\), 
a map \(f\): \(R \to R'\) 
is called a (rack) \textbf{homomorphism} 
when it obeys 
\begin{quote}
(RH) \quad 
\(f(a \lhd b) = f(a) \lhd f(b)\) 
\end{quote}
for any \(a\) and \(b \in R\).
Notice that the other right-distributive 
laws for \(\lhd\) and \(\blhd\) can be derived 
from the axioms (R1) and (R2), 
and that the compatibility of a homomorphism 
for \(\blhd\) is a consequence of (R1) and (RH). 
Note that sometimes we write 
\(\lhd^{+ 1}\) and \(\lhd^{- 1}\) 
instead of \(\lhd\) and \(\blhd\), respectively. \\

The most important example of quandles 
is the \textbf{conjugation quandle} of a group. 
A group \(G\) can be viewed as a quandle 
with operations defined by 
\begin{quote}
\(g \lhd h = h^{- 1} g h\) 
and \(g \blhd h = h g h^{- 1}\)
\end{quote}
for \(g\) and \(h \in G\). 
We denote by \(G_{\mathrm{conj}}\) 
the conjugation quandle derived from \(G\). 

On the other hand, we can associate a group \(G_Q\) with 
a quandle \(Q\). 
The group \(G_Q\) is generated by \(Q\) 
with relations 
\begin{quote}
\(b^{- 1} a b = a \lhd b\) and \(b a b^{- 1} = a \blhd b\) 
\end{quote}
for each \(a\) and \(b \in Q\).
We call \(G_Q\) the \textbf{associated group} of \(Q\). 
There is a canonical homomorphism from \(Q\) to 
\((G_Q)_{\mathrm{conj}}\) which maps \(a \in Q\) 
to \(a \in (G_Q)_{\mathrm{conj}}\). \\

A quandle \(Q\) is called \textbf{trivial} 
if \(a \lhd b = a\) holds for any \(a\) and \(b \in Q\). 
We denote by \(T_n\) the trivial quandle with \(n\) elements. 
Notice that the conjugation quandle of a commutative group \(G\) 
is trivial.\\

For two elements \(a\) and \(b\) in a rack \(R\), 
they are said to be \textbf{connected} 
if there exist finite sequences \((x_1, \ldots, x_n)\) 
and \((\epsilon_1, \ldots, \epsilon_n)\),
where \(x_i\) is in \(R\) and \(\epsilon_i\) 
is in \(\{\pm 1\}\) for each \(i\), such that 
\begin{quote}
\((\cdots (a \lhd^{\epsilon_1} x_1) \cdots ) \lhd^{\epsilon_n} x_n 
= b\)
\end{quote}
holds. 
We call the set of elements connected 
to \(a\) the \textbf{orbit} of \(a\). 
If all elements of \(R\) are connected to one another, 
the rack \(R\) itself is called 
to be \textbf{connected}. 

For example, the trivial quandle has 
no pair of connected elements. 
On the other hand, 
the \textbf{dihedral quandle} \(\Z_p = \Z/p\Z\), 
which has operations defined by 
\begin{quote}
\(a \lhd b = a \blhd b = 2b - a\), 
\end{quote}
is connected when \(p\) is an odd prime.

\subsection{Homology and cohomology theories}
\label{SEC:homology cohomology}

Here we define (co)homology theories of a rack or a quandle. 
We will follow the definitions in [CJKS1] 
except the coefficients of boundary maps. 
It is because, under this modification of boundary maps, 
the shifting homomorphisms in \S \ref{CHAP:shadow class} 
become easily explained diagrammatically. 
Applications and computations of (co)homology theories 
can be found in [CJKS2] or 
in the other papers of the same authors. \\

Let \(R\) be a rack and \(A\) a ring with unit. 
We denote, by \(C^\rck_n(R; A)\), a free \(A\)-module \(A[R^n]\) 
for a positive integer \(n\) and, by \(C^\rck_0(R; A)\), \(A\) itself. 
Then, \(C^\rck_\ast(R; A)\) is a chain complex 
with boundary maps \(\partial_n\): 
\(C^\rck_n(R; A) \to C^\rck_{n - 1}(R; A)\) 
defined by linearly extending a map on the basis 
\begin{quote}
\(\partial_n(x_1, \ldots, x_n) \\
{}\hskip30pt = \sum\limits_{i = 1}^n (- 1)^{n - i} 
\{(x_1 \lhd x_i, \ldots, 
x_{i - 1} \lhd x_i, x_{i + 1}, \ldots, x_n) \\
{}\hskip120pt - (x_1, \ldots, x_{i - 1}, x_{i + 1}, \ldots, x_n)\}\) 
\end{quote}
for \(n \ge 2\) and \(\partial_1(x) = 1\), 
where \(x_1\), \ldots, \(x_n\) and \(x\) are in \(R\).
As usual, we set \(Z^\rck_n(R; A) = \ker \partial_n\) 
and \(B^\rck_n(R; A) = \im \partial_{n + 1}\). 
The \(n\)-th \textbf{rack homology} \(H^\rck_n(R; A)\) of \(R\) 
is defined as the quotient module \(Z^\rck_n(R; A) / B^\rck_n(R; A)\). 
The elements of \(C^\rck_n(R; A)\), \(Z^\rck_n(R; A)\) 
and \(B^\rck_n(R; A)\) are called, respectively, 
\textbf{rack \(n\)-chains}, \textbf{cycles} 
and \textbf{boundaries}. 
If the difference of two rack chains \(c_1\) and \(c_2\) 
is a rack boundary, then they are said to be 
\textbf{rack homologous}.

Considering the dual concepts \(C_\rck^n(R; A) 
= C^\rck_n(R; A)^\ast\) and 
\begin{quote}
\(\delta^n = \partial_n^\ast\): 
\(C_\rck^{n - 1}(R; A) \to C_\rck^n(R; A)\), 
\end{quote}
we obtain a cochain complex \(C_\rck^\ast(R; A)\). 
The \(n\)-th \textbf{rack cohomology} 
\(H_\rck^n(R; A)\) of \(R\) is 
the quotient \(Z_\rck^n(R; A) / B_\rck^n(R; A)\), 
where \(Z_\rck^n(R; A)\) denotes \(\ker \delta^{n + 1}\) 
and \(B_\rck^n(R; A)\) \(\im \delta^n\). 
We call the elements of \(C_\rck^n(R; A)\), 
\(Z_\rck^n(R; A)\) and \(B_\rck^n(R; A)\) 
\textbf{rack \(n\)-cochains}, 
\textbf{cocycles} and \textbf{coboundaries}, 
respectively. 
Also we call two rack cochains \(\phi_1\) and \(\phi_2\) 
\textbf{rack cohomologous} when \(\phi_1 - \phi_2\) 
is a rack coboundary. 
Since \(C^\rck_n(R; A)\) is a free module, 
rack \(n\)-cochains can be considered 
as maps from \(R^n\) to \(A\). \\

For a quandle \(Q\), 
we can define other homology and cohomology theories of \(Q\). 
Let \(C^\dgn_n(Q; A)\) be the submodule of \(C^\rck_n(Q; A)\) 
generated by \(n\)-tuples \((x_1, \ldots, x_n)\) 
such that \(x_i = x_{i + 1}\) for some \(i\), 
and let \(C^\qdl_n(Q; A)\) be the quotient module 
\(C^\rck_n(Q; A) / C^\dgn_n(Q; A)\). 
Their boundary maps are obtained naturally 
from \(\partial_n\),
which make \(C^\dgn_\ast(Q; A)\) and \(C^\qdl_\ast(Q; A)\) 
chain complexes. 
Thus we have homology theories 
\(H^{\mathrm{W}}_\ast(Q; A) 
= Z^{\mathrm{W}}_\ast(Q; A) / B^{\mathrm{W}}_\ast(Q; A)\), 
where \(\mathrm{W}\) denotes \(\dgn\) or \(\qdl\). 
They are called the \textbf{degeneracy homology} 
for \(\mathrm{W} = \dgn\) 
and the \textbf{quandle homology} for \(\mathrm{W} = \qdl\). 

In the case of cohomology theories, 
we first define \(C_\qdl^\ast(Q; A)\). 
The submodule \(C_\qdl^n(Q; A)\) of \(C_\rck^n(Q; A)\) 
consists of all rack \(n\)-cochains whose values on 
\(C^\dgn_n(Q; A)\) are constantly zero, 
and \(C_\dgn^n(Q; A)\) is the quotient module 
\(C_\rck^n(Q; A) / C_\qdl^n(Q; A)\). 
Their coboundary maps are also derived from \(\delta^n\), 
and the cochain complex \(C_{\mathrm{W}}^\ast(Q; A)\) 
induces a cohomology theory \(H_{\mathrm{W}}^\ast(Q; A)\), 
which we call the \textbf{quandle cohomology} 
and the \textbf{degeneracy cohomology}, 
respectively, for \(\mathrm{W} = \qdl\) 
and for \(\mathrm{W} = \dgn\). 

Similarly as in the rack case,
we use terms such as quandle (co)chains, 
quandle (co)cycles, and so on. 
For simplicity,
we will omit notations of the coefficient rings 
of (co)homology groups,
when we are concerned with the integral coefficient case.
\\

Directly from their definitions, 
there exist short exact sequences of complexes
\begin{quote}
\(0 \to C^\dgn_\ast(Q; A) 
\stackrel{\iota}{\to} C^\rck_\ast(Q; A) 
\stackrel{\rho}{\to} C^\qdl_\ast(Q; A) \to 0\), 
\end{quote}
and 
\begin{quote}
\(0 \to C_\qdl^\ast(Q; A) 
\stackrel{\iota}{\to} C_\rck^\ast(Q; A) 
\stackrel{\rho}{\to} C_\dgn^\ast(Q; A) \to 0\). 
\end{quote}
Hence we have long exact sequences 
of (co)ho\-mol\-o\-gy theories 
\begin{quote}
\(\cdots \to H^\dgn_\ast(Q; A) 
\stackrel{\iota_\ast}{\to} H^\rck_\ast(Q; A) 
\stackrel{\rho_\ast}{\to} H^\qdl_\ast(Q; A) 
\stackrel{\Delta_\ast}{\to} H^\dgn_{\ast - 1}(Q; A) \to \cdots\) 
\end{quote}
and 
\begin{quote}
\(\cdots \to H_\qdl^\ast(Q; A) 
\stackrel{\iota^\ast}{\to} H_\rck^\ast(Q; A) 
\stackrel{\rho^\ast}{\to} H_\dgn^\ast(Q; A) 
\stackrel{\Delta^\ast}{\to} H_\qdl^{\ast + 1}(Q; A) \to \cdots\), 
\end{quote}
where \(\Delta_\ast\) and \(\Delta^\ast\) denote 
the connecting homomorphisms. 

In [CJKS1], it is conjectured that 
all the connecting homomorphisms are zero maps. 
This conjecture is affirmatively proved by Litherland-Nelson: 

\begin{theorem}[Litherland-Nelson {[LN]}]
\label{TH:homology splitting}
For an arbitrary quandle \(Q\), 
all the connecting homomorphisms \(\Delta_\ast\) 
are zero maps. 
Moreover, the resulting short exact sequence 
\(0 \to H^\dgn_n(Q; A) 
\to H^\rck_n(Q; A) \to H^\qdl_n(Q; A) \to 0\) 
is splittable, that is, there holds
\begin{quote}
\(H^\rck_n(Q; A) \cong H^\qdl_n(Q; A) \oplus H^\dgn_n(Q; A)\). 
\end{quote}
\end{theorem}

In the proof of this theorem, 
Litherland-Nelson constructed a splitting homomorphism 
of the short exact sequence of chain complexes. 
We will see the detail of this homomorphism 
in \S \ref{SEC:shifting splitting},
which plays an important role 
in \S \ref{SEC:construction shadow class}.

Notice that, by the duality, we have a corollary as to 
cohomology theories: 

\begin{corollary}\label{COR:cohomology splitting}
All the connecting homomorphisms \(\Delta^\ast\) 
in the long exact sequence of the cohomology theories 
of a quandle \(Q\) are zero maps. 
Moreover, the rack homology group \(H_\rck^n(Q; A)\) 
splits as follows:
\begin{quote}
\(H_\rck^n(Q; A) \cong H_\qdl^n(Q; A) \oplus H_\dgn^n(Q; A)\). 
\end{quote}
\end{corollary}

\subsection{Diagrams.}
\label{SEC:diagram}

Carter-Kamada-Saito [CKS1] defined generalised knot diagrams 
for representing rack homology classes. 
Their definition of generalised knot diagrams 
works for arbitrary dimensions, 
but in this paper we only need those 
of dimension at most three. 
Furthermore, the definition in [CKS1] 
contains exceptional diagrams 
such as endpoints, branch points and hemmed crossings, 
which are not necessary for us. 
So our definition below is a simplified one. 

To avoid complication, we only say ``diagrams'' 
instead of generalised knot diagrams. 
A diagram is a pair \((M, D)\) of an oriented compact 
manifold \(M\) and its subspace \(D\)
which can be decomposed into unit diagrams. 
To begin with, 
we define \textbf{unit diagrams} of each dimension. \\

The unit \(0\)-diagram is a pair \(E^0_0 
= (\{0\}, \emptyset)\) with signature 
\(+ 1\) or \(- 1\). 
This diagram is of use only when we consider 
the boundaries of \(1\)-diagrams. \\

Let \(I\) be the interval \([- 1, 1]\) 
and suppose that \(I\) is oriented as usual. 
We have two types of the unit \(1\)-diagrams, 
that is, pairs \(E^1_1 = (I, \{0\})\) 
with signature \(+ 1\) or \(- 1\) 
and \(E^1_0 = (I, \emptyset)\) with signature \(0\). 
The diagrams \(E^1_1\) and \(E^1_0\) are 
drawn in Figure \ref{FIG:unit 1-diagram}, 
\begin{figure}[ht]
\begin{center}
\unitlength 0.1in%
\begin{picture}(36.0400, 5.5000)( 1.9200, -7.0000)%
%
%
\special{pn 8}%
\special{pa 200 400}%
\special{pa 940 400}%
\special{dt 0.045}%
\special{sh 1}%
\special{pn 4}%
\special{pa 1000 400}%
\special{pa 930 370}%
\special{pa 940 400}%
\special{pa 930 430}%
\special{pa 1000 400}%
\special{fp}%
%
\special{pn 8}%
\special{sh 1}%
\special{ar 600 400 50 50  0.0000000 6.2831853}%
%
\special{pn 20}%
\special{pa 600 400}%
\special{pa 720 400}%
\special{fp}%
\special{sh 1}%
\special{pn 4}%
\special{pa 780 400}%
\special{pa 710 370}%
\special{pa 720 400}%
\special{pa 710 430}%
\special{pa 780 400}%
\special{fp}%
%
\put(6.0000,-6.0000){\makebox(0,0){\(+ E^1_1\)}}%
%
%
\special{pn 8}%
\special{pa 1600 400}%
\special{pa 2340 400}%
\special{dt 0.045}%
\special{sh 1}%
\special{pn 4}%
\special{pa 2400 400}%
\special{pa 2330 370}%
\special{pa 2340 400}%
\special{pa 2330 430}%
\special{pa 2400 400}%
\special{fp}%
%
\special{pn 8}%
\special{sh 1}%
\special{ar 2000 400 50 50  0.0000000 6.2831853}%
%
\special{pn 20}%
\special{pa 2000 400}%
\special{pa 1880 400}%
\special{fp}%
\special{sh 1}%
\special{pn 4}%
\special{pa 1820 400}%
\special{pa 1890 370}%
\special{pa 1880 400}%
\special{pa 1890 430}%
\special{pa 1820 400}%
\special{fp}%
%
\put(20.0000,-6.0000){\makebox(0,0){\(- E^1_1\)}}%
%
%
\special{pn 8}%
\special{pa 3000 400}%
\special{pa 3740 400}%
\special{dt 0.045}%
\special{sh 1}%
\special{pn 4}%
\special{pa 3800 400}%
\special{pa 3730 370}%
\special{pa 3740 400}%
\special{pa 3730 430}%
\special{pa 3800 400}%
\special{fp}%
%
\put(34.0000,-6.0000){\makebox(0,0){\(E^1_0\)}}%
%
%
\end{picture}%
\end{center}
\caption{Unit \(1\)-diagrams.}
\label{FIG:unit 1-diagram}
\end{figure}
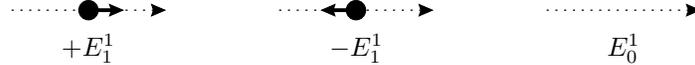
where the first one is \(E^1_1\) with signature \(+ 1\), 
the second is \(E^1_1\) with \(- 1\) 
and the third is \(E^1_0\). 
The signatures of \(E^1_1\)'s are depicted 
with the arrows at the origin. \\

Let \(S\) be a square \(I^2\), 
and let \(l_1\) and \(l_2\) be lines \(I \times \{0\}\) 
and \(\{0\} \times I\) in \(S\), respectively. 
The orientation of \(l_1\) is fixed 
as in Figure \ref{FIG:line orientation}, 
but that of \(l_2\) is not.
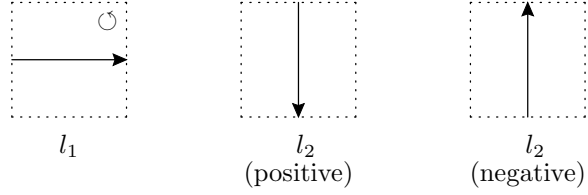
\begin{figure}[ht]
\begin{center}
\unitlength 0.1in%
\begin{picture}( 34.0800,  10.0000)( 1.9200, -11.5)%
%
\special{pn 8}%
\special{pa 200 200}%
\special{pa 200 800}%
\special{dt 0.045}%
\special{pa 200 800}%
\special{pa 800 800}%
\special{dt 0.045}%
\special{pa 800 800}%
\special{pa 800 200}%
\special{dt 0.045}%
\special{pa 800 200}%
\special{pa 200 200}%
\special{dt 0.045}%
%
\special{pn 8}%
\special{pa 200 500}%
\special{pa 740 500}%
\special{fp}%
\special{sh 1}%
\special{pn 4}%
\special{pa 800 500}%
\special{pa 730 470}%
\special{pa 740 500}%
\special{pa 730 530}%
\special{pa 800 500}%
\special{fp}%
%
\put(7.0000,-3.0000){\makebox(0,0){\(\circlearrowleft\)}}%
%
\put(5.0000,-9.5){\makebox(0,0){\(l_1\)}}%
%
%
\special{pn 8}%
\special{pa 1400 200}%
\special{pa 1400 800}%
\special{dt 0.045}%
\special{pa 1400 800}%
\special{pa 2000 800}%
\special{dt 0.045}%
\special{pa 2000 800}%
\special{pa 2000 200}%
\special{dt 0.045}%
\special{pa 2000 200}%
\special{pa 1400 200}%
\special{dt 0.045}%
%
\special{pn 8}%
\special{pa 1700 200}%
\special{pa 1700 740}%
\special{fp}%
\special{sh 1}%
\special{pn 4}%
\special{pa 1700 800}%
\special{pa 1730 730}%
\special{pa 1700 740}%
\special{pa 1670 730}%
\special{pa 1700 800}%
\special{fp}%
%
\put(17.4,-9.5){\makebox(0,0){\(l_2\)}}%
\put(17,-11){\makebox(0,0){(positive)}}%
%
%
\special{pn 8}%
\special{pa 2600 200}%
\special{pa 2600 800}%
\special{dt 0.045}%
\special{pa 2600 800}%
\special{pa 3200 800}%
\special{dt 0.045}%
\special{pa 3200 800}%
\special{pa 3200 200}%
\special{dt 0.045}%
\special{pa 3200 200}%
\special{pa 2600 200}%
\special{dt 0.045}%
%
\special{pn 8}%
\special{pa 2900 800}%
\special{pa 2900 260}%
\special{fp}%
\special{sh 1}%
\special{pn 4}%
\special{pa 2900 200}%
\special{pa 2870 270}%
\special{pa 2900 260}%
\special{pa 2930 270}%
\special{pa 2900 200}%
\special{fp}%
%
\put(29.4,-9.5){\makebox(0,0){\(l_2\)}}%
\put(29,-11){\makebox(0,0){(negative)}}%
%
%
\end{picture}%
\end{center}
\caption{Lines and their orientations.}
\label{FIG:line orientation}
\end{figure}
The unit \(2\)-diagrams are of three types; 
pairs \(E^2_2 = (S, l_1 \cup l_2)\), 
\(E^2_1 = (S, l_1)\) 
and \(E^2_0 = (S, \emptyset)\). 
The diagram \(E^2_2\) has its signature 
\(+ 1\) or \(- 1\) as \(E^1_1\), but the signatures 
of \(E^2_1\) and \(E^2_0\) are \(0\). 
The left two diagrams in Figure \ref{FIG:unit 2-diagram} shows 
\(E^2_2\) with signature \(+ 1\) and with \(- 1\). 
\begin{figure}[ht]
\begin{center}
\unitlength 0.1in%
\begin{picture}( 42.0800,  9.5000)(  1.9200, -10.8000)%
%
%
\special{pn 8}%
\special{pa 200 200}%
\special{pa 200 800}%
\special{dt 0.045}%
\special{pa 200 800}%
\special{pa 800 800}%
\special{dt 0.045}%
\special{pa 800 800}%
\special{pa 800 200}%
\special{dt 0.045}%
\special{pa 800 200}%
\special{pa 200 200}%
\special{dt 0.045}%
%
\special{pn 8}%
\special{pa 200 500}%
\special{pa 740 500}%
\special{fp}%
\special{pn 4}%
\special{sh 1}%
\special{pa 800 500}%
\special{pa 730 470}%
\special{pa 740 500}%
\special{pa 730 530}%
\special{pa 800 500}%
\special{fp}%
%
\special{pn 8}%
\special{pa 500 800}%
\special{pa 500 550}%
\special{fp}%
\special{pn 8}%
\special{pa 500 450}%
\special{pa 500 260}%
\special{fp}%
\special{pn 4}%
\special{sh 1}%
\special{pa 500 200}%
\special{pa 470 270}%
\special{pa 500 260}%
\special{pa 530 270}%
\special{pa 500 200}%
\special{fp}%
%
\put(5.0000,-10.0000){\makebox(0,0){\(+ E^2_2\)}}%
%
%
\special{pn 8}%
\special{pa 1400 200}%
\special{pa 1400 800}%
\special{dt 0.045}%
\special{pa 1400 800}%
\special{pa 2000 800}%
\special{dt 0.045}%
\special{pa 2000 800}%
\special{pa 2000 200}%
\special{dt 0.045}%
\special{pa 2000 200}%
\special{pa 1400 200}%
\special{dt 0.045}%
%
\special{pn 8}%
\special{pa 1400 500}%
\special{pa 1940 500}%
\special{fp}%
\special{pn 4}%
\special{sh 1}%
\special{pa 2000 500}%
\special{pa 1930 470}%
\special{pa 1940 500}%
\special{pa 1930 530}%
\special{pa 2000 500}%
\special{fp}%
%
\special{pn 8}%
\special{pa 1700 200}%
\special{pa 1700 450}%
\special{fp}%
\special{pn 8}%
\special{pa 1700 550}%
\special{pa 1700 740}%
\special{fp}%
\special{pn 4}%
\special{sh 1}%
\special{pa 1700 800}%
\special{pa 1730 730}%
\special{pa 1700 740}%
\special{pa 1670 730}%
\special{pa 1700 800}%
\special{fp}%
%
\put(17.0000,-10.0000){\makebox(0,0){\(- E^2_2\)}}%
%
%
\special{pn 8}%
\special{pa 2600 200}%
\special{pa 2600 800}%
\special{dt 0.045}%
\special{pa 2600 800}%
\special{pa 3200 800}%
\special{dt 0.045}%
\special{pa 3200 800}%
\special{pa 3200 200}%
\special{dt 0.045}%
\special{pa 3200 200}%
\special{pa 2600 200}%
\special{dt 0.045}%
%
\special{pn 8}%
\special{pa 2600 500}%
\special{pa 3140 500}%
\special{fp}%
\special{pn 4}%
\special{sh 1}%
\special{pa 3200 500}%
\special{pa 3130 470}%
\special{pa 3140 500}%
\special{pa 3130 530}%
\special{pa 3200 500}%
\special{fp}%
%
\put(29.0500,-10.0000){\makebox(0,0){\(E^2_1\)}}%
%
%
\special{pn 8}%
\special{pa 3800 200}%
\special{pa 3800 800}%
\special{dt 0.045}%
\special{pa 3800 800}%
\special{pa 4400 800}%
\special{dt 0.045}%
\special{pa 4400 800}%
\special{pa 4400 200}%
\special{dt 0.045}%
\special{pa 4400 200}%
\special{pa 3800 200}%
\special{dt 0.045}%
%
\put(41.0500,-10.0000){\makebox(0,0){\(E^2_0\)}}%
%
%
\end{picture}%
\end{center}
\caption{Unit \(2\)-diagram.}
\label{FIG:unit 2-diagram}
\end{figure}
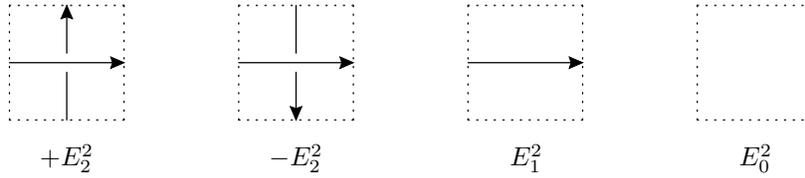
The orientation of \(l_2\) in \(E^2_2\) 
depends on the signature of the diagram. 
We call \(l_1\) and \(l_2\) the higher and the lower lines, 
respectively, 
and cut the lower line near the origin 
to distinguish them. 
The unit \(2\)-diagrams \(E^2_1\) and \(E^2_0\) are, 
respectively, the right two diagrams 
in Figure \ref{FIG:unit 2-diagram}. \\

Let \(B\) be a cube \(I^3\). 
Denote, by \(s_1\), \(s_2\) and \(s_3\), 
sheets \(I^2 \times \{0\}\), 
\(I \times \{0\} \times I\) and \(\{0\} \times I^2\) in \(B\), 
respectively. 
\begin{figure}[ht]
\begin{center}
\unitlength 0.1in%
\begin{picture}( 43.8000,  13.0000)(  1.9200,-14.2000)%
%
%
\special{pn 8}%
\special{pa 200 400}%
\special{pa 200 1000}%
\special{dt 0.045}%
\special{pn 8}%
\special{pa 200 1000}%
\special{pa 800 1000}%
\special{dt 0.045}%
\special{pn 8}%
\special{pa 800 1000}%
\special{pa 800 400}%
\special{dt 0.045}%
\special{pn 8}%
\special{pa 800 400}%
\special{pa 200 400}%
\special{dt 0.045}%
\special{pn 8}%
\special{pa 380 220}%
\special{pa 380 520}%
\special{dt 0.045}%
\special{pn 8}%
\special{pa 380 700}%
\special{pa 380 820}%
\special{dt 0.045}%
\special{pn 8}%
\special{pa 380 820}%
\special{pa 980 820}%
\special{dt 0.045}%
\special{pn 8}%
\special{pa 980 820}%
\special{pa 980 220}%
\special{dt 0.045}%
\special{pn 8}%
\special{pa 980 220}%
\special{pa 380 220}%
\special{dt 0.045}%
\special{pn 8}%
\special{pa 200 400}%
\special{pa 380 220}%
\special{dt 0.045}%
\special{pn 8}%
\special{pa 200 1000}%
\special{pa 380 820}%
\special{dt 0.045}%
\special{pn 8}%
\special{pa 800 1000}%
\special{pa 980 820}%
\special{dt 0.045}%
\special{pn 8}%
\special{pa 800 400}%
\special{pa 980 220}%
\special{dt 0.045}%
%
\special{pn 8}%
\special{pa 200 700}%
\special{pa 740 700}%
\special{fp}%
\special{pn 4}%
\special{sh 1}%
\special{pa 800 700}%
\special{pa 730 670}%
\special{pa 740 700}%
\special{pa 730 730}%
\special{pa 800 700}%
\special{fp}%
\special{pn 8}%
\special{pa 800 700}%
\special{pa 980 520}%
\special{fp}%
\special{pn 4}%
\special{sh 1}%
\special{pa 980 520}%
\special{pa 945 525}%
\special{pa 950 550}%
\special{pa 945 585}%
\special{pa 980 520}%
\special{fp}%
\special{pn 8}%
\special{pa 980 520}%
\special{pa 440 520}%
\special{fp}%
\special{pn 4}%
\special{sh 1}%
\special{pa 380 520}%
\special{pa 450 550}%
\special{pa 440 520}%
\special{pa 450 490}%
\special{pa 380 520}%
\special{fp}%
\special{pn 8}%
\special{pa 380 520}%
\special{pa 200 700}%
\special{fp}%
\special{pn 4}%
\special{sh 1}%
\special{pa 200 700}%
\special{pa 235 695}%
\special{pa 230 670}%
\special{pa 235 635}%
\special{pa 200 700}%
\special{fp}%
%
\put(6.0000,-12.0000){\makebox(0,0){\(s_1\)}}%
%
%
\special{pn 8}%
\special{pa 1400 400}%
\special{pa 1400 1000}%
\special{dt 0.045}%
\special{pn 8}%
\special{pa 1400 1000}%
\special{pa 2000 1000}%
\special{dt 0.045}%
\special{pn 8}%
\special{pa 2000 1000}%
\special{pa 2000 400}%
\special{dt 0.045}%
\special{pn 8}%
\special{pa 2000 400}%
\special{pa 1400 400}%
\special{dt 0.045}%
\special{pn 8}%
\special{pa 1580 220}%
\special{pa 1580 310}%
\special{dt 0.045}%
\special{pn 8}%
\special{pa 2090 820}%
\special{pa 2180 820}%
\special{dt 0.045}%
\special{pn 8}%
\special{pa 2180 820}%
\special{pa 2180 220}%
\special{dt 0.045}%
\special{pn 8}%
\special{pa 2180 220}%
\special{pa 1580 220}%
\special{dt 0.045}%
\special{pn 8}%
\special{pa 1400 400}%
\special{pa 1580 220}%
\special{dt 0.045}%
\special{pn 8}%
\special{pa 1400 1000}%
\special{pa 1490 910}%
\special{dt 0.045}%
\special{pn 8}%
\special{pa 2000 1000}%
\special{pa 2180 820}%
\special{dt 0.045}%
\special{pn 8}%
\special{pa 2000 400}%
\special{pa 2180 220}%
\special{dt 0.045}%
%
\special{pn 8}%
\special{pa 1490 310}%
\special{pa 2030 310}%
\special{fp}%
\special{pn 4}%
\special{sh 1}%
\special{pa 2090 310}%
\special{pa 2040 290}%
\special{pa 2030 310}%
\special{pa 2000 330}%
\special{pa 2090 310}%
\special{fp}%
\special{pn 8}%
\special{pa 2090 310}%
\special{pa 2090 910}%
\special{fp}%
\special{pn 4}%
\special{sh 1}%
\special{pa 2090 910}%
\special{pa 2110 820}%
\special{pa 2090 850}%
\special{pa 2070 860}%
\special{pa 2090 910}%
\special{fp}%
\special{pn 8}%
\special{pa 2090 910}%
\special{pa 1490 910}%
\special{fp}%
\special{pn 4}%
\special{sh 1}%
\special{pa 1490 910}%
\special{pa 1540 930}%
\special{pa 1550 910}%
\special{pa 1580 890}%
\special{pa 1490 910}%
\special{fp}%
\special{pn 8}%
\special{pa 1490 910}%
\special{pa 1490 310}%
\special{fp}%
\special{pn 4}%
\special{sh 1}%
\special{pa 1490 310}%
\special{pa 1470 400}%
\special{pa 1490 370}%
\special{pa 1510 360}%
\special{pa 1490 310}%
\special{fp}%
%
\put(18.0000,-12.0000){\makebox(0,0){\(s_2\)}}%
%
%
\special{pn 8}%
\special{pa 2600 400}%
\special{pa 2600 1000}%
\special{dt 0.045}%
\special{pn 8}%
\special{pa 2600 1000}%
\special{pa 3200 1000}%
\special{dt 0.045}%
\special{pn 8}%
\special{pa 3200 1000}%
\special{pa 3200 400}%
\special{dt 0.045}%
\special{pn 8}%
\special{pa 3200 400}%
\special{pa 2600 400}%
\special{dt 0.045}%
\special{pn 8}%
\special{pa 2780 220}%
\special{pa 2780 820}%
\special{dt 0.045}%
\special{pn 8}%
\special{pa 2780 820}%
\special{pa 2900 820}%
\special{dt 0.045}%
\special{pn 8}%
\special{pa 3080 820}%
\special{pa 3380 820}%
\special{dt 0.045}%
\special{pn 8}%
\special{pa 3380 820}%
\special{pa 3380 220}%
\special{dt 0.045}%
\special{pn 8}%
\special{pa 3380 220}%
\special{pa 2780 220}%
\special{dt 0.045}%
\special{pn 8}%
\special{pa 2600 400}%
\special{pa 2780 220}%
\special{dt 0.045}%
\special{pn 8}%
\special{pa 2600 1000}%
\special{pa 2780 820}%
\special{dt 0.045}%
\special{pn 8}%
\special{pa 3200 1000}%
\special{pa 3380 820}%
\special{dt 0.045}%
\special{pn 8}%
\special{pa 3200 400}%
\special{pa 3380 220}%
\special{dt 0.045}%
%
\special{pn 8}%
\special{pa 3080 220}%
\special{pa 3080 760}%
\special{fp}%
\special{pn 4}%
\special{sh 1}%
\special{pa 3080 820}%
\special{pa 3110 750}%
\special{pa 3080 760}%
\special{pa 3050 750}%
\special{pa 3080 820}%
\special{fp}%
\special{pn 8}%
\special{pa 3080 820}%
\special{pa 2900 1000}%
\special{fp}%
\special{pn 4}%
\special{sh 1}%
\special{pa 2900 1000}%
\special{pa 2905 965}%
\special{pa 2930 970}%
\special{pa 2965 965}%
\special{pa 2900 1000}%
\special{fp}%
\special{pn 8}%
\special{pa 2900 1000}%
\special{pa 2900 460}%
\special{fp}%
\special{pn 4}%
\special{sh 1}%
\special{pa 2900 400}%
\special{pa 2870 470}%
\special{pa 2900 460}%
\special{pa 2930 470}%
\special{pa 2900 400}%
\special{fp}%
\special{pn 8}%
\special{pa 2900 400}%
\special{pa 3050 250}%
\special{fp}%
\special{pn 4}%
\special{sh 1}%
\special{pa 3080 220}%
\special{pa 3075 255}%
\special{pa 3050 250}%
\special{pa 3015 255}%
\special{pa 3080 220}%
\special{fp}%
%
\put(30.0000,-12.0000){\makebox(0,0){\(s_3\)}}%
\put(30.0000,-13.5000){\makebox(0,0){(positive)}}%
%
%
\special{pn 8}%
\special{pa 3800 400}%
\special{pa 3800 1000}%
\special{dt 0.045}%
\special{pn 8}%
\special{pa 3800 1000}%
\special{pa 4400 1000}%
\special{dt 0.045}%
\special{pn 8}%
\special{pa 4400 1000}%
\special{pa 4400 400}%
\special{dt 0.045}%
\special{pn 8}%
\special{pa 4400 400}%
\special{pa 3800 400}%
\special{dt 0.045}%
\special{pn 8}%
\special{pa 3980 220}%
\special{pa 3980 820}%
\special{dt 0.045}%
\special{pn 8}%
\special{pa 3980 820}%
\special{pa 4100 820}%
\special{dt 0.045}%
\special{pn 8}%
\special{pa 4280 820}%
\special{pa 4580 820}%
\special{dt 0.045}%
\special{pn 8}%
\special{pa 4580 820}%
\special{pa 4580 220}%
\special{dt 0.045}%
\special{pn 8}%
\special{pa 4580 220}%
\special{pa 3980 220}%
\special{dt 0.045}%
\special{pn 8}%
\special{pa 3800 400}%
\special{pa 3980 220}%
\special{dt 0.045}%
\special{pn 8}%
\special{pa 3800 1000}%
\special{pa 3980 820}%
\special{dt 0.045}%
\special{pn 8}%
\special{pa 4400 1000}%
\special{pa 4580 820}%
\special{dt 0.045}%
\special{pn 8}%
\special{pa 4400 400}%
\special{pa 4580 220}%
\special{dt 0.045}%
%
\special{pn 8}%
\special{pa 4280 220}%
\special{pa 4130 370}%
\special{fp}%
\special{pn 4}%
\special{sh 1}%
\special{pa 4100 400}%
\special{pa 4105 365}%
\special{pa 4130 370}%
\special{pa 4165 365}%
\special{pa 4100 400}%
\special{fp}%
\special{pn 8}%
\special{pa 4100 400}%
\special{pa 4100 940}%
\special{fp}%
\special{pn 4}%
\special{sh 1}%
\special{pa 4100 1000}%
\special{pa 4130 930}%
\special{pa 4100 940}%
\special{pa 4070 930}%
\special{pa 4100 1000}%
\special{fp}%
\special{pn 8}%
\special{pa 4100 1000}%
\special{pa 4250 850}%
\special{fp}%
\special{pn 4}%
\special{sh 1}%
\special{pa 4280 820}%
\special{pa 4215 855}%
\special{pa 4250 850}%
\special{pa 4275 855}%
\special{pa 4280 820}%
\special{fp}%
\special{pn 8}%
\special{pa 4280 820}%
\special{pa 4280 280}%
\special{fp}%
\special{pn 4}%
\special{sh 1}%
\special{pa 4280 220}%
\special{pa 4250 290}%
\special{pa 4280 280}%
\special{pa 4310 290}%
\special{pa 4280 220}%
\special{fp}%
%
\put(42.0000,-12.0000){\makebox(0,0){\(s_3\)}}%
\put(42.0000,-13.5000){\makebox(0,0){(negative)}}%
%
%
\end{picture}%
\end{center}
\caption{Sheets.}
\label{FIG:sheet orientation}
\end{figure}
The orientations of sheets \(s_1\) and \(s_2\) 
are given as in Figure \ref{FIG:sheet orientation}, 
but that of \(s_3\) is variable. 
There are four types of the unit \(3\)-diagrams, 
denoted by \(E^3_3\), \(E^3_2\), \(E^3_1\) and \(E^3_0\). 
They are defined, respectively, by 
\begin{quote}
\(E^3_3 = (B, s_1 \cup s_2 \cup s_3)\), 
\(E^3_2 = (B, s_1 \cup s_2)\), \\
\(E^3_1 = (B, s_1)\) and 
\(E^3_0 = (B, \emptyset)\). 
\end{quote}
In this case, the signatures of \(E^3_2\), \(E^3_1\) 
and \(E^3_0\) are all \(0\). 
\begin{figure}[ht]
\begin{center}
\input{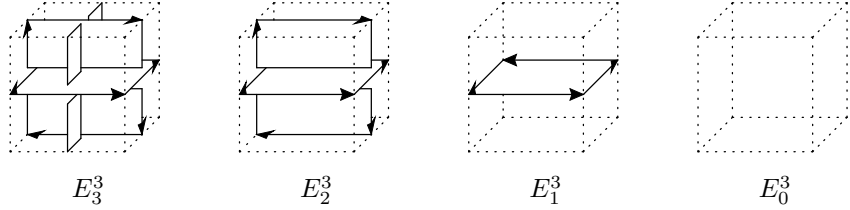}
\end{center}
\caption{Unit \(3\)-diagrams.}
\label{FIG:unit 3-diagram}
\end{figure}
Only \(E^3_3\) has its signature \(+ 1\) or \(- 1\), 
which determines the orientation of \(s_3\) 
as in Figure \ref{FIG:sheet orientation}. 
Each diagram is drawn in Figure \ref{FIG:unit 3-diagram}, 
where the signature of \(E^3_3\) is disregarded. 
We call \(s_1\), \(s_2\) and \(s_3\) in \(E^3_3\), 
respectively, the highest, the middle and the lowest sheets, 
and also call \(s_1\) and \(s_2\) in \(E^3_2\), 
respectively, the higher and the lower ones. 
The lower two sheets \(s_2\) and \(s_3\) are drawn separated 
as in Figure \ref{FIG:unit 3-diagram} 
to distinguish one sheet from the others. 

For a unit diagram \(E = E^n_k\), 
we will call a pair \((n, k)\) its type.
The points on the intersection of two lines or 
of just two sheets are called double points. 
The origin of \(E^3_3\) is said to be a triple point. 
We say multiple points for both double and triple points. \\

For \(n \ne 0\), 
the boundary of a unit \(n\)-diagram 
consists of \(2n\) faces in the usual sense. 
Obviously, each face is one of unit \((n - 1)\)-diagrams. 
For example, the boundary of \(E^3_2\) 
consists of six faces as depicted in Figure \ref{FIG:boundary}. 
Two of them are \(+ E^2_2\) and \(- E^2_2\) 
and the rest are \(E^2_1\)'s. 
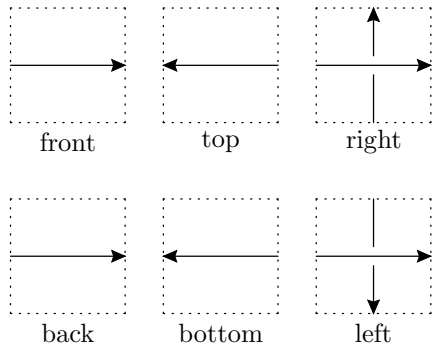
\begin{figure}[ht]
\begin{center}
\unitlength 0.1in%
\begin{picture}( 22.0800, 17.7000)(  1.9200,-19.4000)%
%
%
\special{pn 8}%
\special{pa 200 200}%
\special{pa 200 800}%
\special{dt 0.045}%
\special{pa 200 800}%
\special{pa 800 800}%
\special{dt 0.045}%
\special{pa 800 800}%
\special{pa 800 200}%
\special{dt 0.045}%
\special{pa 800 200}%
\special{pa 200 200}%
\special{dt 0.045}%
%
\special{pn 8}%
\special{pa 200 500}%
\special{pa 800 500}%
\special{fp}%
\special{pn 4}%
\special{sh 1}%
\special{pa 800 500}%
\special{pa 730 470}%
\special{pa 740 500}%
\special{pa 730 530}%
\special{pa 800 500}%
\special{fp}%
%
\put(5.0000,-9.0000){\makebox(0,0){front}}%
%
%
\special{pn 8}%
\special{pa 200 1200}%
\special{pa 200 1800}%
\special{dt 0.045}%
\special{pa 200 1800}%
\special{pa 800 1800}%
\special{dt 0.045}%
\special{pa 800 1800}%
\special{pa 800 1200}%
\special{dt 0.045}%
\special{pa 800 1200}%
\special{pa 200 1200}%
\special{dt 0.045}%
%
\special{pn 8}%
\special{pa 200 1500}%
\special{pa 800 1500}%
\special{fp}%
\special{pn 4}%
\special{sh 1}%
\special{pa 800 1500}%
\special{pa 730 1470}%
\special{pa 740 1500}%
\special{pa 730 1530}%
\special{pa 800 1500}%
\special{fp}%
%
\put(5.0000,-19.0000){\makebox(0,0){back}}%
%
%
\special{pn 8}%
\special{pa 1000 200}%
\special{pa 1000 800}%
\special{dt 0.045}%
\special{pa 1000 800}%
\special{pa 1600 800}%
\special{dt 0.045}%
\special{pa 1600 800}%
\special{pa 1600 200}%
\special{dt 0.045}%
\special{pa 1600 200}%
\special{pa 1000 200}%
\special{dt 0.045}%
%
\special{pn 8}%
\special{pa 1600 500}%
\special{pa 1000 500}%
\special{fp}%
\special{pn 4}%
\special{sh 1}%
\special{pa 1000 500}%
\special{pa 1070 530}%
\special{pa 1060 500}%
\special{pa 1070 470}%
\special{pa 1000 500}%
\special{fp}%
%
\put(13.0000,-9.0000){\makebox(0,0){top}}%
%
%
\special{pn 8}%
\special{pa 1000 1200}%
\special{pa 1000 1800}%
\special{dt 0.045}%
\special{pa 1000 1800}%
\special{pa 1600 1800}%
\special{dt 0.045}%
\special{pa 1600 1800}%
\special{pa 1600 1200}%
\special{dt 0.045}%
\special{pa 1600 1200}%
\special{pa 1000 1200}%
\special{dt 0.045}%
%
\special{pn 8}%
\special{pa 1600 1500}%
\special{pa 1000 1500}%
\special{fp}%
\special{pn 4}%
\special{sh 1}%
\special{pa 1000 1500}%
\special{pa 1070 1530}%
\special{pa 1060 1500}%
\special{pa 1070 1470}%
\special{pa 1000 1500}%
\special{fp}%
%
\put(13.0000,-19.0000){\makebox(0,0){bottom}}%
%
%
\special{pn 8}%
\special{pa 1800 200}%
\special{pa 1800 800}%
\special{dt 0.045}%
\special{pa 1800 800}%
\special{pa 2400 800}%
\special{dt 0.045}%
\special{pa 2400 800}%
\special{pa 2400 200}%
\special{dt 0.045}%
\special{pa 2400 200}%
\special{pa 1800 200}%
\special{dt 0.045}%
%
\special{pn 8}%
\special{pa 1800 500}%
\special{pa 2400 500}%
\special{fp}%
\special{pn 4}%
\special{sh 1}%
\special{pa 2400 500}%
\special{pa 2330 470}%
\special{pa 2340 500}%
\special{pa 2330 530}%
\special{pa 2400 500}%
\special{fp}%
%
\special{pn 8}%
\special{pa 2100 800}%
\special{pa 2100 550}%
\special{fp}%
\special{pn 8}%
\special{pa 2100 450}%
\special{pa 2100 200}%
\special{fp}%
\special{pn 4}%
\special{sh 1}%
\special{pa 2100 200}%
\special{pa 2070 270}%
\special{pa 2100 260}%
\special{pa 2130 270}%
\special{pa 2100 200}%
\special{fp}%
%
\put(21.0000,-9.0000){\makebox(0,0){right}}%
%
%
\special{pn 8}%
\special{pa 1800 1200}%
\special{pa 1800 1800}%
\special{dt 0.045}%
\special{pa 1800 1800}%
\special{pa 2400 1800}%
\special{dt 0.045}%
\special{pa 2400 1800}%
\special{pa 2400 1200}%
\special{dt 0.045}%
\special{pa 2400 1200}%
\special{pa 1800 1200}%
\special{dt 0.045}%
%
\special{pn 8}%
\special{pa 1800 1500}%
\special{pa 2400 1500}%
\special{fp}%
\special{pn 4}%
\special{sh 1}%
\special{pa 2400 1500}%
\special{pa 2330 1470}%
\special{pa 2340 1500}%
\special{pa 2330 1530}%
\special{pa 2400 1500}%
\special{fp}%
%
\special{pn 8}%
\special{pa 2100 1200}%
\special{pa 2100 1450}%
\special{fp}%
\special{pn 8}%
\special{pa 2100 1550}%
\special{pa 2100 1800}%
\special{fp}%
\special{pn 4}%
\special{sh 1}%
\special{pa 2100 1800}%
\special{pa 2130 1730}%
\special{pa 2100 1740}%
\special{pa 2070 1730}%
\special{pa 2100 1800}%
\special{fp}%
%
\put(21.0000,-19.0000){\makebox(0,0){left}}%
%
%
\end{picture}%
\end{center}
\caption{Boundary faces of \(E^3_2\).}
\label{FIG:boundary}
\end{figure}
Diagrams in general are obtained by attaching 
these unit diagrams between their faces.
Let \(F_1 = (B_1, P_1)\) and \(F_2 = (B_2, P_2)\) 
be two faces of the same type with opposite signatures 
(they may belong to the same diagram).
An attaching map \(f\): \(F_1 \to F_2\) 
is an orientation reversing homeomorphism 
\(B_1 \to B_2\) such that \(f\) maps 
\(P_1\) to \(P_2\) with preserving the levels 
of each components. 
As for the orientations of components,
we assume that \(f\) preserves them when the faces 
are odd dimensional, but it reverses them otherwise.
\begin{figure}[ht]
\begin{center}
\unitlength 0.1in%
\begin{picture}( 36.1600,  6.8800)(  1.9200,-10.3200)%
%
%
\special{pn 8}%
\special{pa 200 400}%
\special{pa 200 1000}%
\special{dt 0.045}%
\special{pn 8}%
\special{pa 200 1000}%
\special{pa 800 1000}%
\special{dt 0.045}%
\special{pn 8}%
\special{pa 800 1000}%
\special{pa 800 400}%
\special{dt 0.045}%
\special{pn 8}%
\special{pa 800 400}%
\special{pa 200 400}%
\special{dt 0.045}%
%
\special{pn 8}%
\special{pa 200 700}%
\special{pa 800 700}%
\special{fp}%
\special{pn 4}%
\special{sh 1}%
\special{pa 800 700}%
\special{pa 730 670}%
\special{pa 740 700}%
\special{pa 730 730}%
\special{pa 800 700}%
\special{fp}%
%
\special{pn 8}%
\special{pa 500 1000}%
\special{pa 500  750}%
\special{fp}%
\special{pn 8}%
\special{pa 500  650}%
\special{pa 500  400}%
\special{fp}%
\special{pn 4}%
\special{sh 1}%
\special{pa 500 460}%
\special{pa 530 470}%
\special{pa 500 400}%
\special{pa 470 470}%
\special{pa 500 460}%
\special{fp}%
%
%
\special{pn 8}%
\special{pa 1400 400}%
\special{pa 1400 1000}%
\special{dt 0.045}%
\special{pn 8}%
\special{pa 1400 1000}%
\special{pa 2000 1000}%
\special{dt 0.045}%
\special{pn 8}%
\special{pa 2000 1000}%
\special{pa 2000 400}%
\special{dt 0.045}%
\special{pn 8}%
\special{pa 2000 400}%
\special{pa 1400 400}%
\special{dt 0.045}%
%
\special{pn 8}%
\special{pa 1400 700}%
\special{pa 1650 700}%
\special{fp}%
\special{pn 8}%
\special{pa 1750 700}%
\special{pa 2000 700}%
\special{fp}%
\special{pn 4}%
\special{sh 1}%
\special{pa 2000 700}%
\special{pa 1930 670}%
\special{pa 1940 700}%
\special{pa 1930 730}%
\special{pa 2000 700}%
\special{fp}%
%
\special{pn 8}%
\special{pa 1700 1000}%
\special{pa 1700  400}%
\special{fp}%
\special{pn 4}%
\special{sh 1}%
\special{pa 1700 460}%
\special{pa 1730 470}%
\special{pa 1700 400}%
\special{pa 1670 470}%
\special{pa 1700 460}%
\special{fp}%
%
%
\special{pn 8}%
\special{pa 2600 400}%
\special{pa 2600 1000}%
\special{dt 0.045}%
\special{pn 8}%
\special{pa 2600 1000}%
\special{pa 3800 1000}%
\special{dt 0.045}%
\special{pn 8}%
\special{pa 3800 1000}%
\special{pa 3800 400}%
\special{dt 0.045}%
\special{pn 8}%
\special{pa 3800 400}%
\special{pa 2600 400}%
\special{dt 0.045}%
\special{pn 8}%
\special{pa 3200 400}%
\special{pa 3200 1000}%
\special{dt 0.045}%
%
\special{pn 8}%
\special{pa 2600 700}%
\special{pa 3450 700}%
\special{fp}%
\special{pn 8}%
\special{pa 3550 700}%
\special{pa 3800 700}%
\special{fp}%
\special{pn 4}%
\special{sh 1}%
\special{pa 3800 700}%
\special{pa 3730 670}%
\special{pa 3740 700}%
\special{pa 3730 730}%
\special{pa 3800 700}%
\special{fp}%
%
\special{pn 8}%
\special{pa 2900 1000}%
\special{pa 2900  750}%
\special{fp}%
\special{pn 8}%
\special{pa 2900  650}%
\special{pa 2900  400}%
\special{fp}%
\special{pn 4}%
\special{sh 1}%
\special{pa 2900 460}%
\special{pa 2930 470}%
\special{pa 2900 400}%
\special{pa 2870 470}%
\special{pa 2900 460}%
\special{fp}%
%
\special{pn 8}%
\special{pa 3500 1000}%
\special{pa 3500  400}%
\special{fp}%
\special{pn 4}%
\special{sh 1}%
\special{pa 3500 460}%
\special{pa 3530 470}%
\special{pa 3500 400}%
\special{pa 3470 470}%
\special{pa 3500 460}%
\special{fp}%
%
%
\put(11.0000,-7.0000){\makebox(0,0){\(+\)}}%
\put(23.0000,-7.0000){\makebox(0,0){\(=\)}}%
%
%
\end{picture}%
\begin{center}
\end{center}
\input{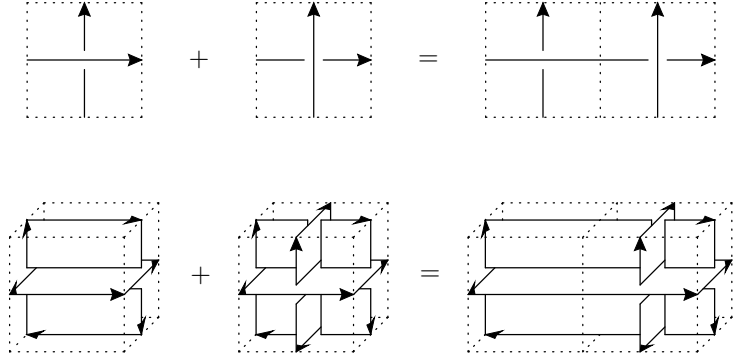}
\end{center}
\caption{Attaching two diagrams.}
\label{FIG:attachment}
\end{figure}
We show examples of attaching maps in Figure \ref{FIG:attachment}.
In the upper line, 
two \(E^2_1\)'s with opposite signatures are attached, 
where the faces are odd dimensional. 
On the other hand, an attaching map 
between \(E^3_2\) and \(E^3_3\) 
are depicted in the lower line. \\

Fix a compact manifold \(M\) with dimension \(d \le 3\). 
Let \(\{E_1, \ldots, E_n\}\) be a finite set 
of copies of unit \(d\)-diagrams, 
that is, \(E_i = (B_i, P_i)\) is one of 
\(E^d_0\), \ldots, \(E^d_d\) for each \(i\), 
and let \(f_i\) be a homeomorphism of \(B_i\) 
into \(M\) for each \(i\).
We suppose that \(M\) is covered with \(\{f_i(B_i)\}\). 
Denote, by \(F_1\), \ldots, \(F_{2d n}\), all the faces 
of \(d\)-diagrams \(E_1\), \ldots, \(E_n\), 
and denote by \(g_j\) the restriction of \(f_i\) 
over \(F_j\), where \(F_j\) is a face of \(E_i\).
If each \(F_i\) is attached to another face \(F_j\) 
by \(g_j^{- 1} \circ g_i\) or contained in \(\partial M\) 
through \(g_i\), and if the interiors of two distinct 
diagrams \(E_i\) and \(E_j\) are disjoint in \(M\),
we call the image \(D = \bigcup f_i(P_i)\) 
a \textbf{diagram} on \(M\).
We also denote a diagram by \((M, D)\).

Obviously, a pair \((M, \emptyset)\) is a diagram, 
which we call the trivial diagram on \(M\). 
For a diagram \(D\) on \(M\),
by reversing the orientation of the whole manifold \(M\) 
with keeping the direction of normal vector 
of each component of \(D\),
we obtain a new diagram, denoted by \(- D\). 
When we consider \(M\) as a based space, 
we assume that the basepoint is in the exterior 
\(M \backslash D\) of the diagram.

\subsection{Colourings and shadow colourings.}
\label{SEC:colouring}

Let \(D\) be a diagram on a manifold \(M\). 
Regarding the line \(l_2\) 
and the sheets \(s_2\) and \(s_3\) of unit diagrams 
as separated in fact, 
we can consider that \(D\) consists of 
connected components. 
Denote by \(\mathcal{C}(D)\) the set of 
all connected components of \(D\), 
and denote by \(\mathcal{R}(D)\) 
the set of all connected components 
of the exterior \(M \backslash D\). 
We call elements of \(\mathcal{C}(D)\) 
and of \(\mathcal{R}(D)\), respectively, 
\textbf{components} and \textbf{regions} of \(D\). 
For example, the unit diagram \(E^2_2\) has 
three components and four regions.

Let \(p\) be a point of \(D\). 
If \(p\) is not a multiple point, 
one component and two regions are found in its neighbourhood. 
Denote by \(c_p\) the component where \(p\) exists, 
and by \(r^\ini_p\) and \(r^\ter_p\) the regions 
so that the normal vector of \(c_p\) points 
to \(r^\ter_p\). 
These regions \(r^\ini_p\) and \(r^\ter_p\) are
called, respectively, the initial and the terminal 
region of \(p\) (or of \(c_p\)).

On the other hand, 
if \(p\) is a double point, 
there are three components of \(D\) in the neighbourhood of \(p\). 
Denote by \(o_p\) the component in the higher level,
or the over-arc. 
The other two components are in the lower level. 
We denote them by \(u^\ini_p\) and \(u^\ter_p\) 
so that the normal vector of \(o_p\) points 
to \(u^\ter_p\). 
They are also called the initial and the terminal 
under-arc of the crossing \(p\).
Additionally, for a multiple point \(p\),
there exists a unique region \(r^\ini_p\) 
adjacent to \(p\) such that \(r^\ini_p\) is 
the initial region of all components 
adjacent to both \(p\) and \(r^\ini_p\).
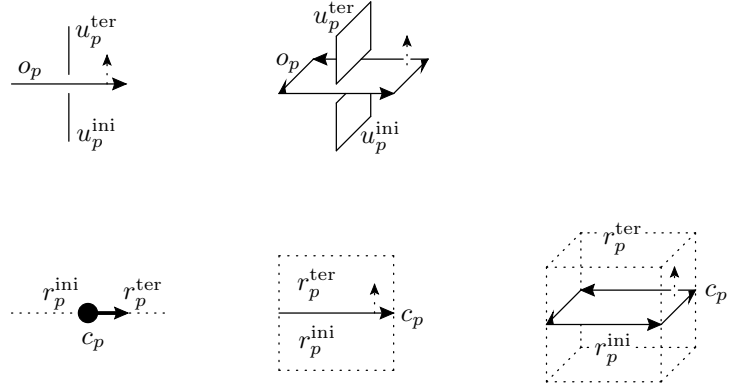
\begin{figure}[ht]
\begin{center}
\unitlength 0.1in%
\begin{picture}( 37.8000, 21.0000)( 1.9200,-23.0000)%
%
%
\special{pn 8}%
\special{pa 200 700}%
\special{pa 800 700}%
\special{fp}%
\special{pn 4}%
\special{sh 1}%
\special{pa 800 700}%
\special{pa 730 670}%
\special{pa 740 700}%
\special{pa 730 730}%
\special{pa 800 700}%
\special{fp}%
\special{pn 8}%
\special{pa 500 1000}%
\special{pa 500 750}%
\special{fp}%
\special{pn 8}%
\special{pa 500 650}%
\special{pa 500 400}%
\special{fp}%
\special{pn 8}%
\special{pa 700 700}%
\special{pa 700 550}%
\special{dt 0.045}%
\special{pn 4}%
\special{sh 1}%
\special{pa 700 550}%
\special{pa 680 600}%
\special{pa 700 590}%
\special{pa 720 600}%
\special{pa 700 550}%
\special{fp}%
\put(3.0000,-6.3000){\makebox(0,0){\(o_p\)}}%
\put(6.4500,-4.2500){\makebox(0,0){\(u^\ter_p\)}}%
\put(6.4500,-9.6500){\makebox(0,0){\(u^\ini_p\)}}%
%
%
\special{pn 8}%
\special{pa 1600 750}%
\special{pa 2200 750}%
\special{fp}%
\special{pn 4}%
\special{sh 1}%
\special{pa 2200 750}%
\special{pa 2130 720}%
\special{pa 2140 750}%
\special{pa 2130 780}%
\special{pa 2200 750}%
\special{fp}%
\special{pn 8}%
\special{pa 2200 750}%
\special{pa 2380 570}%
\special{fp}%
\special{pn 4}%
\special{sh 1}%
\special{pa 2380 570}%
\special{pa 2345 635}%
\special{pa 2350 600}%
\special{pa 2345 575}%
\special{pa 2380 570}%
\special{fp}%
\special{pn 8}%
\special{pa 2380 570}%
\special{pa 2290 570}%
\special{fp}%
\special{pn 8}%
\special{pa 2250 570}%
\special{pa 2030 570}%
\special{fp}%
\special{pn 8}%
\special{pa 1900 570}%
\special{pa 1780 570}%
\special{fp}%
\special{pn 4}%
\special{sh 1}%
\special{pa 1780 570}%
\special{pa 1850 540}%
\special{pa 1840 570}%
\special{pa 1850 600}%
\special{pa 1780 570}%
\special{fp}%
\special{pn 8}%
\special{pa 1780 570}%
\special{pa 1600 750}%
\special{fp}%
\special{pn 4}%
\special{sh 1}%
\special{pa 1600 750}%
\special{pa 1635 745}%
\special{pa 1630 720}%
\special{pa 1635 685}%
\special{pa 1600 750}%
\special{fp}%
\special{pn 8}%
\special{pa 2270 600}%
\special{pa 2270 450}%
\special{dt 0.045}%
\special{pn 4}%
\special{sh 1}%
\special{pa 2270 450}%
\special{pa 2250 500}%
\special{pa 2270 490}%
\special{pa 2290 500}%
\special{pa 2270 450}%
\special{fp}%
%
\special{pn 8}%
\special{pa 2080 270}%
\special{pa 1900 450}%
\special{fp}%
\special{pn 8}%
\special{pa 1900 450}%
\special{pa 1900 700}%
\special{fp}%
\special{pn 8}%
\special{pa 1900 700}%
\special{pa 2080 520}%
\special{fp}%
\special{pn 8}%
\special{pa 2080 520}%
\special{pa 2080 270}%
\special{fp}%
%
\special{pn 8}%
\special{pa 1950 750}%
\special{pa 1900 800}%
\special{fp}%
\special{pn 8}%
\special{pa 1900 800}%
\special{pa 1900 1050}%
\special{fp}%
\special{pn 8}%
\special{pa 1900 1050}%
\special{pa 2080 870}%
\special{fp}%
\special{pn 8}%
\special{pa 2080 870}%
\special{pa 2080 750}%
\special{fp}%
\put(16.5000,-5.9000){\makebox(0,0){\(o_p\)}}%
\put(21.3000,-9.8000){\makebox(0,0){\(u^\ini_p\)}}%
\put(18.9000,-3.5000){\makebox(0,0){\(u^\ter_p\)}}%
%
%
\special{pn 8}%
\special{pa 200 1900}%
\special{pa 1000 1900}%
\special{dt 0.045}%
%
\special{pn 8}%
\special{sh 1}%
\special{ar 600 1900 50 50  0.0000000 6.2831853}%
\special{pn 20}%
\special{pa 600 1900}%
\special{pa 790 1900}%
\special{fp}%
\special{pn 4}%
\special{sh 1}%
\special{pa 810 1900}%
\special{pa 740 1870}%
\special{pa 750 1900}%
\special{pa 740 1930}%
\special{pa 810 1900}%
\special{fp}%
\put(6.3000,-20.5000){\makebox(0,0){\(c_p\)}}%
\put(4.6000,-18.0000){\makebox(0,0){\(r^\ini_p\)}}%
\put(8.9000,-18.0000){\makebox(0,0){\(r^\ter_p\)}}%
%
%
\special{pn 8}%
\special{pa 1600 1600}%
\special{pa 1600 2200}%
\special{dt 0.045}%
\special{pn 8}%
\special{pa 1600 2200}%
\special{pa 2200 2200}%
\special{dt 0.045}%
\special{pn 8}%
\special{pa 2200 2200}%
\special{pa 2200 1600}%
\special{dt 0.045}%
\special{pn 8}%
\special{pa 2200 1600}%
\special{pa 1600 1600}%
\special{dt 0.045}%
%
\special{pn 8}%
\special{pa 1600 1900}%
\special{pa 2200 1900}%
\special{fp}%
\special{pn 4}%
\special{sh 1}%
\special{pa 2200 1900}%
\special{pa 2130 1870}%
\special{pa 2140 1900}%
\special{pa 2130 1930}%
\special{pa 2200 1900}%
\special{fp}%
\special{pn 8}%
\special{pa 2100 1900}%
\special{pa 2100 1750}%
\special{dt 0.045}%
\special{pn 4}%
\special{sh 1}%
\special{pa 2100 1750}%
\special{pa 2080 1800}%
\special{pa 2100 1790}%
\special{pa 2120 1800}%
\special{pa 2100 1750}%
\special{fp}%
\put(23.0000,-19.3000){\makebox(0,0){\(c_p\)}}%
\put(18.0000,-17.5000){\makebox(0,0){\(r^\ter_p\)}}%
\put(18.0000,-20.5000){\makebox(0,0){\(r^\ini_p\)}}%
%
%
\special{pn 8}%
\special{pa 3000 1660}%
\special{pa 3000 2260}%
\special{dt 0.045}%
\special{pn 8}%
\special{pa 3000 2260}%
\special{pa 3600 2260}%
\special{dt 0.045}%
\special{pn 8}%
\special{pa 3600 2260}%
\special{pa 3600 1660}%
\special{dt 0.045}%
\special{pn 8}%
\special{pa 3600 1660}%
\special{pa 3000 1660}%
\special{dt 0.045}%
\special{pn 8}%
\special{pa 3180 1480}%
\special{pa 3180 1780}%
\special{dt 0.045}%
\special{pn 8}%
\special{pa 3180 1960}%
\special{pa 3180 2080}%
\special{dt 0.045}%
\special{pn 8}%
\special{pa 3180 2080}%
\special{pa 3220 2080}%
\special{dt 0.045}%
\special{pn 8}%
\special{pa 3470 2080}%
\special{pa 3780 2080}%
\special{dt 0.045}%
\special{pn 8}%
\special{pa 3780 2080}%
\special{pa 3780 1480}%
\special{dt 0.045}%
\special{pn 8}%
\special{pa 3780 1480}%
\special{pa 3520 1480}%
\special{dt 0.045}%
\special{pn 8}%
\special{pa 3260 1480}%
\special{pa 3180 1480}%
\special{dt 0.045}%
\special{pn 8}%
\special{pa 3000 1660}%
\special{pa 3180 1480}%
\special{dt 0.045}%
\special{pn 8}%
\special{pa 3000 2260}%
\special{pa 3180 2080}%
\special{dt 0.045}%
\special{pn 8}%
\special{pa 3600 2260}%
\special{pa 3780 2080}%
\special{dt 0.045}%
\special{pn 8}%
\special{pa 3600 1660}%
\special{pa 3780 1480}%
\special{dt 0.045}%
%
\special{pn 8}%
\special{pa 3000 1960}%
\special{pa 3600 1960}%
\special{fp}%
\special{pn 4}%
\special{sh 1}%
\special{pa 3600 1960}%
\special{pa 3530 1930}%
\special{pa 3540 1960}%
\special{pa 3530 1990}%
\special{pa 3600 1960}%
\special{fp}%
\special{pn 8}%
\special{pa 3600 1960}%
\special{pa 3780 1780}%
\special{fp}%
\special{pn 4}%
\special{sh 1}%
\special{pa 3780 1780}%
\special{pa 3745 1845}%
\special{pa 3750 1810}%
\special{pa 3745 1785}%
\special{pa 3780 1780}%
\special{fp}%
\special{pn 8}%
\special{pa 3780 1780}%
\special{pa 3690 1780}%
\special{fp}%
\special{pn 8}%
\special{pa 3650 1780}%
\special{pa 3180 1780}%
\special{fp}%
\special{pn 4}%
\special{sh 1}%
\special{pa 3180 1780}%
\special{pa 3250 1810}%
\special{pa 3240 1780}%
\special{pa 3250 1750}%
\special{pa 3180 1780}%
\special{fp}%
\special{pn 8}%
\special{pa 3180 1780}%
\special{pa 3000 1960}%
\special{fp}%
\special{pn 4}%
\special{sh 1}%
\special{pa 3000 1960}%
\special{pa 3035 1955}%
\special{pa 3030 1930}%
\special{pa 3035 1895}%
\special{pa 3000 1960}%
\special{fp}%
\special{pn 8}%
\special{pa 3670 1810}%
\special{pa 3670 1660}%
\special{dt 0.045}%
\special{pn 4}%
\special{sh 1}%
\special{pa 3670 1660}%
\special{pa 3650 1710}%
\special{pa 3670 1700}%
\special{pa 3690 1710}%
\special{pa 3670 1660}%
\special{fp}%
\put(38.9000,-18.0000){\makebox(0,0){\(c_p\)}}%
\put(33.5000,-20.9000){\makebox(0,0){\(r^\ini_p\)}}%
\put(34.0000,-15.1000){\makebox(0,0){\(r^\ter_p\)}}%
%
%
\end{picture}%
\end{center}
\caption{Notations of components and regions.}
\label{FIG:notation}
\end{figure}
Figure \ref{FIG:notation} shows these notations. \\

Let \(R\) be a rack.
An \textbf{\(R\)-colouring} \(C\) is a map 
\(\mathcal{C}(D) \to R\) satisfying
\begin{quote}
(C) \quad \(C(u^\ini_p) \lhd C(o_p) = C(u^\ter_p)\) 
\end{quote}
for each double point \(p\).
A pair of an \(R\)-colouring \(C\) 
and a map \(C'\): \(\mathcal{R}(D) \to R\) is 
called an \textbf{\(R\)-shadow colouring} 
if it satisfies 
\begin{quote}
(SC) \quad \(C'(r^\ini_p) \lhd C(c_p) = C'(r^\ter_p)\), 
\end{quote}
where \(p \in D\) is not a multiple point. 
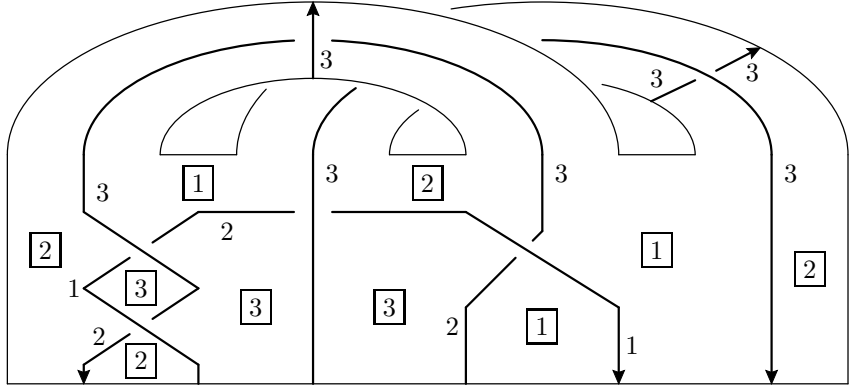
\begin{figure}[ht]
\begin{center}
\unitlength 0.1in%
\begin{picture}( 44.1000, 21.0000)(  3.92000,-26.0000)%
\special{pn 8}%
\special{pa 400 1400}%
\special{pa 400 2600}%
\special{fp}%
\special{pn 8}%
\special{pa 400 2600}%
\special{pa 4800 2600}%
\special{fp}%
\special{pn 8}%
\special{pa 4800 2600}%
\special{pa 4800 1400}%
\special{fp}%
\special{pn 8}%
\special{ar 3200 1400 1600 800  4.4095041 6.2831853}%
\special{pn 8}%
\special{ar 3200 1400 1600 800  3.1415927 3.5850410}%
\special{pn 8}%
\special{pa 1600 1400}%
\special{pa 1200 1400}%
\special{fp}%
\special{pn 8}%
\special{ar 2000 1400 800 400  3.1415927 6.2831853}%
\special{pn 8}%
\special{pa 2800 1400}%
\special{pa 2400 1400}%
\special{fp}%
\special{pn 8}%
\special{ar 3200 1400 800 400  3.1415927 3.7689008}%
\special{pn 8}%
\special{ar 3200 1400 800 400  5.1142596 6.2831853}%
\special{pn 8}%
\special{pa 4000 1400}%
\special{pa 3600 1400}%
\special{fp}%
\special{pn 8}%
\special{ar 2000 1400 1600 800  3.1415927 6.2831853}%
\special{pn 13}%
\special{pa 1400 2600}%
\special{pa 1400 2500}%
\special{fp}%
\special{pn 13}%
\special{pa 1400 2500}%
\special{pa 800 2100}%
\special{fp}%
\special{pn 13}%
\special{pa 800 2100}%
\special{pa 1040 1940}%
\special{fp}%
\special{pn 13}%
\special{pa 1160 1860}%
\special{pa 1400 1700}%
\special{fp}%
\special{pn 13}%
\special{pa 1400 1700}%
\special{pa 1900 1700}%
\special{fp}%
\special{pn 13}%
\special{pa 2100 1700}%
\special{pa 2800 1700}%
\special{fp}%
\special{pn 13}%
\special{pa 2800 1700}%
\special{pa 3600 2200}%
\special{fp}%
\special{pn 13}%
\special{pa 3600 2200}%
\special{pa 3600 2580}%
\special{fp}%
\special{pn 4}%
\special{sh 1}%
\special{pa 3600 2600}%
\special{pa 3630 2530}%
\special{pa 3600 2540}%
\special{pa 3570 2530}%
\special{pa 3600 2600}%
\special{fp}%
\special{pn 13}%
\special{pa 2000 1400}%
\special{pa 2000 2600}%
\special{fp}%
\special{pn 13}%
\special{ar 3200 1400 1200 600  3.1415927 3.7618421}%
\special{pn 13}%
\special{ar 3200 1400 1200 600  4.7123890 6.2831853}%
\special{pn 13}%
\special{pa 800 1700}%
\special{pa 800 1400}%
\special{fp}%
\special{pn 13}%
\special{pa 4400 1400}%
\special{pa 4400 2580}%
\special{fp}%
\special{pn 4}%
\special{sh 1}%
\special{pa 4400 2600}%
\special{pa 4430 2530}%
\special{pa 4400 2540}%
\special{pa 4370 2530}%
\special{pa 4400 2600}%
\special{fp}%
\special{pn 13}%
\special{pa 2800 2600}%
\special{pa 2800 2200}%
\special{fp}%
\special{pn 13}%
\special{pa 2800 2200}%
\special{pa 3060 1940}%
\special{fp}%
\special{pn 13}%
\special{pa 3150 1850}%
\special{pa 3200 1800}%
\special{fp}%
\special{pn 13}%
\special{pa 3200 1800}%
\special{pa 3200 1400}%
\special{fp}%
\special{pn 13}%
\special{ar 2000 1400 1200 600  4.7955302 6.2831853}%
\special{pn 13}%
\special{ar 2000 1400 1200 600  3.1415927 4.6292477}%
\special{pn 13}%
\special{pa 800 1700}%
\special{pa 1400 2100}%
\special{fp}%
\special{pn 13}%
\special{pa 1400 2100}%
\special{pa 1160 2260}%
\special{fp}%
\special{pn 13}%
\special{pa 1040 2340}%
\special{pa 800 2500}%
\special{fp}%
\special{pn 13}%
\special{pa 800 2500}%
\special{pa 800 2580}%
\special{fp}%
\special{pn 4}%
\special{sh 1}%
\special{pa 800 2600}%
\special{pa 830 2530}%
\special{pa 800 2540}%
\special{pa 770 2530}%
\special{pa 800 2600}%
\special{fp}%
\special{pn 13}%
\special{pa 2000 1000}%
\special{pa 2000 620}%
\special{fp}%
\special{pn 4}%
\special{sh 1}%
\special{pa 2000 600}%
\special{pa 1970 670}%
\special{pa 2000 660}%
\special{pa 2030 670}%
\special{pa 2000 600}%
\special{fp}%
\special{pn 13}%
\special{pa 3770 1120}%
\special{pa 4000 1010}%
\special{fp}%
\special{pn 13}%
\special{pa 4110 960}%
\special{pa 4320 850}%
\special{fp}%
\special{pn 4}%
\special{sh 1}%
\special{pa 4340 840}%
\special{pa 4270 844}%
\special{pa 4292 864}%
\special{pa 4295 895}%
\special{pa 4340 840}%
\special{fp}%
\put(9.0000,-16.0000){\makebox(0,0){3}}%
\put(20.7000,-9.0000){\makebox(0,0){3}}%
\put(33.0000,-15.0000){\makebox(0,0){3}}%
\put(45.0000,-15.0000){\makebox(0,0){3}}%
\put(43.0000,-9.7000){\makebox(0,0){3}}%
\put(38.0000,-10.0000){\makebox(0,0){3}}%
\put(36.7000,-24.0000){\makebox(0,0){1}}%
\put(7.5000,-21.0000){\makebox(0,0){1}}%
\put(15.5000,-18.0000){\makebox(0,0){2}}%
\put(27.3000,-23.0000){\makebox(0,0){2}}%
\put(8.8000,-23.5000){\makebox(0,0){2}}%
\put(21.0000,-15.0000){\makebox(0,0){3}}%
\put(17.0000,-22.0000){\makebox(0,0){\fbox{3}}}%
\put(11.0000,-21.0000){\makebox(0,0){\fbox{3}}}%
\put(6.0000,-19.0000){\makebox(0,0){\fbox{2}}}%
\put(14.0000,-15.5000){\makebox(0,0){\fbox{1}}}%
\put(38.0000,-19.0000){\makebox(0,0){\fbox{1}}}%
\put(32.0000,-23.0000){\makebox(0,0){\fbox{1}}}%
\put(24.0000,-22.0000){\makebox(0,0){\fbox{3}}}%
\put(46.0000,-20.0000){\makebox(0,0){\fbox{2}}}%
\put(26.0000,-15.5000){\makebox(0,0){\fbox{2}}}%
\put(11.0000,-24.8000){\makebox(0,0){\fbox{2}}}%
\end{picture}%
\end{center}
\caption{Example of (shadow) coloured diagram.}
\label{FIG:coloured diagram example}
\end{figure}
We call \(C(c)\) the \textbf{colour} of a component \(c\),
and \(C'(r)\) the \textbf{shadow colour} of a region \(r\).
When a pair of a diagram 
and its \(R\)-(shadow) colouring is given, 
we call it an \(R\)-(shadow) coloured diagram. 
An example of (shadow) coloured diagrams is drawn 
in Figure \ref{FIG:coloured diagram example}. 
The diagram is (shadow) coloured 
with the dihedral quandle \(\Z_3\). 
In Figure \ref{FIG:coloured diagram example} 
and in figures below,
the letters in boxes signify 
the shadow colours of regions. \\

Assume that homology theories are 
with integral coefficients. 
We will show the correspondence 
between \(n\)-chains of \(R\) 
and \(R\)-coloured \(n\)-diagrams 
or between \((n + 1)\)-chains 
and \(R\)-shadow coloured \(n\)-diagrams 
for \(n = 0\), \(1\), \(2\) and \(3\). 

First we consider \(R\)-(shadow) coloured unit diagrams.
In each dimension, 
unit diagrams \(E^n_k\) correspond to \(0\) 
if \(k < n\). 
Figure \ref{FIG:coloured diagram} shows the correspondence 
between \(R\)-coloured unit diagrams \(E^n_n\) and 
rack \(n\)-chains, where each \(\epsilon\) denotes 
the signature of the corresponding diagram.
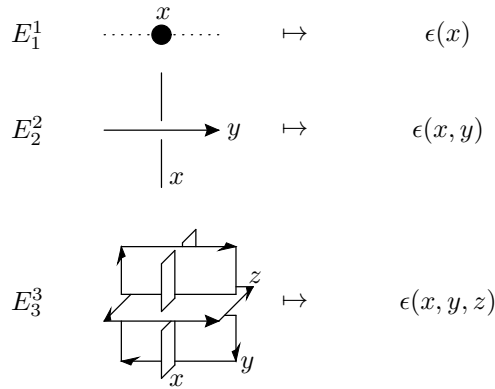
\begin{figure}[ht]
\begin{center}
\unitlength 0.1in%
\begin{picture}( 25.7000, 20.8500)( 3.1000,-25.0000)%
%
%
\special{pn 8}%
\special{pa 800 600}%
\special{pa 1400 600}%
\special{dt 0.045}%
\special{pn 8}%
\special{sh 1}%
\special{ar 1100 600 50 50  0.0000000 6.2831853}%
\put(11.1000,-4.9000){\makebox(0,0){\(x\)}}%
\put(4.0000,-6.0000){\makebox(0,0){\(E^1_1\)}}%
\put(18.0000,-6.0000){\makebox(0,0){\(\mapsto\)}}%
\put(26.0000,-6.0000){\makebox(0,0){\(\epsilon (x)\)}}%
%
%
\special{pn 8}%
\special{pa 800 1100}%
\special{pa 1400 1100}%
\special{fp}%
\special{pn 4}%
\special{sh 1}%
\special{pa 1400 1100}%
\special{pa 1330 1070}%
\special{pa 1340 1100}%
\special{pa 1330 1130}%
\special{pa 1400 1100}%
\special{fp}%
\special{pn 8}%
\special{pa 1100 800}%
\special{pa 1100 1050}%
\special{fp}%
\special{pn 8}%
\special{pa 1100 1150}%
\special{pa 1100 1400}%
\special{fp}%
\put(11.8000,-13.5000){\makebox(0,0){\(x\)}}%
\put(14.8000,-11.1000){\makebox(0,0){\(y\)}}%
\put(4.0000,-11.0000){\makebox(0,0){\(E^2_2\)}}%
\put(18.0000,-11.0000){\makebox(0,0){\(\mapsto\)}}%
\put(26.0000,-11.0000){\makebox(0,0){\(\epsilon (x, y)\)}}%
%
%
\special{pn 8}%
\special{pa 940 1960}%
\special{pa 800 2100}%
\special{fp}%
\special{pn 4}%
\special{sh 1}%
\special{pa 800 2100}%
\special{pa 835 2035}%
\special{pa 830 2070}%
\special{pa 835 2095}%
\special{pa 800 2100}%
\special{fp}%
\special{pn 8}%
\special{pa 800 2100}%
\special{pa 1400 2100}%
\special{fp}%
\special{pn 4}%
\special{sh 1}%
\special{pa 1400 2100}%
\special{pa 1330 2070}%
\special{pa 1340 2100}%
\special{pa 1330 2130}%
\special{pa 1400 2100}%
\special{fp}%
\special{pn 8}%
\special{pa 1400 2100}%
\special{pa 1580 1920}%
\special{fp}%
\special{pn 4}%
\special{sh 1}%
\special{pa 1580 1920}%
\special{pa 1545 1925}%
\special{pa 1550 1950}%
\special{pa 1545 1985}%
\special{pa 1580 1920}%
\special{fp}%
\special{pn 8}%
\special{pa 1580 1920}%
\special{pa 1490 1920}%
\special{fp}%
\special{pn 8}%
\special{pa 890 1710}%
\special{pa 1490 1710}%
\special{fp}%
\special{pn 4}%
\special{sh 1}%
\special{pa 1490 1710}%
\special{pa 1440 1690}%
\special{pa 1430 1710}%
\special{pa 1400 1730}%
\special{pa 1490 1710}%
\special{fp}%
\special{pn 8}%
\special{pa 1490 1710}%
\special{pa 1490 1960}%
\special{fp}%
\special{pn 8}%
\special{pa 1490 1960}%
\special{pa 1170 1960}%
\special{fp}%
\special{pn 8}%
\special{pa 1100 1960}%
\special{pa 890 1960}%
\special{fp}%
\special{pn 8}%
\special{pa 890 1960}%
\special{pa 890 1710}%
\special{fp}%
\special{pn 4}%
\special{sh 1}%
\special{pa 890 1710}%
\special{pa 870 1800}%
\special{pa 890 1770}%
\special{pa 910 1760}%
\special{pa 890 1710}%
\special{fp}%
\special{pn 8}%
\special{pa 1430 2070}%
\special{pa 1490 2070}%
\special{fp}%
\special{pn 8}%
\special{pa 1490 2070}%
\special{pa 1490 2310}%
\special{fp}%
\special{pn 4}%
\special{sh 1}%
\special{pa 1490 2310}%
\special{pa 1510 2220}%
\special{pa 1490 2250}%
\special{pa 1470 2260}%
\special{pa 1490 2310}%
\special{fp}%
\special{pn 8}%
\special{pa 1490 2310}%
\special{pa 1170 2310}%
\special{fp}%
\special{pn 8}%
\special{pa 1100 2310}%
\special{pa 890 2310}%
\special{fp}%
\special{pn 4}%
\special{sh 1}%
\special{pa 890 2310}%
\special{pa 940 2330}%
\special{pa 950 2310}%
\special{pa 980 2290}%
\special{pa 890 2310}%
\special{fp}%
\special{pn 8}%
\special{pa 890 2310}%
\special{pa 890 2100}%
\special{fp}%
\special{pn 8}%
\special{pa 1170 1730}%
\special{pa 1100 1800}%
\special{fp}%
\special{pn 8}%
\special{pa 1100 1800}%
\special{pa 1100 2050}%
\special{fp}%
\special{pn 8}%
\special{pa 1100 2050}%
\special{pa 1170 1980}%
\special{fp}%
\special{pn 8}%
\special{pa 1170 1980}%
\special{pa 1170 1730}%
\special{fp}%
\special{pn 8}%
\special{pa 1150 2100}%
\special{pa 1100 2150}%
\special{fp}%
\special{pn 8}%
\special{pa 1100 2150}%
\special{pa 1100 2400}%
\special{fp}%
\special{pn 8}%
\special{pa 1100 2400}%
\special{pa 1170 2330}%
\special{fp}%
\special{pn 8}%
\special{pa 1170 2330}%
\special{pa 1170 2100}%
\special{fp}%
\special{pn 8}%
\special{pa 1280 1710}%
\special{pa 1280 1620}%
\special{fp}%
\special{pn 8}%
\special{pa 1280 1620}%
\special{pa 1210 1690}%
\special{fp}%
\special{pn 8}%
\special{pa 1210 1690}%
\special{pa 1210 1710}%
\special{fp}%
\put(11.8000,-24.1000){\makebox(0,0){\(x\)}}%
\put(15.5000,-23.3000){\makebox(0,0){\(y\)}}%
\put(15.9000,-18.7000){\makebox(0,0){\(z\)}}%
\put(4.0000,-20.0000){\makebox(0,0){\(E^3_3\)}}%
\put(18.0000,-20.0000){\makebox(0,0){\(\mapsto\)}}%
\put(26.0000,-20.0000){\makebox(0,0){\(\epsilon (x, y, z)\)}}%
%
%
\end{picture}%
\end{center}
\caption{Coloured diagrams and corresponding chains.}
\label{FIG:coloured diagram}
\end{figure}
On the other hand,
Figure \ref{FIG:shadow coloured diagram} shows how 
\(R\)-shadow coloured unit diagrams \(E^n_n\)  
correspond to rack \((n + 1)\)-chains. 
\begin{figure}[ht]
\begin{center}
\input{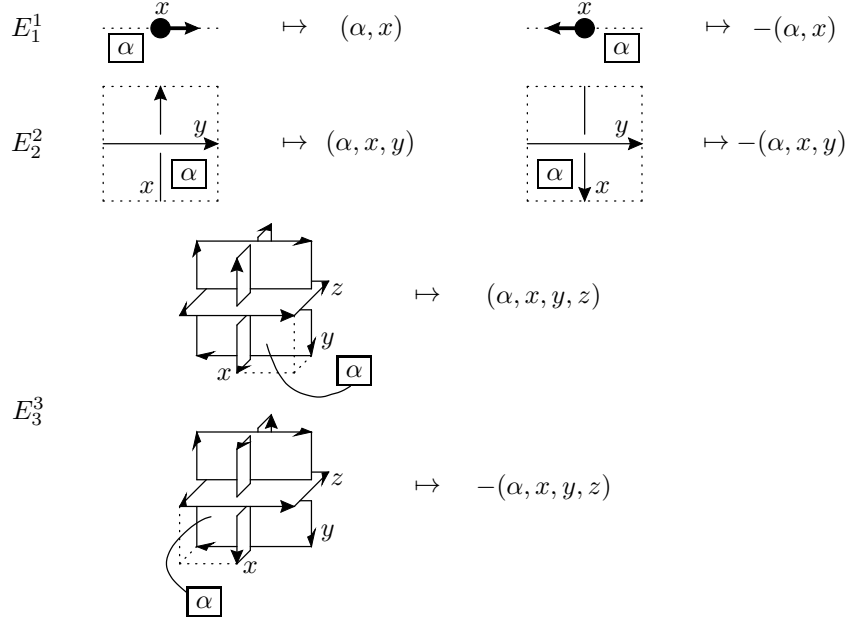}
\end{center}
\caption{Shadow coloured diagrams and corresponding chains.}
\label{FIG:shadow coloured diagram}
\end{figure}
The shadow colour \(\alpha\) 
in Figure \ref{FIG:shadow coloured diagram} is
that of \(r^\ini_p\) of the origin \(p\). 
For an \(R\)-(shadow) coloured unit diagram \(E^n_k\),
denote the corresponding chain by \(\langle E^n_k \rangle\). 

We recall that a diagram \(D\) is the union 
of the images of unit diagrams 
\(E_1\), \ldots, \(E_k\).
Suppose that \(D\) is \(R\)-(shadow) coloured.
Since the (shadow) colouring of \(D\) induces 
that of each unit diagram, 
we can consider a rack chain 
\begin{quote}
\(\langle D \rangle 
= \sum\limits_{i = 1}^k \langle E_i \rangle\) 
\end{quote}
as the corresponding chain of \(D\). 
For a rack chain \(c\), a (shadow) coloured diagram \(D\) 
is said to represent \(c\) 
if the corresponding chain \(\langle D \rangle\) equals to \(c\). \\

The boundary \(\partial D = (\partial M, \partial M \cap D)\) 
of a (shadow) coloured diagram \(D\) 
is also a (shadow) coloured diagram. 
We easily show that the corresponding chain 
\(\langle \partial D \rangle\) of \(\partial D\) is the image of 
\(\langle D \rangle\) via the boundary map of \(C^\rck_\ast(R)\). 
In Figure \ref{FIG:boundary chain}, 
\begin{figure}[ht]
\begin{center}
\input{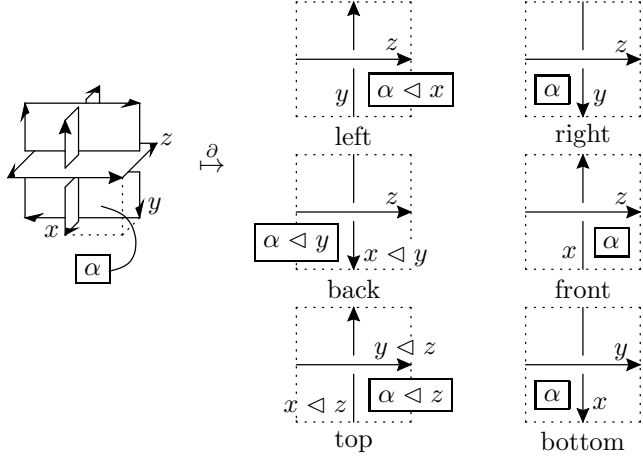}
\end{center}
\caption{Boundary faces of \(E^2_2\).}
\label{FIG:boundary chain}
\end{figure}
we draw a picture of the boundary of \(E^3_3\). 
We can easily check that 
\begin{quote}
\(\partial (x, y, z) = (y, z) - (y, z) - (x \lhd y, z) + (x, z) 
+ (x \lhd z, y \lhd z) - (x, y)\)
\end{quote}
or
\begin{quote}
\(\partial (\alpha, x, y, z) 
= (\alpha \lhd x, y, z) -  (\alpha, y, z)
- (\alpha \lhd y, x \lhd y, z) \\
\hskip60pt + (\alpha, x, z) 
+ (\alpha \lhd z, x \lhd z, y \lhd z) - (\alpha, x, y)\)
\end{quote}
holds in the figure.
Therefore, if a diagram \(D\) is on a closed manifold \(M\), 
the corresponding chain \(\langle D \rangle\) is a rack cycle. 
We have the inverse: 

\begin{theorem}[Carter-Kamada-Saito {[CKS1]}]
\label{TH:realization}
For \(n = 1\), \(2\) and \(3\),
let \(c\) be a rack \(n\)-cycle of \(R\). 
There exists a coloured \(n\)- 
or shadow coloured \((n - 1)\)-diagram \(D\) 
on a closed manifold \(M\) 
such that \(D\) represents \(c\). 
\end{theorem}

\begin{proof}
We prove this only for the case that 
a rack \(3\)-cycle is represented by 
a shadow coloured \(2\)-diagram. 
See [CKS1] for the precise proof. 

A \(3\)-cycle \(c\) is written in the form of 
\begin{quote}
\(\sum\limits_{i = 1}^k \epsilon_i(\alpha_i, x_i, y_i)\), 
\end{quote}
where \(\epsilon_i\) is \(+ 1\) or \(- 1\), 
and \(\alpha_i\), \(x_i\) and \(y_i \in R\). 
Let \(E_1\), \ldots, \(E_k\) be copies of 
the unit \(2\)-diagram \(E^2_2\) 
and give them \(R\)-shadow colourings and signatures 
so that each \(E_i\) represents \(\epsilon_i (\alpha_i, x_i, y_i)\). 
There are \(4k\) faces of \(E_1\), \ldots, \(E_k\). 
Since \(\partial c = 0\), 
we can name the faces as \(F_1\), \ldots, \(F_{2k}\), 
\(F'_1\), \ldots, \(F'_{2k}\) such that 
\(\langle F_j \rangle + \langle F'_j \rangle = 0\) for each \(j\). 
Therefore, by attaching \(F_j\) and \(F'_j\) canonically,
we obtain a diagram \(D\) on an oriented closed surface \(M\) 
such that \(\langle D \rangle = c\) holds.
\end{proof}

If an \(R\)-(shadow) coloured diagram \(D\) 
represents a rack cycle, 
we denote by \([D]\) the rack homology class 
which \(\langle D \rangle\) belongs to.

\subsection{Knot quandles and their presentations.}
\label{SEC:knot quandle}

Let \(D\) be a regular projection 
of an oriented link \(L\) 
on \(\sphere^2\), or a link diagram of \(L\). 
It is clear that \(D\) is a diagram on \(\sphere^2\) 
in the sense of \S \ref{SEC:diagram}. 

Joyce [J] defined the knot quandles topologically, 
but he proved in the same paper that 
the knot quandles can be defined 
through Wirtinger presentations. 
We will use the second definition of knot quandles. 
The symbols here follow that in \S \ref{SEC:colouring}. 

The knot quandle \(Q(L)\) of a link \(L\) is 
generated by \(\mathcal{C}(D)\) with relations 
\begin{quote}
\(u^\ini_p \lhd o_p = u^\ter_p\)
\end{quote}
for each double point (i.e., crossing) \(p\). 
It is well known that \(Q(L)\) completely determines 
the unoriented link type of \(L\).
Directly from the definition, 
the diagram \(D\) can be considered 
as coloured by \(Q(L)\). 
We call this colouring the canonical colouring 
of a link diagram \(D\). 

Since the knot group \(\pi(L)\) of \(L\) 
has the Wirtinger presentation obtained from 
that of \(Q(L)\) by replacing \(u^\ini_p \lhd o_p\) 
with \(o_p^{- 1} u^\ini_p o_p\), 
the associated group of \(Q(L)\) is equivalent to \(\pi(L)\). 
Notice that the canonical homomorphism in this case is 
induced by the identity map on 
the generating set \(\mathcal{C}(D)\). \\

In [E], the structure of the second quandle 
(co)homology group 
of a knot quandle is determined:

\begin{theorem}[Eisermann {[E]}]
\label{TH:second homology}
Let \(L\) be an \(n\)-component link 
and let \(m\) be the number 
of non-trivial components of \(L\). 
The knot quandle \(Q(L)\) has its second 
quandle (co)homology groups 
\begin{quote}
\(H^\qdl_2(Q(L); \Z) \cong 
H_\qdl^2(Q(L); \Z) \cong \Z^m\).
\end{quote}
\end{theorem}

\begin{remark}
\label{REM:fundamental class}
As noticed,
since \(D\) is a \(Q(L)\)-coloured 
\(2\)-diagram on \(\sphere^2\), 
it represents a \(2\)-cycle \(\langle D\rangle\) of \(Q(L)\). 
We call the second homology class \([D]\) 
corresponding to \(\langle D\rangle\)
the \textbf{diagram class} of \(D\), 
and call its image \([L]\) via \(\rho_\ast\): 
\(H^\rck_2(Q(L)) \to H^\qdl_2(Q(L))\) 
the \textbf{fundamental class} of \(L\).

It is shown in [E] that the fundamental class is 
uniquely determined, and, if \(L\) is non-trivial,
\([L]\) is proved to be a non-zero element 
of \(H^\qdl_2(Q(L))\).
Moreover, when \(L\) is a non-trivial knot,
\(H^\qdl_2(Q(L))\) is shown to be \(\Z[L]\). 
\end{remark}

\section{Lemmas on coloured diagrams.}
\label{CHAP:lemma}

Though our purpose is to construct 
third homology classes of knot quandles, 
\S \ref{CHAP:lemma} is devoted to lemmas 
on \(1\)-cycles and on \(2\)-chains 
of quandles in general,
which play important roles in the later sections.

\subsection{Shadow colourability of \(1\)-diagrams.}
\label{SEC:1-diagram}

Let \(Q\) be a quandle. 
Though we are concerned with rack \(1\)-cycles, 
it is useful to suppose \(Q\) to be a quandle here,
for the associated group \(G_Q\) is considered.
As already mentioned (Theorem \ref{TH:realization}), 
a rack \(1\)-cycle \(c \in Z^\rck_1(Q)\) can be 
represented by a \(Q\)-coloured \(1\)-diagram 
on a closed \(1\)-manifold,
that is, by a diagram on a disjoint union of copies 
of \(\sphere^1\).
Shadow colourability can be considered on 
each circle independently. 
So it is sufficient only to consider diagrams 
on \(\sphere^1\). 

Let \(D\) be a \(Q\)-coloured \(1\)-diagram on \(\sphere^1\). 
By reading colours and signatures of \(D\) 
along \(\sphere^1\) starting from its basepoint, 
we obtain two sequences \((x_1, \ldots, x_n)\) 
of elements of \(Q\) 
and \((\epsilon_1, \ldots, \epsilon_n)\) of \(\pm 1\). 
Denote by \(\Pi(D)\) the product 
\(x_1^{\epsilon_1} \cdots x_n^{\epsilon_n}\) 
in the associated group \(G_Q\). 

\begin{lemma}
\label{LEM:shadow 1-diagram}
Let \(\alpha\) be an element of \(Q\). 
If a \(Q\)-coloured \(1\)-diagram \(D\) on \(\sphere^1\) 
has a \(Q\)-shadow colouring 
such that the shadow colour of the base-region is \(\alpha\), 
then \(\Pi(D)\) commutes with \(\alpha\). 

Moreover, when the canonical map \(Q \to G_Q\) 
is injective, the inverse holds.
\end{lemma}

\begin{proof}
Suppose that \(D\) is \(Q\)-shadow coloured 
as in the statement. 
Obviously, we obtain an equation 
\begin{quote}
\(( \cdots (\alpha \lhd^{\epsilon_1} x_1) 
\cdots ) \lhd^{\epsilon_n} x_n = \alpha\) 
\end{quote}
(see Figure \ref{FIG:shadow 1-diagram}). 
It follows that 
\((x_1^{\epsilon_1} \cdots x_n^{\epsilon_n})^{- 1} 
\alpha (x_1^{\epsilon_1} \cdots x_n^{\epsilon_n}) = \alpha\) 
holds in \(G_Q\), that is, \(\Pi(D)\) commutes with \(\alpha\). 
\begin{figure}[ht]
\begin{center}
\unitlength 0.1in%
\begin{picture}(40,14.6)(10.5,-16.0)%
%
\special{pn 8}%
\special{ar 2400 996 600 600  1.8925469 6.2831853}%
\special{ar 2400 996 600 600  0.0000000 1.2490458}%
%
\put(24.0000,-16){\makebox(0,0){\(\cdots\)}}%
%
\put(24.0000,-10){\makebox(0,0){\(\circlearrowleft\)}}%
%
\special{pn 8}%
\special{sh 1}%
\special{pa 2410 420}%
\special{pa 2390 420}%
\special{pa 2390 370}%
\special{pa 2410 370}%
\special{pa 2410 420}%
\special{fp}%
%
\put(24.0000,-5.2000){\makebox(0,0){\(\ast\)}}%
%
\put(24.0000,-2.5){\makebox(0,0){\fbox{\(\alpha\)}}}%
%
\special{pn 8}%
\special{sh 1}%
\special{ar 1950 600 50 50  0.0000000 6.2831853}%
%
\special{pn 8}%
\special{sh 1}%
\special{ar 2850 600 50 50  0.0000000 6.2831853}%
%
\special{pn 8}%
\special{sh 1}%
\special{ar 2850 1400 50 50  0.0000000 6.2831853}%
%
\special{pn 8}%
\special{sh 1}%
\special{ar 1950 1400 50 50  0.0000000 6.2831853}%
%
\put(16.0000,-6.0000){\makebox(0,0){\((x_1, \epsilon_1)\)}}%
\put(16.0000,-14.0000){\makebox(0,0){\((x_2, \epsilon_2)\)}}%
\put(33.0000,-14.0000){\makebox(0,0){\((x_{n - 1}, \epsilon_{n - 1})\)}}%
\put(32.0000,-6.0000){\makebox(0,0){\((x_n, \epsilon_n)\)}}%
%
\put(14.0000,-10.0000){\makebox(0,0)%
{\fbox{\(\alpha \lhd^{\epsilon_1} x_1\)}}}%
%
\put(40.5,-10.0000){\makebox(0,0)%
{\fbox{\(( \cdots (\alpha \lhd^{\epsilon_1} x_1) %
\cdots ) \lhd^{\epsilon_{n - 1}} x_{n - 1}\)}}}%
\end{picture}%
\end{center}
\caption{Shadow coloured \(1\)-diagram.}
\label{FIG:shadow 1-diagram}
\end{figure}
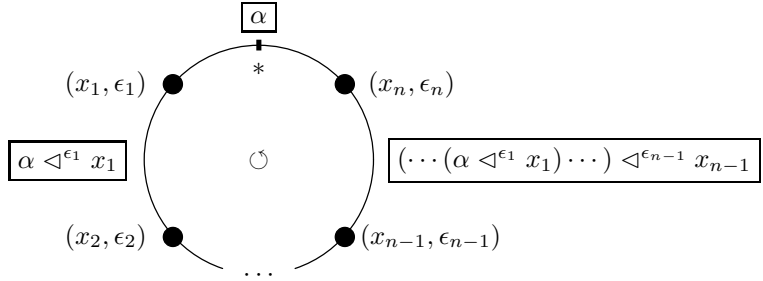

Inversely, if \(\Pi(D)\) commutes with \(\alpha\)
in \(G_Q\), the elements of \(Q\),
\(( \cdots (\alpha \lhd^{\epsilon_1} x_1) 
\cdots ) \lhd^{\epsilon_n} x_n\)
and \(\alpha\), map to the same element of \(G_Q\).
Therefore, when the canonical map \(Q \to G_Q\)
supposed to be injective, \(( \cdots (\alpha \lhd^{\epsilon_1} x_1) 
\cdots ) \lhd^{\epsilon_n} x_n = \alpha\) holds also in \(Q\).
Then, by giving shadow colours to regions along \(\sphere^1\),
we obtain a \(Q\)-shadow colouring of \(D\) 
without contradiction. 
\end{proof}

We denote by \(D_\alpha\) the \(Q\)-shadow coloured 
diagram as in Lemma \ref{LEM:shadow 1-diagram}. 
In other dimensional cases, if a whole manifold \(M\) is 
connected, we also denote by \(D_\alpha\) 
the \(Q\)-shadow coloured diagram which is obtained 
from a \(Q\)-coloured diagram \(D\) by colouring 
the base-region with \(\alpha\). 
A \(Q\)-coloured diagram \(D\) is called to be 
\textbf{freely \(Q\)-shadow colourable} 
if \(D_\alpha\) exists for any \(\alpha \in Q\). 
As a consequence of Lemma \ref{LEM:shadow 1-diagram}, 
we obtain: 

\begin{corollary}
\label{COR:free shadow}
If a \(Q\)-coloured \(1\)-diagram \(D\) 
on \(\sphere^1\) is freely \(Q\)-shadow col\-our\-able,
\(\Pi(D)\) belongs to the centre of \(G_Q\).
\end{corollary}

We have seen the facts on the shadow 
colourability of \(1\)-diagrams.
Before dealing with the shadow colourability of 
\(2\)-diagrams in \S \ref{SEC:2-diagram},
we focus on the way to construct a coloured 
\(2\)-diagram which connects rack homologous 
\(1\)-cycles.

\begin{lemma}
\label{LEM:cobordism 1-diagram}
Let \(D\) and \(D'\) be 
\(Q\)-coloured \(1\)-diagrams on \(\sphere^1\). 
There exists a \(Q\)-coloured \(2\)-diagram \(\tilde{D}\) 
on an annulus \(\sphere^1 \times I\) such that 
\(\partial \tilde{D} = D \cup (- D')\), 
if and only if \(\Pi(D)\) and \(\Pi(D')\) are conjugate in \(G_Q\). 
\end{lemma}

\begin{proof}
As in Lemma \ref{LEM:shadow 1-diagram}, 
let \((x_1, \ldots, x_n)\) and \((\epsilon_1, \ldots, \epsilon_n)\) 
be the colours and the signatures of \(D\), 
and let \((y_1, \ldots, y_m)\) and \((\delta_1, \ldots, \delta_m)\) 
be those of \(D'\). 
If \(\Pi(D)\) and \(\Pi(D')\) are conjugate, 
two words \(x_1^{\epsilon_1} \cdots x_n^{\epsilon_n}\) 
and \(y_1^{\delta_1} \cdots y_m^{\delta_m}\) are 
connected by a finite sequence of replacing operations 
each of which is one of three types shown later. 
There exist subdiagrams on \(\sphere^1 \times I\) 
that correspond to these operations, 
and by gluing them we obtain a new diagram \(\tilde{D}\). 
As the orientation of its boundary \(\sphere^1 \times \{\pm 1\}\) 
is considered, \(\partial \tilde{D}\) is 
the disjoint union of two diagrams \(D\) and \(- D'\). 

To prove ``only if'' part, 
we recall the way to decompose 
tangles or braids into fundamental diagrams. 
By a similar argument, 
we obtain decomposability of \(1\)-diagrams 
on \(\sphere^1 \times I\)
into subdiagrams as depicted 
in Figures \ref{FIG:operation type1}, 
\ref{FIG:operation type2} 
and \ref{FIG:operation type3}. 
Therefore, directly from this fact, 
we see that \(\Pi(D)\) and \(\Pi(D')\) are conjugate.

Now we will construct subdiagrams 
which correspond to the replacing operations. 
In the figures, 
\(w_1\), \(w_2\) and \(w\) are words, 
and \(a\) and \(b\) are elements of \(Q\). 
Each rectangle is considered to be an annulus 
as identified the left and the right edges. \\

\noindent
1) Figure \ref{FIG:operation type1} shows 
an operation \(w_1 a^{\pm 1} a^{\mp 1} w_2 
\longleftrightarrow w_1 w_2\). 
The indices \(\pm 1\) of \(a\)'s are their signatures, 
which are determined by the orientation of 
the arc adjacent to them. 
We also allow the diagram turned upside down. 
\begin{figure}[ht]
\begin{center}
\unitlength 0.1in%
\begin{picture}(16.0800,10.3)(1.9200,-12.1)%
%
\special{pn 8}%
\special{pa 200 400}%
\special{pa 1800 400}%
\special{pa 1800 1000}%
\special{pa 200 1000}%
\special{pa 200 400}%
\special{dt 0.045}%
%
\special{pn 8}%
\special{pa 300 400}%
\special{pa 300 1000}%
\special{fp}%
%
\put(5.0000,-7.0000){\makebox(0,0){\(\cdots\)}}%
%
\special{pn 8}%
\special{pa 700 400}%
\special{pa 700 1000}%
\special{fp}%
%
\put(5.0000,-11.0000){\makebox(0,0){\(\underbrace{\hskip30pt}\)}}%
%
\put(5.0000,-12.0000){\makebox(0,0){\(w_1\)}}%
%
\special{pn 8}%
\special{pa 1300 400}%
\special{pa 1300 1000}%
\special{fp}%
%
\put(15.0000,-7.0000){\makebox(0,0){\(\cdots\)}}%
%
\special{pn 8}%
\special{pa 1700 400}%
\special{pa 1700 1000}%
\special{fp}%
%
\put(15.0000,-11.0000){\makebox(0,0){\(\underbrace{\hskip30pt}\)}}%
%
\put(15.0000,-12.0000){\makebox(0,0){\(w_2\)}}%
%
\special{pn 8}%
\special{pa 800 400}%
\special{pa 800 600}%
\special{fp}%
%
\special{pn 8}%
\special{pa 1200 400}%
\special{pa 1200 600}%
\special{fp}%
%
\special{pn 8}%
\special{ar 1000 600 200 200  6.2831853 6.2831853}%
\special{ar 1000 600 200 200  0.0000000 3.1415927}%
%
\special{pn 8}%
\special{sh 1}%
\special{ar 800 400 40 40  0.0000000 6.2831853}%
%
\put(8.9000,-2.8000){\makebox(0,0){\(a^{\pm 1}\)}}%
%
\special{pn 8}%
\special{sh 1}%
\special{ar 1200 400 40 40  0.0000000 6.2831853}%
%
\put(12.9000,-2.8000){\makebox(0,0){\(a^{\mp 1}\)}}%
\end{picture}%
\end{center}
\caption{Replacing operation of type I.}
\label{FIG:operation type1}
\end{figure}
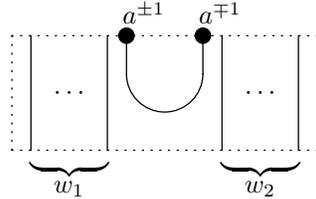

\noindent
2) The operation of type II comes from 
the definition of associated group. 
Replacing such as
\begin{quote}
\(w_1 (a \lhd^\epsilon b)^{\pm 1} w_2 
\longleftrightarrow w_1 b^{- \epsilon} a^{\pm 1} b^{\epsilon} w_2\) 
\end{quote}
is considered, where \(\epsilon = \pm 1\). 
The corresponding subdiagram 
is depicted in Figure \ref{FIG:operation type2}, 
where \(\epsilon\) is the signature of the crossing.
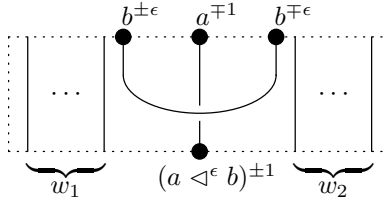
\begin{figure}[ht]
\begin{center}
\unitlength 0.1in%
\begin{picture}(20.0800,10.3)(1.9200,-12.1)%
%
\special{pn 8}%
\special{pa 200 400}%
\special{pa 2200 400}%
\special{pa 2200 1000}%
\special{pa 200 1000}%
\special{pa 200 400}%
\special{dt 0.045}%
%
\special{pn 8}%
\special{pa 300 400}%
\special{pa 300 1000}%
\special{fp}%
%
\put(5.0000,-7.0000){\makebox(0,0){\(\cdots\)}}%
%
\special{pn 8}%
\special{pa 700 400}%
\special{pa 700 1000}%
\special{fp}%
%
\put(5.0000,-11.0000){\makebox(0,0){\(\underbrace{\hskip30pt}\)}}%
%
\put(5.0000,-12.0000){\makebox(0,0){\(w_1\)}}%
%
\special{pn 8}%
\special{pa 1700 400}%
\special{pa 1700 1000}%
\special{fp}%
%
\put(19.0000,-7.0000){\makebox(0,0){\(\cdots\)}}%
%
\special{pn 8}%
\special{pa 2100 400}%
\special{pa 2100 1000}%
\special{fp}%
%
\put(19.0000,-11.0000){\makebox(0,0){\(\underbrace{\hskip30pt}\)}}%
%
\put(19.0000,-12.0000){\makebox(0,0){\(w_2\)}}%
%
\special{pn 8}%
\special{pa 1200 400}%
\special{pa 1200 1000}%
\special{fp}%
%
\special{pn 8}%
\special{sh 0}%
\special{ia 1200 800 40 40  0.0000000 6.2831853}%
%
\special{pn 8}%
\special{sh 1}%
\special{ar 1200 400 40 40  0.0000000 6.2831853}%
%
\special{pn 8}%
\special{sh 1}%
\special{ar 1200 1000 40 40  0.0000000 6.2831853}%
%
\put(12.9000,-2.8000){\makebox(0,0){\(a^{\mp 1}\)}}%
%
\put(12.9000,-11.4000){\makebox(0,0){\((a \lhd^\epsilon b)^{\pm 1}\)}}%
%
\special{pn 8}%
\special{pa 800 400}%
\special{pa 800 600}%
\special{fp}%
%
\special{pn 8}%
\special{pa 1600 400}%
\special{pa 1600 600}%
\special{fp}%
%
\special{pn 8}%
\special{ar 1200 600 400 200  6.2831853 6.2831853}%
\special{ar 1200 600 400 200  0.0000000 3.1415927}%
%
\special{pn 8}%
\special{sh 1}%
\special{ar 1600 400 40 40  0.0000000 6.2831853}%
%
\special{pn 8}%
\special{sh 1}%
\special{ar 800 400 40 40  0.0000000 6.2831853}%
%
\put(8.9000,-2.8000){\makebox(0,0){\(b^{\pm \epsilon}\)}}%
%
\put(16.9000,-2.8000){\makebox(0,0){\(b^{\mp \epsilon}\)}}%
\end{picture}%
\end{center}
\caption{Replacing operation of type II.}
\label{FIG:operation type2}
\end{figure}

\noindent
3) The last operation is conjugation, 
that is, 
\(w a^{\pm 1} \longleftrightarrow a^{\pm 1} w\). 
It corresponds to a subdiagram 
where a component crosses a line \(\{\ast\} \times I 
\subset \sphere^1 \times I\) as drawn 
in Figure \ref{FIG:operation type3}.
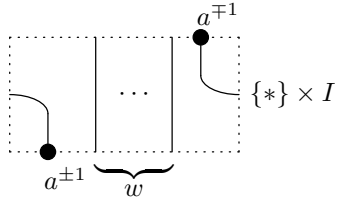
\begin{figure}[ht]
\begin{center}
\unitlength 0.1in%
\begin{picture}(18.0800, 10.2000)(1.9200,-12.3000)%
%
\special{pn 8}%
\special{pa 1400 400}%
\special{pa 200 400}%
\special{pa 200 1000}%
\special{pa 1400 1000}%
\special{pa 1400 400}%
\special{dt 0.045}%
%
\put(16.8,-7.0000){\makebox(0,0){\(\{\ast\} \times I\)}}%
%
\special{pn 8}%
\special{pa 650 400}%
\special{pa 650 1000}%
\special{fp}%
%
\put(8.5,-7.0000){\makebox(0,0){\(\cdots\)}}%
%
\special{pn 8}%
\special{pa 1050 400}%
\special{pa 1050 1000}%
\special{fp}%
%
\put(8.5,-11.0000){\makebox(0,0){\(\underbrace{\hskip30pt}\)}}%
%
\put(8.5,-12.0000){\makebox(0,0){\(w\)}}%
%
\special{pn 8}%
\special{ar 200 800 200 100  4.7123890 6.2831853}%
%
\special{pn 8}%
\special{pa 400 800}%
\special{pa 400 1000}%
\special{fp}%
%
\special{pn 8}%
\special{sh 1}%
\special{ar 400 1000 40 40  0.0000000 6.2831853}%
%
\put(4.9000,-11.4000){\makebox(0,0){\(a^{\pm 1}\)}}%
%
\special{pn 8}%
\special{ar 1400 600 200 100  1.5707963 3.1415927}%
%
\special{pn 8}%
\special{pa 1200 400}%
\special{pa 1200 600}%
\special{fp}%
%
\special{pn 8}%
\special{sh 1}%
\special{ar 1200 400 40 40  0.0000000 6.2831853}%
%
\put(12.9000,-2.8000){\makebox(0,0){\(a^{\mp 1}\)}}%
\end{picture}%
\end{center}
\caption{Replacing operation of type III.}
\label{FIG:operation type3}
\end{figure}

Thus we have completed the proof 
of Lemma \ref{LEM:cobordism 1-diagram}.
\end{proof}

\newpage
\begin{remark}
\label{REM:revised operation}
We can exchange the operation of type II with 
\begin{quote}
\(w_1 b^{\epsilon} (a \lhd^{\epsilon} b)^{\pm 1} w_2 
\longleftrightarrow w_1 a^{\pm 1} b^{\epsilon} w_2\)
\end{quote}
and 
\begin{quote}
\(w_1 (a \lhd^{\epsilon} b)^{\pm 1} b^{- \epsilon}  w_2 
\longleftrightarrow w_1 b^{- \epsilon} a^{\pm 1} w_2\). 
\end{quote}
Figure \ref{FIG:revised operation} shows the subdiagram corresponding 
to the first operation, and shows its decomposition into 
the product of subdiagrams of types I and II.
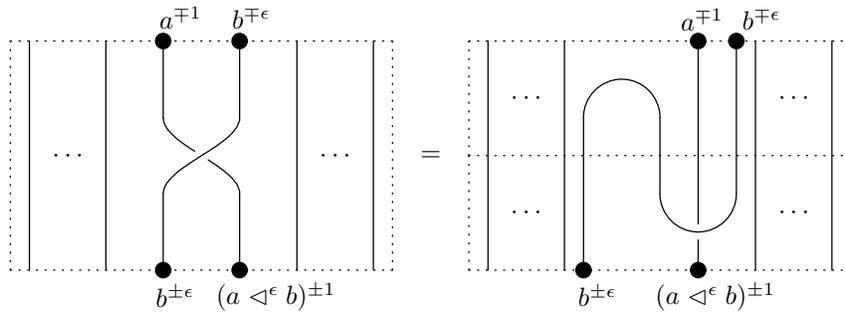
\begin{figure}[ht]
\begin{center}
\unitlength 0.1in%
\begin{picture}( 44.0800,16.3000)(1.9200,-18.1500)%
%
\special{pn 8}%
\special{pa 200 400}%
\special{pa 2200 400}%
\special{pa 2200 1600}%
\special{pa 200 1600}%
\special{pa 200 400}%
\special{dt 0.045}%
%
\special{pn 8}%
\special{pa 300 400}%
\special{pa 300 1600}%
\special{fp}%
%
\put(5.0000,-10.0000){\makebox(0,0){\(\cdots\)}}%
%
\special{pn 8}%
\special{pa 700 400}%
\special{pa 700 1600}%
\special{fp}%
%
\special{pn 8}%
\special{pa 1700 400}%
\special{pa 1700 1600}%
\special{fp}%
%
\put(19.0000,-10.0000){\makebox(0,0){\(\cdots\)}}%
%
\special{pn 8}%
\special{pa 2100 400}%
\special{pa 2100 1600}%
\special{fp}%
%
\special{pn 8}%
\special{pa 1000 400}%
\special{pa 1000 800}%
\special{fp}%
%
\special{pn 8}%
\special{pa 1000 800}%
\special{pa 1000 806}%
\special{pa 1002 810}%
\special{pa 1002 816}%
\special{pa 1002 820}%
\special{pa 1004 826}%
\special{pa 1006 830}%
\special{pa 1008 836}%
\special{pa 1010 840}%
\special{pa 1012 846}%
\special{pa 1016 850}%
\special{pa 1018 856}%
\special{pa 1022 860}%
\special{pa 1026 866}%
\special{pa 1030 870}%
\special{pa 1034 876}%
\special{pa 1038 880}%
\special{pa 1044 886}%
\special{pa 1048 890}%
\special{pa 1054 896}%
\special{pa 1060 900}%
\special{pa 1064 906}%
\special{pa 1070 910}%
\special{pa 1076 916}%
\special{pa 1082 920}%
\special{pa 1090 926}%
\special{pa 1096 930}%
\special{pa 1102 936}%
\special{pa 1110 940}%
\special{pa 1116 946}%
\special{pa 1124 950}%
\special{pa 1132 956}%
\special{pa 1138 960}%
\special{pa 1146 966}%
\special{pa 1154 970}%
\special{pa 1162 976}%
\special{pa 1170 980}%
\special{pa 1176 986}%
\special{pa 1184 990}%
\special{pa 1192 996}%
\special{pa 1200 1000}%
\special{pa 1208 1006}%
\special{pa 1216 1010}%
\special{pa 1224 1016}%
\special{pa 1232 1020}%
\special{pa 1240 1026}%
\special{pa 1248 1030}%
\special{pa 1254 1036}%
\special{pa 1262 1040}%
\special{pa 1270 1046}%
\special{pa 1278 1050}%
\special{pa 1284 1056}%
\special{pa 1292 1060}%
\special{pa 1298 1066}%
\special{pa 1304 1070}%
\special{pa 1312 1076}%
\special{pa 1318 1080}%
\special{pa 1324 1086}%
\special{pa 1330 1090}%
\special{pa 1336 1096}%
\special{pa 1342 1100}%
\special{pa 1348 1106}%
\special{pa 1352 1110}%
\special{pa 1358 1116}%
\special{pa 1362 1120}%
\special{pa 1366 1126}%
\special{pa 1372 1130}%
\special{pa 1374 1136}%
\special{pa 1378 1140}%
\special{pa 1382 1146}%
\special{pa 1386 1150}%
\special{pa 1388 1156}%
\special{pa 1390 1160}%
\special{pa 1392 1166}%
\special{pa 1394 1170}%
\special{pa 1396 1176}%
\special{pa 1398 1180}%
\special{pa 1400 1186}%
\special{pa 1400 1190}%
\special{pa 1400 1196}%
\special{pa 1400 1200}%
\special{sp}%
%
\special{pn 8}%
\special{pa 1400 1200}%
\special{pa 1400 1600}%
\special{fp}%
%
\special{pn 8}%
\special{sh 0}%
\special{ia 1200 1000 40 40  0.0000000 6.2831853}%
%
\special{pn 8}%
\special{pa 1400 400}%
\special{pa 1400 800}%
\special{fp}%
%
\special{pn 8}%
\special{pa 1400 800}%
\special{pa 1400 806}%
\special{pa 1400 810}%
\special{pa 1400 816}%
\special{pa 1398 820}%
\special{pa 1396 826}%
\special{pa 1394 830}%
\special{pa 1392 836}%
\special{pa 1390 840}%
\special{pa 1388 846}%
\special{pa 1386 850}%
\special{pa 1382 856}%
\special{pa 1378 860}%
\special{pa 1374 866}%
\special{pa 1372 870}%
\special{pa 1366 876}%
\special{pa 1362 880}%
\special{pa 1358 886}%
\special{pa 1352 890}%
\special{pa 1348 896}%
\special{pa 1342 900}%
\special{pa 1336 906}%
\special{pa 1330 910}%
\special{pa 1324 916}%
\special{pa 1318 920}%
\special{pa 1312 926}%
\special{pa 1304 930}%
\special{pa 1298 936}%
\special{pa 1292 940}%
\special{pa 1284 946}%
\special{pa 1278 950}%
\special{pa 1270 956}%
\special{pa 1262 960}%
\special{pa 1254 966}%
\special{pa 1248 970}%
\special{pa 1240 976}%
\special{pa 1232 980}%
\special{pa 1224 986}%
\special{pa 1216 990}%
\special{pa 1208 996}%
\special{pa 1200 1000}%
\special{pa 1192 1006}%
\special{pa 1184 1010}%
\special{pa 1176 1016}%
\special{pa 1170 1020}%
\special{pa 1162 1026}%
\special{pa 1154 1030}%
\special{pa 1146 1036}%
\special{pa 1138 1040}%
\special{pa 1132 1046}%
\special{pa 1124 1050}%
\special{pa 1116 1056}%
\special{pa 1110 1060}%
\special{pa 1102 1066}%
\special{pa 1096 1070}%
\special{pa 1090 1076}%
\special{pa 1082 1080}%
\special{pa 1076 1086}%
\special{pa 1070 1090}%
\special{pa 1064 1096}%
\special{pa 1060 1100}%
\special{pa 1054 1106}%
\special{pa 1048 1110}%
\special{pa 1044 1116}%
\special{pa 1038 1120}%
\special{pa 1034 1126}%
\special{pa 1030 1130}%
\special{pa 1026 1136}%
\special{pa 1022 1140}%
\special{pa 1018 1146}%
\special{pa 1016 1150}%
\special{pa 1012 1156}%
\special{pa 1010 1160}%
\special{pa 1008 1166}%
\special{pa 1006 1170}%
\special{pa 1004 1176}%
\special{pa 1002 1180}%
\special{pa 1002 1186}%
\special{pa 1002 1190}%
\special{pa 1000 1196}%
\special{pa 1000 1200}%
\special{sp}%
%
\special{pn 8}%
\special{pa 1000 1200}%
\special{pa 1000 1600}%
\special{fp}%
%
\special{pn 8}%
\special{sh 1}%
\special{ar 1000 400 40 40  0.0000000 6.2831853}%
%
\put(10.9000,-2.8000){\makebox(0,0){\(a^{\mp 1}\)}}%
%
\special{pn 8}%
\special{sh 1}%
\special{ar 1400 1600 40 40  0.0000000 6.2831853}%
%
\put(15.9000,-17.4000){\makebox(0,0){\((a \lhd^\epsilon b)^{\pm 1}\)}}%
%
\special{pn 8}%
\special{sh 1}%
\special{ar 1400 400 40 40  0.0000000 6.2831853}%
%
\put(14.6000,-2.8000){\makebox(0,0){\(b^{\mp \epsilon}\)}}%
%
\special{pn 8}%
\special{sh 1}%
\special{ar 1000 1600 40 40  0.0000000 6.2831853}%
%
\put(10.6000,-17.4000){\makebox(0,0){\(b^{\pm \epsilon}\)}}%
%
%
\put(24,-10){\makebox(0,0){\(=\)}}%
%
\special{pn 8}%
\special{pa 2600 400}%
\special{pa 4600 400}%
\special{pa 4600 1000}%
\special{pa 2600 1000}%
\special{pa 2600 400}%
\special{dt 0.045}%
%
\special{pn 8}%
\special{pa 4600 1000}%
\special{pa 4600 1600}%
\special{pa 2600 1600}%
\special{pa 2600 1000}%
\special{dt 0.045}%
%
\special{pn 8}%
\special{pa 2700 400}%
\special{pa 2700 1600}%
\special{fp}%
%
\put(29,-13){\makebox(0,0){\(\cdots\)}}%
%
\put(29,-7){\makebox(0,0){\(\cdots\)}}%
%
\special{pn 8}%
\special{pa 3100 400}%
\special{pa 3100 1600}%
\special{fp}%
%
\special{pn 8}%
\special{pa 4100 400}%
\special{pa 4100 1600}%
\special{fp}%
%
\put(43,-7){\makebox(0,0){\(\cdots\)}}%
%
\put(43,-13){\makebox(0,0){\(\cdots\)}}%
%
\special{pn 8}%
\special{pa 4500 400}%
\special{pa 4500 1600}%
\special{fp}%
%
\special{pn 8}%
\special{pa 3800 400}%
\special{pa 3800 1600}%
\special{fp}%
%
\special{pn 8}%
\special{sh 0}%
\special{ia 3800 1400 40 40  0.0000000 6.2831853}%
%
\special{pn 8}%
\special{sh 1}%
\special{ar 3800 400 40 40  0.0000000 6.2831853}%
%
\put(38.2,-2.9){\makebox(0,0){\(a^{\mp 1}\)}}%
%
\special{pn 8}%
\special{sh 1}%
\special{ar 3800 1600 40 40  0.0000000 6.2831853}%
%
\put(38.9,-17.4){\makebox(0,0){\((a \lhd^\epsilon b)^{\pm 1}\)}}%
%
\special{pn 8}%
\special{pa 3200 1600}%
\special{pa 3200 800}%
\special{fp}%
%
\special{pn 8}%
\special{ar 3400 800 200 200  3.1415927 6.2831853}%
%
\special{pn 8}%
\special{pa 3600 800}%
\special{pa 3600 1200}%
\special{fp}%
%
\special{pn 8}%
\special{ar 3800 1200 200 200  0.0000000 3.1415927}%
%
\special{pn 8}%
\special{pa 4000 1200}%
\special{pa 4000 400}%
\special{fp}%
%
\special{pn 8}%
\special{sh 1}%
\special{ar 3200 1600 40 40  0.0000000 6.2831853}%
%
\put(32.7,-17.3500){\makebox(0,0){\(b^{\pm \epsilon}\)}}%
%
\special{pn 8}%
\special{sh 1}%
\special{ar 4000 400 40 40  0.0000000 6.2831853}%
%
\put(41.3,-2.8500){\makebox(0,0){\(b^{\mp \epsilon}\)}}%
\end{picture}%
\end{center}
\caption{Revised operation and its decomposition.}
\label{FIG:revised operation}
\end{figure}
\end{remark}

For a trivial diagram \((\sphere^1, \emptyset)\),
the product \(\Pi(\sphere^1, \emptyset)\) is 
equal to \(e\).
Thus, supposing \(D'\) 
in Lemma \ref{LEM:cobordism 1-diagram} to be trivial,
we have:

\begin{corollary}
\label{COR:boundary 1-diagram}
A \(Q\)-coloured \(1\)-diagram \(D\) on \(\sphere^1\) is 
a boundary of some \(Q\)-coloured \(2\)-diagram \(\tilde{D}\)
on \(\disk^2\),
if and only if \(\Pi(D) = e\).
\end{corollary}

\begin{proof}
Let \(D'\) be a trivial diagram on \(\disk^2\). 
Since \(\Pi(D) = \Pi(\partial D') = e\), 
Lemma \ref{LEM:cobordism 1-diagram} shows that 
there exists a \(Q\)-coloured diagram \(\tilde{D}'\) 
on \(\sphere^1 \times I\) 
such that \(\partial \tilde{D}' = D \cup (- \partial D')\). 
Attaching \(D'\) to \(\tilde{D}'\) along \(\partial D'\), 
we obtain \(\tilde{D}\) on \(\disk^2\). 
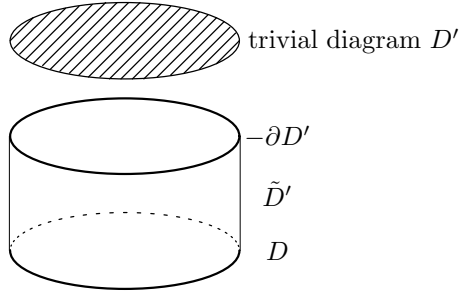
\begin{figure}[ht]
\begin{center}
\unitlength 0.1in%
\begin{picture}( 23.5000, 16.5000)(  7.9200,-18.5000)%
%
\special{pn 8}%
\special{ar 1400 500 600 200  0.0000000 6.2831853}%
%
\special{pn 8}%
\special{pa 920 380}%
\special{pa 800 500}%
\special{fp}%
\special{pa 1010 350}%
\special{pa 820 540}%
\special{fp}%
\special{pa 1090 330}%
\special{pa 850 570}%
\special{fp}%
\special{pa 1160 320}%
\special{pa 880 600}%
\special{fp}%
\special{pa 1230 310}%
\special{pa 920 620}%
\special{fp}%
\special{pa 1300 300}%
\special{pa 970 630}%
\special{fp}%
\special{pa 1360 300}%
\special{pa 1010 650}%
\special{fp}%
\special{pa 1420 300}%
\special{pa 1060 660}%
\special{fp}%
\special{pa 1480 300}%
\special{pa 1110 670}%
\special{fp}%
\special{pa 1530 310}%
\special{pa 1160 680}%
\special{fp}%
\special{pa 1590 310}%
\special{pa 1210 690}%
\special{fp}%
\special{pa 1640 320}%
\special{pa 1270 690}%
\special{fp}%
\special{pa 1690 330}%
\special{pa 1320 700}%
\special{fp}%
\special{pa 1740 340}%
\special{pa 1380 700}%
\special{fp}%
\special{pa 1790 350}%
\special{pa 1440 700}%
\special{fp}%
\special{pa 1835 365}%
\special{pa 1500 700}%
\special{fp}%
\special{pa 1880 380}%
\special{pa 1570 690}%
\special{fp}%
\special{pa 1920 400}%
\special{pa 1640 680}%
\special{fp}%
\special{pa 1950 430}%
\special{pa 1710 670}%
\special{fp}%
\special{pa 1980 460}%
\special{pa 1790 650}%
\special{fp}%
\special{pa 2000 500}%
\special{pa 1880 620}%
\special{fp}%
%
\put(26,-5.0000){\makebox(0,0){trivial diagram \(D'\)}}%
%
\special{pn 13}%
\special{ar 1400 1000 600 198  0.0000000 6.2831853}%
%
\put(22.0000,-10.0000){\makebox(0,0){\(- \partial D'\)}}%
%
\special{pn 4}%
\special{pa 796 1000}%
\special{pa 796 1600}%
\special{fp}%
%
\special{pn 4}%
\special{pa 2004 1000}%
\special{pa 2004 1600}%
\special{fp}%
%
\put(22.0000,-13.0000){\makebox(0,0){\(\tilde{D}'\)}}%
%
\special{pn 13}%
\special{ar 1400 1600 600 200  0.0000000 3.1415927}%
%
\special{pn 8}%
\special{ar 1400 1600 600 200  3.1415927 3.1715927}%
\special{ar 1400 1600 600 200  3.2615927 3.2915927}%
\special{ar 1400 1600 600 200  3.3815927 3.4115927}%
\special{ar 1400 1600 600 200  3.5015927 3.5315927}%
\special{ar 1400 1600 600 200  3.6215927 3.6515927}%
\special{ar 1400 1600 600 200  3.7415927 3.7715927}%
\special{ar 1400 1600 600 200  3.8615927 3.8915927}%
\special{ar 1400 1600 600 200  3.9815927 4.0115927}%
\special{ar 1400 1600 600 200  4.1015927 4.1315927}%
\special{ar 1400 1600 600 200  4.2215927 4.2515927}%
\special{ar 1400 1600 600 200  4.3415927 4.3715927}%
\special{ar 1400 1600 600 200  4.4615927 4.4915927}%
\special{ar 1400 1600 600 200  4.5815927 4.6115927}%
\special{ar 1400 1600 600 200  4.7015927 4.7315927}%
\special{ar 1400 1600 600 200  4.8215927 4.8515927}%
\special{ar 1400 1600 600 200  4.9415927 4.9715927}%
\special{ar 1400 1600 600 200  5.0615927 5.0915927}%
\special{ar 1400 1600 600 200  5.1815927 5.2115927}%
\special{ar 1400 1600 600 200  5.3015927 5.3315927}%
\special{ar 1400 1600 600 200  5.4215927 5.4515927}%
\special{ar 1400 1600 600 200  5.5415927 5.5715927}%
\special{ar 1400 1600 600 200  5.6615927 5.6915927}%
\special{ar 1400 1600 600 200  5.7815927 5.8115927}%
\special{ar 1400 1600 600 200  5.9015927 5.9315927}%
\special{ar 1400 1600 600 200  6.0215927 6.0515927}%
\special{ar 1400 1600 600 200  6.1415927 6.1715927}%
\special{ar 1400 1600 600 200  6.2615927 6.2831853}%
%
\put(22.0000,-16.0000){\makebox(0,0){\(D\)}}%
\end{picture}%
\end{center}
\caption{Attaching or removing a trivial diagram on a disk.}
\label{FIG:spherical surgery}
\end{figure}

If \(D = \partial \tilde{D}\) holds for some diagram \(\tilde{D}\) 
on \(\disk^2\), by removing from \(\tilde{D}\)
a trivial diagram \(D'\) on a disk, 
we obtain a diagram \(\tilde{D}'\) on an annulus 
whose boundary is \(D \cup (- \partial D')\).
By Lemma \ref{LEM:cobordism 1-diagram}, 
\(\Pi(D)\) is conjugate 
with \(\Pi(\partial D') = e\),
that is, \(\Pi(D) = e\).
\end{proof}

\subsection{Shadow colourability of \(2\)-diagrams.}
\label{SEC:2-diagram}

We consider here only the case that \(2\)-diagrams 
are on \(\disk^2\) or on \(\sphere^2\), 
for both \(\disk^2\) and \(\sphere^2\) 
have trivial fundamental groups. 
\begin{figure}[ht]
\begin{center}
\unitlength 0.1in%
\begin{picture}(40.0800,30.0000)(3.9200,-32.7000)%
%
\special{pn 8}%
\special{pa 400 400}%
\special{pa 2200 400}%
\special{pa 2200 1400}%
\special{pa 400 1400}%
\special{pa 400 400}%
\special{dt 0.045}%
%
\special{pn 8}%
\special{pa 500 400}%
\special{pa 500 1400}%
\special{fp}%
%
\put(6.0000,-9.0000){\makebox(0,0){\(\cdots\)}}%
%
\special{pn 8}%
\special{pa 700 400}%
\special{pa 700 1400}%
\special{fp}%
%
\special{pn 8}%
\special{pa 1900 400}%
\special{pa 1900 1400}%
\special{fp}%
%
\put(20.0000,-9.0000){\makebox(0,0){\(\cdots\)}}%
%
\special{pn 8}%
\special{pa 2100 400}%
\special{pa 2100 1400}%
\special{fp}%
%
\special{pn 8}%
\special{pa 1000 1400}%
\special{pa 1000 1000}%
\special{fp}%
%
\special{pn 8}%
\special{ar 1300 1000 300 200  3.1415927 6.2831853}%
%
\special{pn 8}%
\special{pa 1600 1000}%
\special{pa 1600 1400}%
\special{fp}%
%
\put(17.0000,-6){\makebox(0,0){\(\circlearrowright\)}}%
%
\put(13.0000,-7.5000){\makebox(0,0){\(x\)}}%
%
\put(10,-14.8){\makebox(0,0){\(\epsilon\)}}%
%
\put(15.7,-14.8){\makebox(0,0){\(- \epsilon\)}}%
%
\put(11.0000,-6.0000){\makebox(0,0){\fbox{\(\alpha\)}}}%
%
\put(13.0000,-12.0000){\makebox(0,0){\fbox{\(\alpha \lhd^\epsilon x\)}}}%
%
%
\special{pn 8}%
\special{pa 2600 400}%
\special{pa 4400 400}%
\special{pa 4400 1400}%
\special{pa 2600 1400}%
\special{pa 2600 400}%
\special{dt 0.045}%
%
\special{pn 8}%
\special{pa 2700 400}%
\special{pa 2700 1400}%
\special{fp}%
%
\put(28.0000,-9.0000){\makebox(0,0){\(\cdots\)}}%
%
\special{pn 8}%
\special{pa 2900 400}%
\special{pa 2900 1400}%
\special{fp}%
%
\special{pn 8}%
\special{pa 4100 400}%
\special{pa 4100 1400}%
\special{fp}%
%
\put(42.0000,-9.0000){\makebox(0,0){\(\cdots\)}}%
%
\special{pn 8}%
\special{pa 4300 400}%
\special{pa 4300 1400}%
\special{fp}%
%
\special{pn 8}%
\special{pa 3200 400}%
\special{pa 3200 800}%
\special{fp}%
%
\special{pn 8}%
\special{ar 3500 800 300 200  0.0000000 3.1415927}%
%
\special{pn 8}%
\special{pa 3800 800}%
\special{pa 3800 400}%
\special{fp}%
%
\put(35.0000,-10.5000){\makebox(0,0){\(x\)}}%
%
\put(38,-3.4){\makebox(0,0){\(\epsilon\)}}%
%
\put(31.7,-3.4){\makebox(0,0){\(- \epsilon\)}}%
%
\put(31.0000,-12.0000){\makebox(0,0){\fbox{\(\alpha\)}}}%
%
\put(35.0000,-6.0000){\makebox(0,0){\fbox{\(\alpha \lhd^\epsilon x\)}}}%
%
%
\special{pn 8}%
\special{pa 400 1800}%
\special{pa 2200 1800}%
\special{pa 2200 2800}%
\special{pa 400 2800}%
\special{pa 400 1800}%
\special{dt 0.045}%
%
\special{pn 8}%
\special{pa 500 1800}%
\special{pa 500 2800}%
\special{fp}%
%
\put(6,-23.0000){\makebox(0,0){\(\cdots\)}}%
%
\special{pn 8}%
\special{pa 700 1800}%
\special{pa 700 2800}%
\special{fp}%
%
\special{pn 8}%
\special{pa 1900 1800}%
\special{pa 1900 2800}%
\special{fp}%
%
\put(20,-23.0000){\makebox(0,0){\(\cdots\)}}%
%
\special{pn 8}%
\special{pa 2100 1800}%
\special{pa 2100 2800}%
\special{fp}%
%
\special{pn 8}%
\special{pa 1000 1800}%
\special{pa 1000 2200}%
\special{pa 1600 2400}%
\special{pa 1600 2800}%
\special{fp}%
%
\put(9.4,-22){\makebox(0,0){\(x\)}}%
%
\put(9.7,-17.2){\makebox(0,0){\(- \epsilon\)}}%
%
\put(16,-28.8){\makebox(0,0){\(\epsilon\)}}%
%
\special{pn 8}%
\special{sh 0}%
\special{ia 1300 2300 50 50  0.0000000 6.2831853}%
%
\special{pn 8}%
\special{pa 1600 1800}%
\special{pa 1600 2200}%
\special{pa 1000 2400}%
\special{pa 1000 2800}%
\special{fp}%
%
\put(9.4,-25){\makebox(0,0){\(y\)}}%
%
\put(15.7,-17.2){\makebox(0,0){\(- \eta\)}}%
%
\put(10,-28.8){\makebox(0,0){\(\eta\)}}%
%
\put(8.5,-20.0000){\makebox(0,0){\fbox{\(\alpha\)}}}%
%
\put(13.0000,-20.0000){\makebox(0,0){\fbox{\(\alpha \lhd^\epsilon x\)}}}%
%
\put(13.0000,-26.0000){\makebox(0,0){\fbox{\(\alpha \lhd^\eta y\)}}}%
%
\put(16.5000,-31.5){\makebox(0,0){%
\fbox{\((\alpha \lhd^\epsilon x) \lhd^\eta y\)}}}%
%
\special{pn 8}%
\special{pa 1650 3030}%
\special{pa 1750 2900}%
\special{pa 1745 2700}%
\special{pa 1720 2600}%
\special{sp}%
%
%
\special{pn 8}%
\special{pa 2600 1800}%
\special{pa 4400 1800}%
\special{pa 4400 2800}%
\special{pa 2600 2800}%
\special{pa 2600 1800}%
\special{dt 0.045}%
%
\special{pn 8}%
\special{pa 2700 1800}%
\special{pa 2700 2800}%
\special{fp}%
%
\put(28,-23.0000){\makebox(0,0){\(\cdots\)}}%
%
\special{pn 8}%
\special{pa 2900 1800}%
\special{pa 2900 2800}%
\special{fp}%
%
\special{pn 8}%
\special{pa 4100 1800}%
\special{pa 4100 2800}%
\special{fp}%
%
\put(42,-23.0000){\makebox(0,0){\(\cdots\)}}%
%
\special{pn 8}%
\special{pa 4300 1800}%
\special{pa 4300 2800}%
\special{fp}%
%
\special{pn 8}%
\special{pa 3200 2800}%
\special{pa 3200 2400}%
\special{pa 3800 2200}%
\special{pa 3800 1800}%
\special{fp}%
%
\put(31.4,-25){\makebox(0,0){\(y\)}}%
%
\put(37.7,-17.2){\makebox(0,0){\(- \eta\)}}%
%
\put(32,-28.8){\makebox(0,0){\(\eta\)}}%
%
\special{pn 8}%
\special{sh 0}%
\special{ia 3500 2300 50 50  0.0000000 6.2831853}%
%
\special{pn 8}%
\special{pa 3800 2800}%
\special{pa 3800 2400}%
\special{pa 3200 2200}%
\special{pa 3200 1800}%
\special{fp}%
%
\put(31.4,-22){\makebox(0,0){\(x\)}}%
%
\put(31.7,-17.2){\makebox(0,0){\(- \epsilon\)}}%
%
\put(38,-28.8){\makebox(0,0){\(\epsilon\)}}%
%
\put(30.5,-20){\makebox(0,0){\fbox{\(\alpha\)}}}%
%
\put(35.0000,-20.0000){\makebox(0,0){\fbox{\(\alpha \lhd^\epsilon x\)}}}%
%
\put(35.0000,-26.0000){\makebox(0,0){\fbox{\(\alpha \lhd^\eta y\)}}}%
%
\put(38.5000,-31.5){\makebox(0,0){%
\fbox{\((\alpha \lhd^\eta y) \lhd^\epsilon x\)}}}%
%
\special{pn 8}%
\special{pa 3850 3030}%
\special{pa 3950 2900}%
\special{pa 3945 2700}%
\special{pa 3920 2600}%
\special{sp}%
\end{picture}%
\end{center}
\caption{Four types of subdiagrams.}
\label{FIG:tangle diagram}
\end{figure}
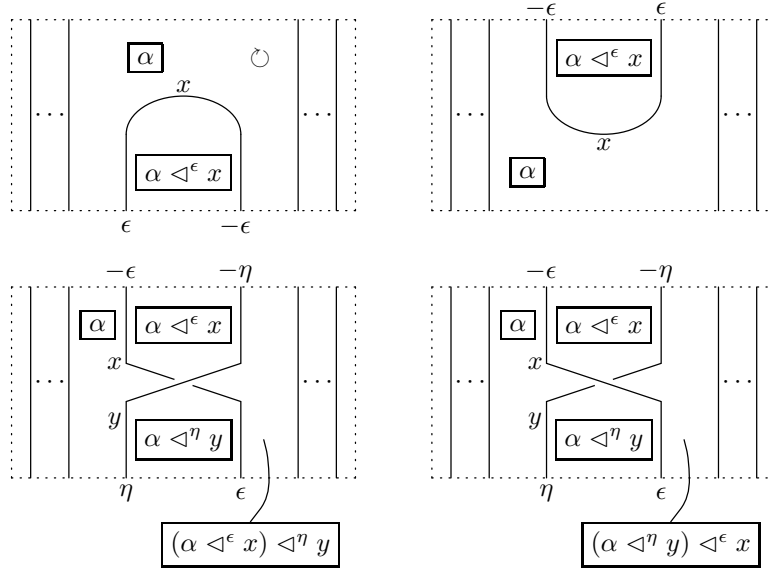
Shadow colourability of \(2\)-diagrams on a manifold \(M\) 
is closely related to representations 
of \(\pi_1(M)\) on \(Q\). \\

Let \(D\) be a \(2\)-diagram on a square \(I^2 \simeq \disk^2\). 
We suppose that \(\partial D\) is a set of points 
on \(I \times \{\pm 1\}\). 
By the assumption, 
we can decompose \(D\) into subdiagrams 
as drawn in Figure \ref{FIG:tangle diagram}. 
As for these subdiagrams, 
if the shadow colour of the left-end region is given, 
we can decide the colour of each region 
from left to right order. 
Since the diagram \(D\) is obtained 
by arranging subdiagrams from up to down,
we have: 

\begin{lemma}
\label{LEM:disk 2-diagram}
If \(D\) is a \(Q\)-coloured \(2\)-diagram on \(\disk^2\), 
\(D\) is freely \(Q\)-shadow colourable. 
\end{lemma}

For a diagram \(D\) on \(\sphere^2\), 
by removing a sufficiently small disk from \(\sphere^2\), 
we can regard \(D\) as a diagram on \(\disk^2\). 
Therefore, directly from Lemma \ref{LEM:disk 2-diagram}, 
we obtain: 

\begin{corollary}
\label{COR:sphere 2-diagram}
If \(D\) is a \(Q\)-coloured \(2\)-diagram on \(\sphere^2\), 
\(D\) is freely \(Q\)-shadow colourable. 
\end{corollary}

As shown in the construction 
of \(Q\)-shadow colouring of \(D\), 
the whole shadow colouring of \(D\) is uniquely determined 
when a shadow colour \(\alpha\) to the base-region is given. 
We denote this \(Q\)-shadow coloured diagram by \(D_\alpha\) 
as in \S \ref{SEC:1-diagram}.

\begin{lemma}
\label{LEM:cobordism 2-diagram}
Let \(D\) be a \(Q\)-coloured \(2\)-diagram on \(\sphere^2\),
and let \(\alpha\) and \(\beta\) be connected elements of \(Q\).
Two shadow coloured diagrams \(D_\alpha\) and \(D_\beta\) 
represent rack homologous \(3\)-cycles.
\end{lemma}

\begin{proof}
Since \(\alpha\) and \(\beta\) are connected,
there exists a sequence \((x_1, \epsilon_1)\),
\ldots, \((x_n, \epsilon_n)\) such that
\(( \cdots (\alpha \lhd^{\epsilon_1} x_1) \cdots ) \lhd^{\epsilon_n} x_n
= \beta\).
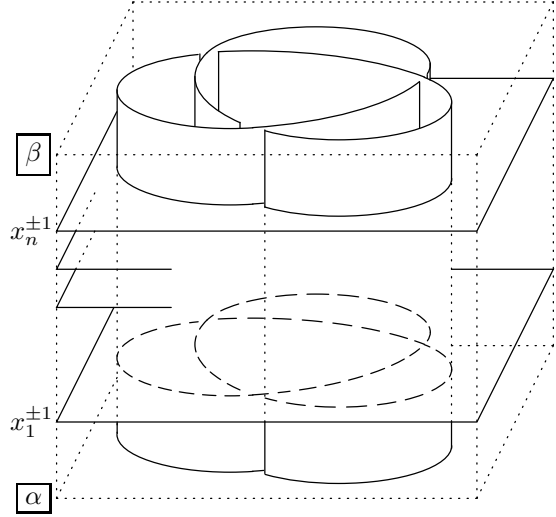
\begin{figure}[ht]
\begin{center}
\unitlength 0.1in%
\begin{picture}(28.6000, 28.0000)(5.5000,-30.0000)%
%
%
\special{pn 8}%
\special{pa 1200 300}%
\special{pa 3400 300}%
\special{pa 3000 1100}%
\special{pa 800 1100}%
\special{pa 1200 300}%
\special{dt 0.045}%
%
\put(6.8,-11.0000){\makebox(0,0){\fbox{\(\beta\)}}}%
%
\special{pn 8}%
\special{pa 1200 300}%
\special{pa 1200 640}%
\special{dt 0.045}%
\special{pn 8}%
\special{pa 3400 300}%
\special{pa 3400 2100}%
\special{dt 0.045}%
\special{pn 8}%
\special{pa 3000 1100}%
\special{pa 3000 2900}%
\special{dt 0.045}%
\special{pn 8}%
\special{pa 800 1100}%
\special{pa 800 2900}%
\special{dt 0.045}%
%
\special{pn 8}%
\special{pa 2867 2100}%
\special{pa 3400 2100}%
\special{pa 3000 2900}%
\special{pa 800 2900}%
\special{pa 1117 2266}%
\special{dt 0.045}%
%
\put(6.8,-29.0000){\makebox(0,0){\fbox{\(\alpha\)}}}%
%
\special{pn 8}%
\special{pa 2755 700}%
\special{pa 3400 700}%
\special{pa 3000 1500}%
\special{pa 800 1500}%
\special{pa 1115 875}%
\special{fp}%
%
\put(6.7000,-15.0000){\makebox(0,0){\(x_n^{\pm 1}\)}}%
%
\special{pn 8}%
\special{pa 1000 1300}%
\special{pa 900 1500}%
\special{dt 0.045}%
\special{pn 8}%
\special{pa 900 1500}%
\special{pa 800 1700}%
\special{pa 1400 1700}%
\special{fp}%
%
\special{pn 8}%
\special{pa 1000 1500}%
\special{pa 900 1700}%
\special{dt 0.045}%
\special{pn 8}%
\special{pa 900 1700}%
\special{pa 800 1900}%
\special{pa 1400 1900}%
\special{fp}%
%
\special{pn 8}%
\special{pa 2867 1700}%
\special{pa 3400 1700}%
\special{pa 3000 2500}%
\special{pa 800 2500}%
\special{pa 1100 1900}%
\special{fp}%
%
\put(6.7000,-25.0000){\makebox(0,0){\(x_1^{\pm 1}\)}}%
%
%
\special{pn 8}%
\special{ar 1588  566 1200 400  0.3859699 1.3717577}%
\special{ar 1708  765  592 198  1.3664357 4.4263729}%
\special{ar 1824  956 1196 400  4.5631845 5.5659307}%
\special{ar 2276  824  592 198  5.5707115 6.2831853}%
\special{ar 2276  824  592 198  0.0000000 2.2780132}%
\special{ar 2724  694 1200 400  2.5037074 3.4675626}%
\special{ar 2156  630  600 200  3.4664978 6.2831853}%
\special{ar 2156  630  600 200  0.0000000 0.2259699}%
%
%
\special{pn 8}%
\special{ar 1588  966 1200 400  1.3217577 1.3717577}%
\special{ar 1708 1165  592 198  1.3664357 3.1415926}%
\special{ar 2276 1224  592 198  0.0000000 2.2780132}%
%
%
\special{pn 8}%
\special{ar 1588 1966 1200 400  0.3859699 0.4859699}%
\special{ar 1588 1966 1200 400  0.5359699 0.6359699}%
\special{ar 1588 1966 1200 400  0.6859699 0.7759699}%
\special{ar 1588 1966 1200 400  0.8259699 0.9059699}%
\special{ar 1588 1966 1200 400  0.9509699 1.0309699}%
\special{ar 1588 1966 1200 400  1.0759699 1.1559699}%
\special{ar 1588 1966 1200 400  1.2009699 1.2559699}%
\special{ar 1588 1966 1200 400  1.3009699 1.3717577}%
\special{ar 1708 2165 592 198  1.3664357 1.4064357}%
\special{ar 1708 2165 592 198  1.4864357 1.6164357}%
\special{ar 1708 2165 592 198  1.6964357 1.8164357}%
\special{ar 1708 2165 592 198  1.8964357 2.0364357}%
\special{ar 1708 2165 592 198  2.1164357 2.2564357}%
\special{ar 1708 2165 592 198  2.3464357 2.5064357}%
\special{ar 1708 2165 592 198  2.6164357 2.7964357}%
\special{ar 1708 2165 592 198  2.9264357 3.3064357}%
\special{ar 1708 2165 592 198  3.4564357 3.6864357}%
\special{ar 1708 2165 592 198  3.8064357 4.0064357}%
\special{ar 1708 2165 592 198  4.1064357 4.2564357}%
\special{ar 1708 2165 592 198  4.3464357 4.4263729}%
\special{ar 1824 2356 1196 400  4.5631845 4.6431845}%
\special{ar 1824 2356 1196 400  4.6931845 4.7621845}%
\special{ar 1824 2356 1196 400  4.8031845 4.8831845}%
\special{ar 1824 2356 1196 400  4.9331845 5.0231845}%
\special{ar 1824 2356 1196 400  5.0731845 5.1631845}%
\special{ar 1824 2356 1196 400  5.2131845 5.3031845}%
\special{ar 1824 2356 1196 400  5.3591845 5.4591845}%
\special{ar 1824 2356 1196 400  5.5033726 5.5659307}%
\special{ar 2276 2224 592 198  5.5707115 5.6507115}%
\special{ar 2276 2224 592 198  5.7507115 5.9507115}%
\special{ar 2276 2224 592 198  6.1007115 6.2831853}%
\special{ar 2276 2224 592 198  0.0000000 0.2000000}%
\special{ar 2276 2224 592 198  0.3500000 0.6000000}%
\special{ar 2276 2224 592 198  0.7300000 0.9400000}%
\special{ar 2276 2224 592 198  1.0400000 1.2200000}%
\special{ar 2276 2224 592 198  1.3200000 1.5000000}%
\special{ar 2276 2224 592 198  1.6000000 1.7800000}%
\special{ar 2276 2224 592 198  1.8700000 2.0500000}%
\special{ar 2276 2224 592 198  2.1500000 2.2780132}%
\special{ar 2724 2094 1200 400  2.5037074 2.6037074}%
\special{ar 2724 2094 1200 400  2.6637074 2.7737074}%
\special{ar 2724 2094 1200 400  2.8337074 2.9737074}%
\special{ar 2724 2094 1200 400  3.0437074 3.2537074}%
\special{ar 2724 2094 1200 400  3.3437074 3.4675626}%
\special{ar 2156 2030  600 200  3.4664978 3.6064978}%
\special{ar 2156 2030  600 200  3.7364978 3.9764978}%
\special{ar 2156 2030  600 200  4.0664978 4.2664978}%
\special{ar 2156 2030  600 200  4.3564978 4.5464978}%
\special{ar 2156 2030  600 200  4.6364978 4.8264978}%
\special{ar 2156 2030  600 200  4.9064978 5.0964978}%
\special{ar 2156 2030  600 200  5.1864978 5.3864978}%
\special{ar 2156 2030  600 200  5.4864978 5.6964978}%
\special{ar 2156 2030  600 200  5.7964978 6.0164978}%
\special{ar 2156 2030  600 200  6.1372941 6.2831853}%
\special{ar 2156 2030  600 200  0.0000000 0.2259699}%
%
%
\special{pn 8}%
\special{ar 1588 2366 1200 400  1.3217577 1.3717577}%
\special{ar 1708 2565  592 198  1.3664357 3.1415926}%
\special{ar 2276 2624  592 198  0.0000000 2.2780132}%
%
%
\special{pn 8}%
\special{pa 1117 760}%
\special{pa 1117 1160}%
\special{fp}%
\special{pn 8}%
\special{pa 1117 1160}%
\special{pa 1117 2560}%
\special{dt 0.045}%
\special{pn 8}%
\special{pa 1117 2500}%
\special{pa 1117 2560}%
\special{fp}%
\special{pn 8}%
\special{pa 1524 690}%
\special{pa 1524 950}%
\special{fp}%
\special{pn 8}%
\special{pa 1550  572}%
\special{pa 1550  605}%
\special{fp}%
\special{pn 8}%
\special{pa 1648 560}%
\special{pa 1648 870}%
\special{fp}%
\special{pn 8}%
\special{pa 1760  930}%
\special{pa 1760  960}%
\special{fp}%
\special{pn 8}%
\special{pa 1890 970}%
\special{pa 1890 1370}%
\special{fp}%
\special{pn 8}%
\special{pa 1890 1370}%
\special{pa 1890 2500}%
\special{dt 0.045}%
\special{pn 8}%
\special{pa 1890 2500}%
\special{pa 1890 2770}%
\special{fp}%
\special{pn 8}%
\special{pa 2700 720}%
\special{pa 2700 960}%
\special{fp}%
\special{pn 8}%
\special{pa 2740  680}%
\special{pa 2740  700}%
\special{fp}%
\special{pn 8}%
\special{pa 2755  620}%
\special{pa 2755  710}%
\special{fp}%
\special{pn 8}%
\special{pa 2867 830}%
\special{pa 2867 1240}%
\special{fp}%
\special{pn 8}%
\special{pa 2867 1250}%
\special{pa 2867 2500}%
\special{dt 0.045}%
\special{pn 8}%
\special{pa 2867 2500}%
\special{pa 2867 2640}%
\special{fp}%
%
%
\end{picture}%
\end{center}
\caption{Diagram \(\tilde{D}\).}
\label{FIG:cobordism}
\end{figure}
We prove this lemma by constructing 
a \(Q\)-shadow coloured \(3\)-diagram \(\tilde{D}\) 
on \(\sphere^2 \times I\) such that \(\partial \tilde{D} 
= D_\alpha \cup (- D_\beta)\) holds. 
The diagram \(\tilde{D}\) consists 
of \(D \times I\) and \(n\) parallel sheets 
\(S_i = \sphere^2 \times p_i\), 
where \(- 1 < p_1 < \cdots < p_n < 1\) 
(See Figure \ref{FIG:cobordism}). 
Suppose that every sheet \(S_i\) is in the lowest level. 
The intersection of \(S_i\) with \(D \times I\) is 
a diagram on \(S_i \simeq \sphere^2\), 
which is equivalent to \(D\) (See Figure \ref{FIG:contents}). 
By Corollary \ref{COR:sphere 2-diagram}, 
we can regard each sheet \(S_i\) as a \(Q\)-shadow coloured 
diagram with the shadow colour of the base-region being \(x_i\). 
Also we assume that \(S_i\) is oriented as its normal vector 
is in the same direction as that of \(\{\ast\} \times I\) 
when the signature \(\epsilon_i\) is \(+ 1\), 
or in the opposite when \(\epsilon_i\) is \(- 1\). 

On these conditions, 
the regions between two sheets \(S_i\) and \(S_{i + 1}\) 
can be shadow coloured with \(Q\) by the similar argument 
in Corollary \ref{COR:sphere 2-diagram} 
(See the right one of Figure \ref{FIG:contents}). 
So the whole diagram \(\tilde{D}\) is \(Q\)-shadow colourable. 

Clearly, when we colour the regions in \(\sphere \times [- 1, p_1)\) 
by \(Q\) so that \(\sphere \times \{- 1\}\) becomes \(D_\alpha\), 
the whole diagram \(\tilde{D}\) has \(D_\alpha \cup (- D_\beta)\) 
as its boundary. 
Therefore, \(\langle D_\alpha \rangle - \langle D_\beta \rangle 
= \partial \langle \tilde{D} \rangle\) holds. 
Since \(\langle D_\alpha \rangle\) and \(\langle D_\beta \rangle\) 
are rack \(3\)-cycles, 
it completes the proof. 
\end{proof}
\begin{figure}[ht]
\begin{center}
\unitlength 0.1in%
\begin{picture}(45.5000, 17.0000)(21.5000,-18.7000)%
%
%
\special{pn 8}%
\special{pa 2160 280}%
\special{pa 3600 280}%
\special{pa 3600 1400}%
\special{pa 2160 1400}%
\special{pa 2160 280}%
\special{fp}%
%
\special{pn 8}%
\special{ar 2580 700 640 640  0.1000000 1.0454295}%
\special{ar 2740 977 320 320  1.0390723 4.0306649}%
\special{ar 2900 1254 640 640  4.2870221 5.2377559}%
\special{ar 3060 977 320 320  5.2441130 6.2831853}%
\special{ar 3060 977 320 320  0.0000000 1.9525204}%
\special{ar 3220 700 640 640  2.1961632 3.1415927}%
\special{ar 2900 700 320 320  3.1415927 6.0831853}%
\put(23.3,-12.5){\makebox(0,0){\fbox{\(x_i\)}}}%
\put(29,-15.5){\makebox(0,0){Sheet \(S_i\)}}%
%
%
\special{pn 8}%
\special{pa 4200 300}%
\special{pa 6400 300}%
\special{pa 6000 1100}%
\special{pa 3800 1100}%
\special{pa 4200 300}%
\special{dt 0.045}%
\put(65.8000,-3.0000){\makebox(0,0){\(S_{i + 1}\)}}%
%
\special{pn 8}%
\special{pa 4200 300}%
\special{pa 4200 640}%
\special{dt 0.045}%
\special{pn 8}%
\special{pa 6400 300}%
\special{pa 6400 700}%
\special{dt 0.045}%
\special{pn 8}%
\special{pa 6000 1100}%
\special{pa 6000 1500}%
\special{dt 0.045}%
\special{pn 8}%
\special{pa 3800 1100}%
\special{pa 3800 1500}%
\special{dt 0.045}%
%
\special{pn 8}%
\special{pa 5755 700}%
\special{pa 6400 700}%
\special{pa 6000 1500}%
\special{pa 3800 1500}%
\special{pa 4110 880}%
\special{dt 0.045}%
\put(65,-7.0000){\makebox(0,0){\(S_i\)}}%
%
%
\special{pn 8}%
\special{ar 4588  566 1200 400  0.3859699 0.8755119}%
\special{ar 4588  566 1200 400  0.8755119 1.3717577}%
\special{ar 4708  765  592 198  1.3664357 3.1415926}%
\special{ar 4708  765  592 198  3.1415926 4.4263729}%
\special{ar 4824  956 1196 400  4.5631845 4.7123889}%
\special{ar 4824  956 1196 400  4.7123889 5.5659307}%
\special{ar 5276  824  592 198  5.5707115 6.2831853}%
\special{ar 5276  824  592 198  0.0000000 2.2780132}%
\special{ar 5724  694 1200 400  2.5037074 3.1415926}%
\special{ar 5724  694 1200 400  3.1415926 3.4675626}%
\special{ar 5156  630  600 200  3.4664978 6.2831853}%
\special{ar 5156  630  600 200  0.0000000 0.2259699}%
%
%
\special{pn 8}%
\special{pa 4117  760}%
\special{pa 4117 1160}%
\special{fp}%
\special{pn 8}%
\special{pa 4524  690}%
\special{pa 4524  950}%
\special{fp}%
\special{pn 8}%
\special{pa 4550  572}%
\special{pa 4550  605}%
\special{fp}%
\special{pn 8}%
\special{pa 4648  560}%
\special{pa 4648  870}%
\special{fp}%
\special{pn 8}%
\special{pa 4760  930}%
\special{pa 4760  960}%
\special{fp}%
\special{pn 8}%
\special{pa 4890  970}%
\special{pa 4890 1370}%
\special{fp}%
\special{pn 8}%
\special{pa 5700  720}%
\special{pa 5700  960}%
\special{fp}%
\special{pn 8}%
\special{pa 5740  680}%
\special{pa 5740  700}%
\special{fp}%
\special{pn 8}%
\special{pa 5755  620}%
\special{pa 5755  710}%
\special{fp}%
\special{pn 8}%
\special{pa 5867  830}%
\special{pa 5867 1240}%
\special{fp}%
%
%
\special{pn 8}%
\special{ar 4588  966 1200 400  1.3217577 1.3717577}%
\special{ar 4708 1165  592 198  1.3664357 3.1415926}%
\special{ar 5276 1224  592 198  0.0000000 2.2780132}%
%
%
\special{pn 8}%
\special{pa 4670 1620}%
\special{pa 4640 1616}%
\special{pa 4608 1610}%
\special{pa 4578 1604}%
\special{pa 4548 1594}%
\special{pa 4516 1584}%
\special{pa 4486 1570}%
\special{pa 4456 1552}%
\special{pa 4434 1528}%
\special{pa 4430 1500}%
\special{pa 4446 1468}%
\special{pa 4466 1442}%
\special{pa 4480 1430}%
\special{sp}%
\put(49.0000,-17.3200){\makebox(0,0){%
\fbox{\(( \cdots (\alpha \lhd^{\epsilon_1} x_1) %
\cdots ) \lhd^{\epsilon_i} x_i\)}}}%
%
%
\end{picture}%
\end{center}
\caption{Sheet \(S_i\) and regions between \(S_i\) and \(S_{i + 1}\).}
\label{FIG:contents}
\end{figure}
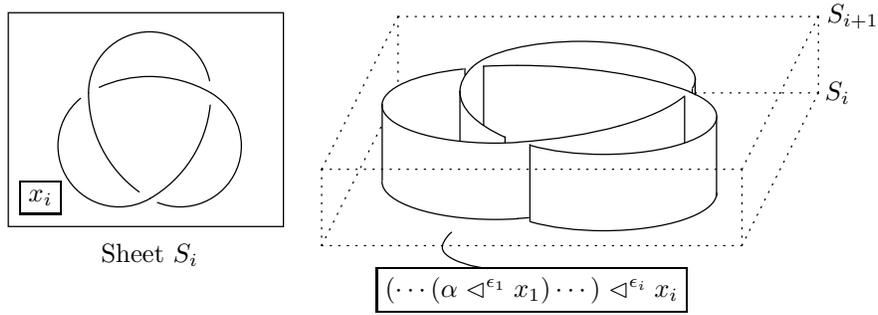

Directly from Lemma \ref{LEM:cobordism 2-diagram},
we have:

\begin{corollary}
\label{COR:diagram class}
When \(Q\) is connected, 
a \(Q\)-coloured \(2\)-diagram \(D\) on \(\sphere^2\) 
determines a unique third rack homology class of \(Q\). 
\end{corollary}

\subsection{Deformations of \(2\)-diagrams.}
\label{SEC:Reidemeister move}

In the rest of \S \ref{CHAP:lemma}, 
we consider deformations of \(2\)-diagrams 
which preserve the rack (or quandle) homology classes 
represented by the diagrams. \\

Let \(D\) be a \(2\)-diagram on a surface \(M\).
\textbf{Reidemeister deformations} 
are operations to replace a disk on \(M\) 
with a new disk. 
There are three types of Reidemeister deformations, 
called R-I, R-II and R-III.
In Figures \ref{FIG:R-II deformation}, 
\ref{FIG:R-III deformation} 
and \ref{FIG:R-I deformation}, there are drawn the disks 
that each deformation removes from and attaches to \(M\).
\begin{figure}[ht]
\begin{center}
\unitlength 0.1in%
\begin{picture}( 36.6000, 10.5000)( 1.9200,-11.5000)%
%
%
\special{pn 8}%
\special{pa 200 200}%
\special{pa 200 1100}%
\special{pa 1600 1100}%
\special{pa 1600 200}%
\special{pa 200 200}%
\special{dt 0.045}%
\special{pn 8}%
\special{pa 200 520}%
\special{pa 900 520}%
\special{pa 1100 820}%
\special{pa 1200 820}%
\special{pa 1400 520}%
\special{pa 1600 520}%
\special{fp}%
\special{pn 4}%
\special{sh 1}%
\special{pa 1600 520}%
\special{pa 1530 550}%
\special{pa 1540 520}%
\special{pa 1530 490}%
\special{pa 1600 520}%
\special{fp}%
\special{pn 8}%
\special{pa 200 840}%
\special{pa 900 840}%
\special{pa 988 708}%
\special{fp}%
\special{pa 1024 654}%
\special{pa 1100 540}%
\special{pa 1200 540}%
\special{pa 1276 654}%
\special{fp}%
\special{pa 1312 708}%
\special{pa 1400 840}%
\special{pa 1600 840}%
\special{fp}%
\put(17.5000,-7.8000){\makebox(0,0){\(x^{\pm 1}\)}}%
\put(17.0000,-5.0000){\makebox(0,0){\(y\)}}%
\put(9.0000,-9.7000){\makebox(0,0){\fbox{\(\alpha\)}}}%
\put(5.5000,-6.7500){\makebox(0,0){\fbox{\(\alpha \lhd^{\pm 1} x\)}}}%
\put(9.0000,-3.6000){\makebox(0,0){%
\fbox{\((\alpha \lhd^{\pm 1} x) \lhd y\)}}}%
%
%
%
\special{pn 8}%
\special{pa 2200 200}%
\special{pa 2200 1100}%
\special{pa 3600 1100}%
\special{pa 3600 200}%
\special{pa 2200 200}%
\special{dt 0.045}%
\special{pn 8}%
\special{pa 2200 520}%
\special{pa 3600 520}%
\special{fp}%
\special{pn 4}%
\special{sh 1}%
\special{pa 3600 520}%
\special{pa 3530 550}%
\special{pa 3540 520}%
\special{pa 3530 490}%
\special{pa 3600 520}%
\special{fp}%
\special{pn 8}%
\special{pa 2200 840}%
\special{pa 3600 840}%
\special{fp}%
\put(37.5000,-7.8000){\makebox(0,0){\(x^{\pm 1}\)}}%
\put(37.0000,-5.0000){\makebox(0,0){\(y\)}}%
\put(29.0000,-9.7000){\makebox(0,0){\fbox{\(\alpha\)}}}%
\put(29.0000,-6.7500){\makebox(0,0){\fbox{\(\alpha \lhd^{\pm 1} x\)}}}%
\put(29.0000,-3.6000){\makebox(0,0){%
\fbox{\((\alpha \lhd^{\pm 1} x) \lhd y\)}}}%
%
%
\end{picture}%
\end{center}
\caption{R-II deformation.}
\label{FIG:R-II deformation}
\end{figure}
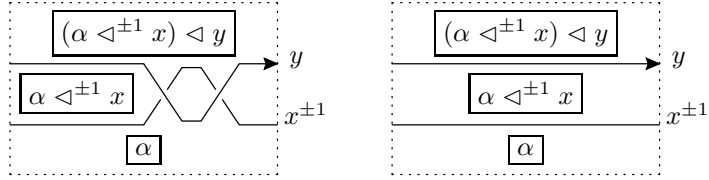
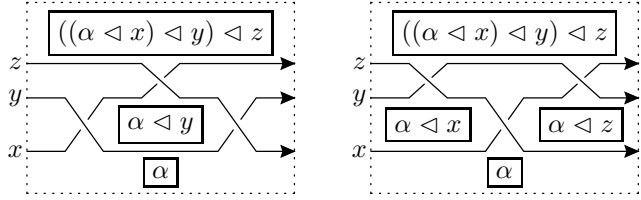
\begin{figure}[ht]
\begin{center}
\unitlength 0.1in%
\begin{picture}( 33.1000, 10.8000)( 1.0000,-12.1000)%
%
%
\special{pn 8}%
\special{pa 200 200}%
\special{pa 1600 200}%
\special{pa 1600 1200}%
\special{pa 200 1200}%
\special{pa 200 200}%
\special{dt 0.045}%
\special{pn 8}%
\special{pa 200 520}%
\special{pa 800 520}%
\special{pa 1000 700}%
\special{pa 1200 700}%
\special{pa 1400 980}%
\special{pa 1600 980}%
\special{fp}%
\special{pn 4}%
\special{sh 1}%
\special{pa 1600 980}%
\special{pa 1530 1010}%
\special{pa 1540 980}%
\special{pa 1530 950}%
\special{pa 1600 980}%
\special{fp}%
\special{pn 8}%
\special{pa 200 700}%
\special{pa 400 700}%
\special{pa 600 980}%
\special{pa 1200 980}%
\special{pa 1280 868}%
\special{fp}%
\special{pa 1320 812}%
\special{pa 1400 700}%
\special{pa 1600 700}%
\special{fp}%
\special{pn 4}%
\special{sh 1}%
\special{pa 1600 700}%
\special{pa 1530 730}%
\special{pa 1540 700}%
\special{pa 1530 670}%
\special{pa 1600 700}%
\special{fp}%
\special{pn 8}%
\special{pa 200 980}%
\special{pa 400 980}%
\special{pa 480 868}%
\special{fp}%
\special{pa 520 812}%
\special{pa 600 700}%
\special{pa 800 700}%
\special{pa 880 628}%
\special{fp}%
\special{pa 920 592}%
\special{pa 1000 520}%
\special{pa 1600 520}%
\special{fp}%
\special{pn 4}%
\special{sh 1}%
\special{pa 1600 520}%
\special{pa 1530 550}%
\special{pa 1540 520}%
\special{pa 1530 490}%
\special{pa 1600 520}%
\special{fp}%
\put(1.4000,-9.8000){\makebox(0,0){\(x\)}}%
\put(1.4000,-7.0000){\makebox(0,0){\(y\)}}%
\put(1.4000,-5.2000){\makebox(0,0){\(z\)}}%
\put(9.0000,-10.9000){\makebox(0,0){\fbox{\(\alpha\)}}}%
\put(9.0000,-8.4000){\makebox(0,0){\fbox{\(\alpha \lhd y\)}}}%
\put(9.0000,-3.6000){\makebox(0,0){%
\fbox{\(((\alpha \lhd x) \lhd y) \lhd z\)}}}%
%
%
\special{pn 8}%
\special{pa 2000 200}%
\special{pa 3400 200}%
\special{pa 3400 1200}%
\special{pa 2000 1200}%
\special{pa 2000 200}%
\special{dt 0.045}%
\special{pn 8}%
\special{pa 2000 520}%
\special{pa 2200 520}%
\special{pa 2400 700}%
\special{pa 2600 700}%
\special{pa 2800 980}%
\special{pa 3400 980}%
\special{fp}%
\special{pn 4}%
\special{sh 1}%
\special{pa 3400 980}%
\special{pa 3330 1010}%
\special{pa 3340 980}%
\special{pa 3330 950}%
\special{pa 3400 980}%
\special{fp}%
\special{pn 8}%
\special{pa 2000 700}%
\special{pa 2200 700}%
\special{pa 2280 628}%
\special{fp}%
\special{pa 2320 592}%
\special{pa 2400 520}%
\special{pa 3000 520}%
\special{pa 3200 700}%
\special{pa 3400 700}%
\special{fp}%
\special{pn 4}%
\special{sh 1}%
\special{pa 3400 700}%
\special{pa 3330 730}%
\special{pa 3340 700}%
\special{pa 3330 670}%
\special{pa 3400 700}%
\special{fp}%
\special{pn 8}%
\special{pa 2000 980}%
\special{pa 2600 980}%
\special{pa 2680 868}%
\special{fp}%
\special{pa 2720 812}%
\special{pa 2800 700}%
\special{pa 3000 700}%
\special{pa 3080 628}%
\special{fp}%
\special{pa 3120 592}%
\special{pa 3200 520}%
\special{pa 3400 520}%
\special{fp}%
\special{pn 4}%
\special{sh 1}%
\special{pa 3400 520}%
\special{pa 3330 550}%
\special{pa 3340 520}%
\special{pa 3330 490}%
\special{pa 3400 520}%
\special{fp}%
\put(19.4000,-9.8000){\makebox(0,0){\(x\)}}%
\put(19.4000,-7.0000){\makebox(0,0){\(y\)}}%
\put(19.4000,-5.2000){\makebox(0,0){\(z\)}}%
\put(27.0000,-10.9000){\makebox(0,0){\fbox{\(\alpha\)}}}%
\put(23.0000,-8.4000){\makebox(0,0){\fbox{\(\alpha \lhd x\)}}}%
\put(31.0000,-8.4000){\makebox(0,0){\fbox{\(\alpha \lhd z\)}}}%
\put(27.0000,-3.6000){\makebox(0,0){%
\fbox{\(((\alpha \lhd x) \lhd y) \lhd z\)}}}%
%
%
\end{picture}%
\end{center}
\caption{R-III deformation.}
\label{FIG:R-III deformation}
\end{figure}

Lemmas \ref{LEM:R-II R-III deformation} 
and \ref{LEM:R-I deformation} are easily seen 
from the figures below.
See [CKS1] for a detailed proof.

\begin{lemma}
\label{LEM:R-II R-III deformation}
Let \(D\) and \(D'\) be 
\(Q\)-(shadow) coloured \(2\)-diagrams on \(M\).
If \(D'\) can be obtained from \(D\) 
by an R-II or R-III deformation, 
then chains \(\langle D \rangle\) 
and \(\langle D' \rangle\) are rack homologous. 
\end{lemma}

On the other hand,
an R-I deformation changes the represented 
rack homology class, but it keeps 
the quandle homology class:

\begin{lemma}
\label{LEM:R-I deformation}
If \(D'\) can be obtained from \(D\) 
by an R-I deformation, 
then chains \(\langle D \rangle\) 
and \(\langle D' \rangle\) are quandle homologous. 
\end{lemma}
\begin{figure}[ht]
\begin{center}
\unitlength 0.1in%
\begin{picture}( 31.4000, 7.6000)( 1.9200,-9.0200)%
%
%
\special{pn 8}%
\special{pa 200 200}%
\special{pa 200 900}%
\special{pa 1000 900}%
\special{pa 1000 200}%
\special{pa 200 200}%
\special{dt 0.045}%
\special{pn 8}%
\special{pa 200 660}%
\special{pa 380 660}%
\special{fp}%
\special{pn 8}%
\special{ar 380 550 440 108  0.0000000 1.5707963}%
\special{ar 600 550 220 110  3.1415926 6.2831853}%
\special{ar 820 550 440 108  2.4007963 3.1415926}%
\special{ar 820 550 440 108  1.5707963 1.8415926}%
\special{pn 8}%
\special{pa 820 660}%
\special{pa 1000 660}%
\special{fp}%
\special{pn 4}%
\special{sh 1}%
\special{pa 1000 660}%
\special{pa 930 630}%
\special{pa 940 660}%
\special{pa 930 690}%
\special{pa 1000 660}%
\special{fp}%
\put(10.8000,-6.6000){\makebox(0,0){\(x\)}}%
\put(6.0000,-7.8000){\makebox(0,0){\fbox{\(\alpha\)}}}%
\put(6.0000,-3.2000){\makebox(0,0){\fbox{\(\alpha \lhd x\)}}}%
%
%
\special{pn 8}%
\special{pa 1300 200}%
\special{pa 1300 900}%
\special{pa 2100 900}%
\special{pa 2100 200}%
\special{pa 1300 200}%
\special{dt 0.045}%
\special{pn 8}%
\special{pa 1300 550}%
\special{pa 2100 550}%
\special{fp}%
\special{pn 4}%
\special{sh 1}%
\special{pa 2100 550}%
\special{pa 2030 520}%
\special{pa 2040 550}%
\special{pa 2030 580}%
\special{pa 2100 550}%
\special{fp}%
\put(21.8000,-6.6000){\makebox(0,0){\(x\)}}%
\put(17.0000,-7.2500){\makebox(0,0){\fbox{\(\alpha\)}}}%
\put(17.0000,-3.7500){\makebox(0,0){\fbox{\(\alpha \lhd x\)}}}%
%
%
\special{pn 8}%
\special{pa 2400 200}%
\special{pa 2400 900}%
\special{pa 3200 900}%
\special{pa 3200 200}%
\special{pa 2400 200}%
\special{dt 0.045}%
\special{pn 8}%
\special{pa 2400 660}%
\special{pa 2580 660}%
\special{fp}%
\special{pn 8}%
\special{ar 2580 550 440 108  0.0000000 0.7407963}%
\special{ar 2580 550 440 108  1.3000000 1.5707963}%
\special{ar 2800 550 220 110  3.1415926 6.2831853}%
\special{ar 3020 550 440 108  1.5707963 3.1415926}%
\special{pn 8}%
\special{pa 3020 660}%
\special{pa 3200 660}%
\special{fp}%
\special{pn 4}%
\special{sh 1}%
\special{pa 3200 660}%
\special{pa 3130 630}%
\special{pa 3140 660}%
\special{pa 3130 690}%
\special{pa 3200 660}%
\special{fp}%
\put(32.8000,-6.6000){\makebox(0,0){\(x\)}}%
\put(28.0000,-7.8000){\makebox(0,0){\fbox{\(\alpha\)}}}%
\put(28.0000,-3.2000){\makebox(0,0){\fbox{\(\alpha \lhd x\)}}}%
%
%
\end{picture}%
\end{center}
\caption{R-I deformation.}
\label{FIG:R-I deformation}
\end{figure}
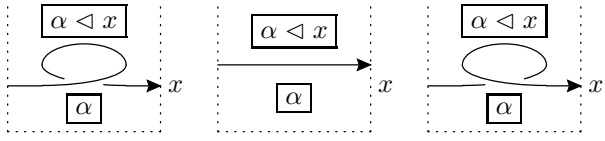

\section{Shadow diagram classes of a link.}
\label{CHAP:shadow class}

In this section,
we will construct some elements 
of the third homology groups of 
the quandle \(Q(L)\) of a link \(L\).
These homology classes are 
derived from the concepts of the diagram classes 
and of the shadow colourings,
thus we call them the shadow diagram classes of \(L\).
The construction is motivated 
by the shadow cocycle invariants in [CKS1].
We also show the relation between the shadow cocycle
invariants and the shadow diagram classes.
As an application,
we generalise the result in [S] using 
the shadow diagram classes.

\subsection{Shifting and splitting homomorphisms.}
\label{SEC:shifting splitting}

Before constructing shadow diagram classes,
we introduce two homomorphisms of rack homology groups:
the one is a rack version of 
the ``shifting homomorphism'' defined in [CJKS3],
and the another is the splitting homomorphism in [LN].\\

For a rack \(R\),
define a homomorphism \(\sigma_n\): \(C^\rck_n(R; A) 
\to C^\rck_{n - 1}(R; A)\) by linearly extending 
a map on its basis 
\begin{quote}
\((x_1, x_2, \ldots, x_n) \mapsto (x_2, \ldots, x_n)\).
\end{quote}
By easy calculation,
we can see that
\begin{quote}
\(\partial_{n - 1} \circ \sigma_n (x_1, x_2, \ldots, x_n) 
= \partial_{n - 1}(x_2, \ldots, x_n) \\
{}\hskip30pt = \sum\limits_{i = 2}^n (- 1)^{(n - 1) - (i - 1)} 
\{(x_2 \lhd x_i, \ldots, x_{i - 1} \lhd x_i, 
x_{i + 1}, \ldots, x_n) \\
{}\hskip140pt - (x_2, \ldots, x_{i - 1}, x_{i + 1}, \ldots, x_n)\} \\
{}\hskip30pt = \sigma_{n - 1} \biggl(\sum\limits_{i = 1}^n (- 1)^{n - i} 
\{(x_1 \lhd x_i, \ldots, x_{i - 1} \lhd x_i, 
x_{i + 1}, \ldots, x_n) \\
{}\hskip140pt - (x_1, \ldots, x_{i - 1}, x_{i + 1}, \ldots, x_n)\}\biggr) \\
{}\hskip30pt = \sigma_{n - 1} \circ \partial_n (x_1, x_2, \ldots, x_n)\)
\end{quote}
for each \(n\), thus \(\sigma_\ast\) is a chain map 
\(C^\rck_\ast(R; A) \to C^\rck_{\ast - 1}(R; A)\),
which induces a homomorphism 
\(H^\rck_\ast(R; A) \to H^\rck_{\ast - 1}(R; A)\). 
We call this the \textbf{shifting homomorphism}
and denote it also by \(\sigma_\ast\).\\

For a quandle \(Q\),
Litherland-Nelson constructed 
an endomorphism \(\alpha_\ast\) of \(C^\rck_\ast(Q; A)\) 
which gives a splitting of \(C^\rck_\ast(Q; A)\) into 
the direct sum of \(C^\dgn_\ast(Q; A)\) 
and \(\alpha_\ast C^\rck_\ast(Q; A)\) as 
chain complexes.
To begin with, we introduce a product 
of \(\bigoplus C^\rck_\ast(Q; A)\) 
defined on its bases as 
\begin{quote}
\((x_1, \ldots, x_n) \bullet (y_1, \ldots, y_m) 
= (x_1, \ldots, x_n, y_1, \ldots, y_m)\).
\end{quote}
Through this product, 
we can obtain a series of 
endomorphisms of \(C^\rck_n(Q; A)\) defined as 
\begin{quote}
\(\alpha_n\): \((x_1, x_2, \ldots, x_n)
\mapsto x_1 \bullet (x_2 - x_1) \bullet \cdots \bullet (x_n - x_{n - 1})\)
\end{quote}
on its basis, where, in the R.H.S., \(x_1\), \(x_2 - x_1\), \ldots,
\(x_n - x_{n - 1}\) are all considered as elements 
of \(C^\rck_1(Q; A)\).
It is proved in [LN] that \(\alpha_\ast\) 
becomes a chain map from \(C^\rck_\ast(Q; A)\) to itself. 
Through this chain map, 
we have: 
\begin{theorem}[Litherland-Nelson {[LN]}]
\label{TH:chain complex splitting}
For an arbitrary quandle \(R\), 
the rack chain complex splits into 
the direct sum of the degeneracy chain complex 
and the image of \(\alpha_\ast\), that is,
\begin{quote}
\(C^\rck_\ast(Q; A) = C^\dgn_\ast(Q; A) 
\oplus \alpha_\ast C^\rck_\ast(Q; A)\) 
\end{quote}
holds, and therefore \(C^\qdl_\ast(Q; A)\) 
is isomorphic to \(\alpha_\ast C^\rck_\ast(Q; A)\).
\end{theorem}

Now we have a projection \(H^\rck_\ast(Q; A) 
\to H^\qdl_\ast(Q; A)\), denoted also by \(\alpha_\ast\), 
and call it the \textbf{splitting homomorphism}.
Theorem \ref{TH:homology splitting} 
and Corollary \ref{COR:cohomology splitting} are 
direct consequences of this theorem.

\subsection{Two homomorphisms on diagrams.}
\label{SEC:homomorphism diagram}

At first, we give a diagrammatic translation 
of the shifting homomorphisms.
Let \(R\) be a rack.
As noticed in \S \ref{SEC:colouring}, 
every rack \(n\)-chain \(c\) of \(R\) can be represented 
by an \(R\)-shadow coloured \((n - 1)\)-diagram \(D\). 
The first factor of each basis of \(c\) 
corresponds to the shadow colour of some region,
therefore, we can see that the \((n - 1)\)-chain 
\(\sigma_n(c)\) is represented by the \(R\)-coloured 
\((n - 1)\)-diagram obtained by disregarding 
the shadow colours of \(D\). 
We denote this diagram by \(\sigma_n D\). 

Considering the inverse operation, 
that is, to give shadow colours to regions 
of a coloured diagram, 
we can make some higher homology classes 
from a lower one. 
In \S \ref{CHAP:lemma},
we have already seen sufficient conditions for a diagram 
to have an additional shadow colouring. \\

Splitting homomorphism is more complicated 
to be translated diagrammatically. 
For each cycle \(c \in C^\qdl_n(Q)\) 
of a quandle \(Q\), 
let \(\hat{c}\) be an element of \(C^\rck_n(Q)\) 
such that \(\hat{c}\) is mapped to \(c\) 
through the canonical projection \(\rho\): 
\(C^\rck_n(Q) \to C^\qdl_n(Q)\). 
Similarly as in Theorem \ref{TH:realization}, 
we have a shadow coloured diagram \(D\) 
which represents \(\hat{c}\). 
Though the diagram representing a rack cycle 
can be supposed to be on some closed manifold, 
\(D\) may be a diagram with boundary in this case. 
Since \(\partial_n(c) = 0\) holds 
in \(C^\qdl_\ast(Q) = C^\rck_\ast(Q) / C^\dgn_\ast(Q)\), 
\(\partial D\) can be supposed to represent 
the degeneracy chain \(\partial_n(\hat{c})\). 

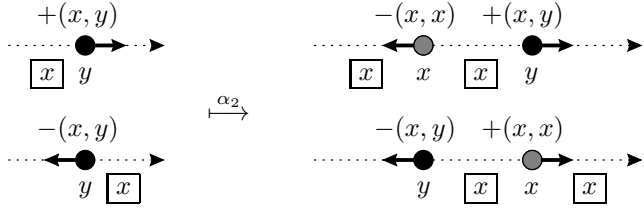
\begin{figure}[th]
\begin{center}
\unitlength 0.1in
\begin{picture}( 33.1400, 11.0000)(  1.9200, -12.3000)
%
%
\special{pn 8}%
\special{pa 200 400}%
\special{pa 1000 400}%
\special{dt 0.045}%
\special{sh 1}%
\special{pn 4}%
\special{pa 1010 400}%
\special{pa 940 370}%
\special{pa 950 400}%
\special{pa 940 430}%
\special{pa 1010 400}%
\special{fp}%
%
\special{pn 20}%
\special{pa 600 400}%
\special{pa 750 400}%
\special{fp}%
\special{sh 1}%
\special{pn 4}%
\special{pa 810 400}%
\special{pa 740 370}%
\special{pa 750 400}%
\special{pa 740 430}%
\special{pa 810 400}%
\special{fp}%
%
\special{pn 8}%
\special{sh 1}%
\special{ar 600 400 50 50  0.0000000 6.2831853}%
%
\put(5.6000,-2.4000){\makebox(0,0){\(+ (x, y)\)}}%
%
\put(4.0000,-5.5000){\makebox(0,0){\fbox{\(x\)}}}%
%
\put(6.0000,-5.5900){\makebox(0,0){\(y\)}}%
%
%
\special{pn 8}%
\special{pa 1800 400}%
\special{pa 3450 400}%
\special{dt 0.045}%
\special{sh 1}%
\special{pn 4}%
\special{pa 3510 400}%
\special{pa 3440 370}%
\special{pa 3450 400}%
\special{pa 3440 430}%
\special{pa 3510 400}%
\special{fp}%
%
\special{pn 20}%
\special{pa 2370 400}%
\special{pa 2220 400}%
\special{fp}%
\special{sh 1}%
\special{pn 4}%
\special{pa 2160 400}%
\special{pa 2230 370}%
\special{pa 2220 400}%
\special{pa 2230 430}%
\special{pa 2160 400}%
\special{fp}%
%
\special{pn 8}%
\special{sh 0}%
\special{ar 2370 400 50 50  0.0000000 6.2831853}%
\special{pn 8}%
\special{sh 0.5}%
\special{ar 2370 400 50 50  0.0000000 6.2831853}%
%
\special{pn 20}%
\special{pa 2940 400}%
\special{pa 3090 400}%
\special{fp}%
\special{sh 1}%
\special{pn 4}%
\special{pa 3150 400}%
\special{pa 3080 370}%
\special{pa 3090 400}%
\special{pa 3080 430}%
\special{pa 3150 400}%
\special{fp}%
%
\special{pn 8}%
\special{sh 1}%
\special{ar 2940 400 50 50  0.0000000 6.2831853}%
%
\put(23.2500,-2.4000){\makebox(0,0){\(- (x, x)\)}}%
%
\put(28.9000,-2.4000){\makebox(0,0){\(+ (x, y)\)}}%
%
\put(20.7000,-5.5000){\makebox(0,0){\fbox{\(x\)}}}%
%
\put(23.7000,-5.5000){\makebox(0,0){\(x\)}}%
%
\put(26.7000,-5.5000){\makebox(0,0){\fbox{\(x\)}}}%
%
\put(29.3500,-5.5900){\makebox(0,0){\(y\)}}%
%
%
\special{pn 8}%
\special{pa  200 1000}%
\special{pa  950 1000}%
\special{dt 0.045}%
\special{sh 1}%
\special{pn 4}%
\special{pa 1010 1000}%
\special{pa  940  970}%
\special{pa  950 1000}%
\special{pa  940 1030}%
\special{pa 1010 1000}%
\special{fp}%
%
\special{pn 8}%
\special{sh 1}%
\special{ar 600 1000 50 50  0.0000000 6.2831853}%
%
\special{pn 20}%
\special{pa 600 1000}%
\special{pa 450 1000}%
\special{fp}%
\special{sh 1}%
\special{pn 4}%
\special{pa 390 1000}%
\special{pa 460 970}%
\special{pa 450 1000}%
\special{pa 460 1030}%
\special{pa 390 1000}%
\special{fp}%
%
\put(5.6000,-8.4000){\makebox(0,0){\(- (x, y)\)}}%
%
\put(8.0000,-11.5000){\makebox(0,0){\fbox{\(x\)}}}%
%
\put(6.0000,-11.5900){\makebox(0,0){\(y\)}}%
%
%
\special{pn 8}%
\special{pa 1800 1000}%
\special{pa 3450 1000}%
\special{dt 0.045}%
\special{sh 1}%
\special{pn 4}%
\special{pa 3510 1000}%
\special{pa 3440 970}%
\special{pa 3450 1000}%
\special{pa 3440 1030}%
\special{pa 3510 1000}%
\special{fp}%
%
\special{pn 20}%
\special{pa 2370 1000}%
\special{pa 2220 1000}%
\special{fp}%
\special{sh 1}%
\special{pn 4}%
\special{pa 2160 1000}%
\special{pa 2230 970}%
\special{pa 2220 1000}%
\special{pa 2230 1030}%
\special{pa 2160 1000}%
\special{fp}%
%
\special{pn 8}%
\special{sh 1}%
\special{ar 2370 1000 50 50  0.0000000 6.2831853}%
%
\special{pn 20}%
\special{pa 2940 1000}%
\special{pa 3090 1000}%
\special{fp}%
\special{sh 1}%
\special{pn 4}%
\special{pa 3150 1000}%
\special{pa 3080 970}%
\special{pa 3090 1000}%
\special{pa 3080 1030}%
\special{pa 3150 1000}%
\special{fp}%
%
\special{pn 8}%
\special{sh 0}%
\special{ar 2940 1000 50 50  0.0000000 6.2831853}%
\special{pn 8}%
\special{sh 0.5}%
\special{ar 2940 1000 50 50  0.0000000 6.2831853}%
%
\put(23.2500,-8.4000){\makebox(0,0){\(- (x, y)\)}}%
%
\put(28.9000,-8.4000){\makebox(0,0){\(+ (x, x)\)}}%
%
\put(32.4000,-11.5000){\makebox(0,0){\fbox{\(x\)}}}%
%
\put(29.3500,-11.5000){\makebox(0,0){\(x\)}}%
%
\put(26.7000,-11.5000){\makebox(0,0){\fbox{\(x\)}}}%
%
\put(23.7000,-11.5900){\makebox(0,0){\(y\)}}%
%
%
\put(13.5000,-7.2000){\makebox(0,0){%
\(\stackrel{\alpha_2}{\longmapsto}\)}}%
\end{picture}%
\end{center}
\caption{Modification of shadow coloured 1-diagrams.}
\label{FIG:1-diagram modification}
\end{figure}
Our purpose here is to realise \(\alpha_\ast\) 
diagrammatically as an operation to make 
a diagram on a closed manifold representing 
a given quandle cycle. 
Obviously, from Theorem \ref{TH:homology splitting}, 
each quandle homology class can be represented 
by some rack cycle, and thus it can be represented 
by a (shadow) coloured diagram on a closed manifold. 
However, since we are using a diagram to obtain 
higher homology classes, 
it is necessary to examine that every diagram representing 
the quandle cycle can be transformed to another 
on a closed manifold. \\

We are concerned only with 
cases of \(Q\)-shadow coloured \(1\)- or \(2\)-diagrams.
In these cases,
splitting homomorphisms are explicitly 
in the following forms:
\begin{quote}
\(\alpha_2(x, y) = x \bullet (y - x) 
= (x, y) - (x, x)\),\\
\(\alpha_3(x, y, z) 
= x \bullet (y - x) \bullet (z - y) \\
\hskip43.2pt = (x, y, z) - (x, y, y) 
+ (x, x, y) - (x, x, z)\).
\end{quote}

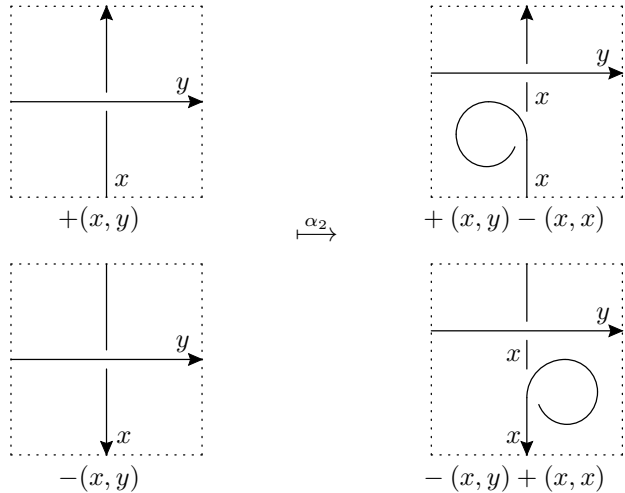
\begin{figure}[ht]
\begin{center}
\unitlength 0.1in%
\begin{picture}( 32.0800, 26.2000)(  3.9200,-29.5000)%
%
\special{pn 8}%
\special{pa 400 400}%
\special{pa 1400 400}%
\special{pa 1400 1400}%
\special{pa 400 1400}%
\special{pa 400 400}%
\special{dt 0.045}%
%
\special{pn 8}%
\special{pa 400 900}%
\special{pa 1400 900}%
\special{fp}%
\special{pn 4}%
\special{sh 1}%
\special{pa 1400 900}%
\special{pa 1330 870}%
\special{pa 1340 900}%
\special{pa 1330 930}%
\special{pa 1400 900}%
\special{fp}%
%
\special{pn 8}%
\special{pa 900 1400}%
\special{pa 900 950}%
\special{fp}%
\special{pn 8}%
\special{pa 900 850}%
\special{pa 900 400}%
\special{fp}%
\special{pn 4}%
\special{sh 1}%
\special{pa 900 400}%
\special{pa 870 470}%
\special{pa 900 460}%
\special{pa 930 470}%
\special{pa 900 400}%
\special{fp}%
%
\put(9.8000,-13.1000){\makebox(0,0){\(x\)}}%
%
\put(13.0000,-8.2000){\makebox(0,0){\(y\)}}%
%
\put(8.6000,-15.2000){\makebox(0,0){\(+ (x, y)\)}}%
%
%
\special{pn 8}%
\special{pa 2600 400}%
\special{pa 3600 400}%
\special{pa 3600 1400}%
\special{pa 2600 1400}%
\special{pa 2600 400}%
\special{dt 0.045}%
%
\special{pn 8}%
\special{pa 2600 750}%
\special{pa 3600 750}%
\special{fp}%
\special{pn 4}%
\special{sh 1}%
\special{pa 3600 750}%
\special{pa 3530 720}%
\special{pa 3540 750}%
\special{pa 3530 780}%
\special{pa 3600 750}%
\special{fp}%
%
\special{pn 8}%
\special{pa 3100 1400}%
\special{pa 3100 1100}%
\special{fp}%
\special{pn 8}%
\special{ar 2900 1100 200 200  4.7123890 6.2831853}%
\special{pn 8}%
\special{ar 2900 1070 170 170  3.1415927 4.7123890}%
\special{pn 8}%
\special{ar 2890 1070 160 170  0.4000000 3.1415927}%
\special{pn 8}%
\special{pa 3100 950}%
\special{pa 3100 800}%
\special{fp}%
\special{pn 8}%
\special{pa 3100 700}%
\special{pa 3100 400}%
\special{fp}%
\special{pn 4}%
\special{sh 1}%
\special{pa 3100 400}%
\special{pa 3070 470}%
\special{pa 3100 460}%
\special{pa 3130 470}%
\special{pa 3100 400}%
\special{fp}%
%
\put(31.8000,-13.1000){\makebox(0,0){\(x\)}}%
%
\put(31.8000,-8.8000){\makebox(0,0){\(x\)}}%
%
\put(35.0000,-6.7000){\makebox(0,0){\(y\)}}%
%
\put(30.2000,-15.2000){\makebox(0,0){\({} + (x, y) - (x, x)\)}}%
%
%
\special{pn 8}%
\special{pa 400 1750}%
\special{pa 1400 1750}%
\special{pa 1400 2750}%
\special{pa 400 2750}%
\special{pa 400 1750}%
\special{dt 0.045}%
%
\special{pn 8}%
\special{pa 400 2250}%
\special{pa 1400 2250}%
\special{fp}%
\special{pn 4}%
\special{sh 1}%
\special{pa 1400 2250}%
\special{pa 1330 2220}%
\special{pa 1340 2250}%
\special{pa 1330 2280}%
\special{pa 1400 2250}%
\special{fp}%
%
\special{pn 8}%
\special{pa 900 1750}%
\special{pa 900 2200}%
\special{fp}%
\special{pn 8}%
\special{pa 900 2300}%
\special{pa 900 2750}%
\special{fp}%
\special{pn 4}%
\special{sh 1}%
\special{pa 900 2750}%
\special{pa 870 2680}%
\special{pa 900 2690}%
\special{pa 930 2680}%
\special{pa 900 2750}%
\special{fp}%
%
\put(9.9000,-26.6000){\makebox(0,0){\(x\)}}%
%
\put(13.0000,-21.7000){\makebox(0,0){\(y\)}}%
%
\put(8.6000,-28.7000){\makebox(0,0){\(- (x, y)\)}}%
%
%
\special{pn 8}%
\special{pa 2600 1750}%
\special{pa 3600 1750}%
\special{pa 3600 2750}%
\special{pa 2600 2750}%
\special{pa 2600 1750}%
\special{dt 0.045}%
%
\special{pn 8}%
\special{pa 2600 2100}%
\special{pa 3600 2100}%
\special{fp}%
\special{pn 4}%
\special{sh 1}%
\special{pa 3600 2100}%
\special{pa 3530 2070}%
\special{pa 3540 2100}%
\special{pa 3530 2130}%
\special{pa 3600 2100}%
\special{fp}%
%
\special{pn 8}%
\special{pa 3100 1750}%
\special{pa 3100 2050}%
\special{fp}%
\special{pn 8}%
\special{pa 3100 2150}%
\special{pa 3100 2300}%
\special{fp}%
\special{pn 8}%
\special{ar 3310 2420 160 170  0.0000000 2.7415927}%
\special{pn 8}%
\special{ar 3300 2420 170 170  4.7123890 6.2831854}%
\special{pn 8}%
\special{ar 3300 2450 200 200  3.1415927 4.7123890}%
\special{pn 8}%
\special{pa 3100 2450}%
\special{pa 3100 2750}%
\special{fp}%
\special{pn 4}%
\special{sh 1}%
\special{pa 3100 2750}%
\special{pa 3070 2680}%
\special{pa 3100 2690}%
\special{pa 3130 2680}%
\special{pa 3100 2750}%
\special{fp}%
%
\put(30.3000,-26.6000){\makebox(0,0){\(x\)}}%
%
\put(30.3000,-22.3000){\makebox(0,0){\(x\)}}%
%
\put(35.0000,-20.2000){\makebox(0,0){\(y\)}}%
%
\put(30.2000,-28.7000){\makebox(0,0){\({} - (x, y) + (x, x)\)}}%
%
%
\put(20.0000,-15.6500){\makebox(0,0){%
\(\stackrel{\alpha_2}{\longmapsto}\)}}%
\end{picture}%
\end{center}
\caption{Modification of coloured 2-diagrams.}
\label{FIG:2-diagram twist}
\end{figure}
As for quandle \(2\)-cycles, 
since \(C^\dgn_1(Q)\) is a zero module,
they are also rack cycles in themselves.
Thus a quandle \(2\)-cycle can be represented 
by a shadow coloured \(1\)-diagram on a circle 
or a disjoint union of finitely many circles.
So it is sufficient to add some points on this diagram 
as shown in Figure \ref{FIG:1-diagram modification}.
The new diagram \(D'\) obtained by this operation 
obviously represents \(\alpha_2(\hat{c})\). \\

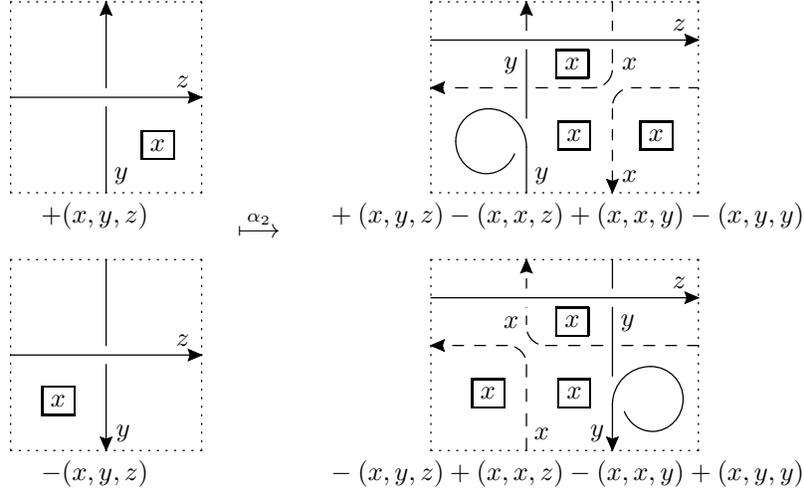
\begin{figure}[ht]
\begin{center}
\unitlength 0.1in
\begin{picture}( 41.6500, 26.0000)(  3.9200,-29.5000)
%
\special{pn 8}%
\special{pa 400 400}%
\special{pa 1400 400}%
\special{pa 1400 1400}%
\special{pa 400 1400}%
\special{pa 400 400}%
\special{dt 0.045}%
%
\special{pn 8}%
\special{pa 400 900}%
\special{pa 1400 900}%
\special{fp}%
\special{pn 4}%
\special{sh 1}%
\special{pa 1400 900}%
\special{pa 1330 870}%
\special{pa 1340 900}%
\special{pa 1330 930}%
\special{pa 1400 900}%
\special{fp}%
%
\special{pn 8}%
\special{pa 900 1400}%
\special{pa 900 950}%
\special{fp}%
\special{pn 8}%
\special{pa 900 850}%
\special{pa 900 400}%
\special{fp}%
\special{pn 4}%
\special{sh 1}%
\special{pa 900 400}%
\special{pa 870 470}%
\special{pa 900 460}%
\special{pa 930 470}%
\special{pa 900 400}%
\special{fp}%
%
\put(11.7000,-11.5000){\makebox(0,0){\fbox{\(x\)}}}%
%
\put(9.8000,-13.1000){\makebox(0,0){\(y\)}}%
%
\put(13.0000,-8.2000){\makebox(0,0){\(z\)}}%
%
\put(8.4000,-15.2000){\makebox(0,0){\(+ (x, y, z)\)}}%
%
%
\special{pn 8}%
\special{pa 2600 400}%
\special{pa 4000 400}%
\special{pa 4000 1400}%
\special{pa 2600 1400}%
\special{pa 2600 400}%
\special{dt 0.045}%
%
\special{pn 8}%
\special{pa 2600 600}%
\special{pa 4000 600}%
\special{fp}%
\special{pn 4}%
\special{sh 1}%
\special{pa 4000 600}%
\special{pa 3930 570}%
\special{pa 3940 600}%
\special{pa 3930 630}%
\special{pa 4000 600}%
\special{fp}%
%
\special{pn 8}%
\special{pa 3100 1400}%
\special{pa 3100 1160}%
\special{fp}%
\special{pn 8}%
\special{ar 2900 1160 200 200  4.7123890 6.2831853}%
\special{pn 8}%
\special{ar 2900 1130 170 170  3.1415927 4.7123890}%
\special{pn 8}%
\special{ar 2890 1130 160 170  0.4000000 3.1415927}%
\special{pn 8}%
\special{pa 3100 1010}%
\special{pa 3100 650}%
\special{fp}%
\special{pn 8}%
\special{pa 3100 550}%
\special{pa 3100 400}%
\special{fp}%
\special{pn 4}%
\special{sh 1}%
\special{pa 3100 400}%
\special{pa 3070 470}%
\special{pa 3100 460}%
\special{pa 3130 470}%
\special{pa 3100 400}%
\special{fp}%
%
\special{pn 8}%
\special{pa 3550 400}%
\special{pa 3550 550}%
\special{da 0.06}%
\special{pn 8}%
\special{pa 3550 650}%
\special{pa 3550 720}%
\special{da 0.06}%
\special{pn 8}%
\special{ar 3450 750 100 100  0.1000000 0.9853932}%
\special{ar 3450 750 100 100  1.3853932 1.5707863}%
\special{pn 8}%
\special{pa 3450 850}%
\special{pa 3150 850}%
\special{da 0.06}%
\special{pn 8}%
\special{pa 3050 850}%
\special{pa 2600 850}%
\special{da 0.06}%
\special{pn 4}%
\special{sh 1}%
\special{pa 2600 850}%
\special{pa 2670 880}%
\special{pa 2660 850}%
\special{pa 2670 820}%
\special{pa 2600 850}%
\special{fp}%
%
\special{pn 8}%
\special{pa 4000 850}%
\special{pa 3650 850}%
\special{da 0.06}%
\special{pn 8}%
\special{ar 3650 950 100 100  3.3415927 4.2123890}%
\special{pn 8}%
\special{pa 3550 960}%
\special{pa 3550 1340}%
\special{da 0.06}%
\special{pn 4}%
\special{sh 1}%
\special{pa 3550 1400}%
\special{pa 3580 1330}%
\special{pa 3550 1340}%
\special{pa 3520 1330}%
\special{pa 3550 1400}%
\special{fp}%
%
\put(33.4000,-7.2500){\makebox(0,0){\fbox{\(x\)}}}%
%
\put(33.5000,-11.0000){\makebox(0,0){\fbox{\(x\)}}}%
%
\put(37.8000,-11.0000){\makebox(0,0){\fbox{\(x\)}}}%
%
\put(36.4000,-7.2500){\makebox(0,0){\(x\)}}%
%
\put(36.4000,-13.1000){\makebox(0,0){\(x\)}}%
%
\put(31.8000,-13.1000){\makebox(0,0){\(y\)}}%
%
\put(30.2000,-7.3000){\makebox(0,0){\(y\)}}%
%
\put(39.0000,-5.2000){\makebox(0,0){\(z\)}}%
%
\put(33.0000,-15.2000){\makebox(0,0){%
\({} + (x, y, z) - (x, x, z) + (x, x, y) - (x, y, y)\)}}%
%
%
\special{pn 8}%
\special{pa 400 1750}%
\special{pa 1400 1750}%
\special{pa 1400 2750}%
\special{pa 400 2750}%
\special{pa 400 1750}%
\special{dt 0.045}%
%
\special{pn 8}%
\special{pa 400 2250}%
\special{pa 1400 2250}%
\special{fp}%
\special{pn 4}%
\special{sh 1}%
\special{pa 1400 2250}%
\special{pa 1330 2220}%
\special{pa 1340 2250}%
\special{pa 1330 2280}%
\special{pa 1400 2250}%
\special{fp}%
%
\special{pn 8}%
\special{pa 900 1750}%
\special{pa 900 2200}%
\special{fp}%
\special{pn 8}%
\special{pa 900 2300}%
\special{pa 900 2750}%
\special{fp}%
\special{pn 4}%
\special{sh 1}%
\special{pa 900 2750}%
\special{pa 870 2680}%
\special{pa 900 2690}%
\special{pa 930 2680}%
\special{pa 900 2750}%
\special{fp}%
%
\put(6.5000,-24.9000){\makebox(0,0){\fbox{\(x\)}}}%
%
\put(9.9000,-26.6000){\makebox(0,0){\(y\)}}%
%
\put(13.0000,-21.7000){\makebox(0,0){\(z\)}}%
%
\put(8.4000,-28.7000){\makebox(0,0){\(- (x, y, z)\)}}%
%
%
\special{pn 8}%
\special{pa 2600 1750}%
\special{pa 4000 1750}%
\special{pa 4000 2750}%
\special{pa 2600 2750}%
\special{pa 2600 1750}%
\special{dt 0.045}%
%
\special{pn 8}%
\special{pa 2600 1950}%
\special{pa 4000 1950}%
\special{fp}%
\special{pn 4}%
\special{sh 1}%
\special{pa 4000 1950}%
\special{pa 3930 1920}%
\special{pa 3940 1950}%
\special{pa 3930 1980}%
\special{pa 4000 1950}%
\special{fp}%
%
\special{pn 8}%
\special{pa 3550 1750}%
\special{pa 3550 1900}%
\special{fp}%
\special{pn 8}%
\special{pa 3550 2000}%
\special{pa 3550 2360}%
\special{fp}%
\special{pn 8}%
\special{ar 3760 2480 160 170  0.0000000 2.7415927}%
\special{pn 8}%
\special{ar 3750 2480 170 170  4.7123890 6.2831854}%
\special{pn 8}%
\special{ar 3750 2510 200 200  3.1415927 4.7123890}%
\special{pn 8}%
\special{pa 3550 2510}%
\special{pa 3550 2750}%
\special{fp}%
\special{pn 4}%
\special{sh 1}%
\special{pa 3550 2750}%
\special{pa 3520 2680}%
\special{pa 3550 2690}%
\special{pa 3580 2680}%
\special{pa 3550 2750}%
\special{fp}%
%
\special{pn 8}%
\special{pa 4000 2200}%
\special{pa 3600 2200}%
\special{da 0.06}%
\special{pn 8}%
\special{pa 3500 2200}%
\special{pa 3200 2200}%
\special{da 0.06}%
\special{pn 8}%
\special{ar 3200 2100 100 100  1.9707863 2.8415927}%
\special{pn 8}%
\special{pa 3100 2090}%
\special{pa 3100 2000}%
\special{da 0.06}%
\special{pn 8}%
\special{pa 3100 1900}%
\special{pa 3100 1750}%
\special{da 0.06}%
\special{pn 4}%
\special{sh 1}%
\special{pa 3100 1750}%
\special{pa 3130 1820}%
\special{pa 3100 1810}%
\special{pa 3070 1820}%
\special{pa 3100 1750}%
\special{fp}%
%
\special{pn 8}%
\special{pa 3100 2750}%
\special{pa 3100 2300}%
\special{da 0.06}%
\special{pn 8}%
\special{ar 3000 2300 100 100  5.1123890 5.8831854}%
\special{pn 8}%
\special{pa 3000 2200}%
\special{pa 2600 2200}%
\special{da 0.06}%
\special{pn 4}%
\special{sh 1}%
\special{pa 2600 2200}%
\special{pa 2670 2170}%
\special{pa 2660 2200}%
\special{pa 2670 2230}%
\special{pa 2600 2200}%
\special{fp}%
%
\put(33.4000,-20.7500){\makebox(0,0){\fbox{\(x\)}}}%
%
\put(33.5000,-24.5000){\makebox(0,0){\fbox{\(x\)}}}%
%
\put(29.0000,-24.5000){\makebox(0,0){\fbox{\(x\)}}}%
%
\put(31.8000,-26.6000){\makebox(0,0){\(x\)}}%
%
\put(30.2000,-20.8000){\makebox(0,0){\(x\)}}%
%
\put(34.7000,-26.6000){\makebox(0,0){\(y\)}}%
%
\put(36.3000,-20.8000){\makebox(0,0){\(y\)}}%
%
\put(39.0000,-18.7000){\makebox(0,0){\(z\)}}%
%
\put(33.0000,-28.7000){\makebox(0,0){%
\({} - (x, y, z) + (x, x, z) - (x, x, y) + (x, y, y)\)}}%
%
%
\put(17.0000,-15.6500){\makebox(0,0){%
\(\stackrel{\alpha_2}{\longmapsto}\)}}%
\end{picture}%
\end{center}
\caption{Modification of shadow coloured 2-diagrams.}
\label{FIG:2-diagram parallelisation}
\end{figure}
Before considering the case of \(3\)-cycles,
we will discuss the diagrammatic translation 
of \(\alpha_2\) by coloured diagrams without shadow colourings.
Also in this case,
we can suppose a \(2\)-diagram 
representing a quandle \(2\)-cycle 
to be on a closed surface 
or a disjoint union of closed surfaces.
To add degeneracy terms \((x, x)\) of \(\alpha_2(x, y)\),
we modify the diagram at each crossing 
by RI transformations, as drawn 
in Figure \ref{FIG:2-diagram twist}.
In this figure,
the initial under-arc is twisted 
to generate a term \((x, x)\). \\

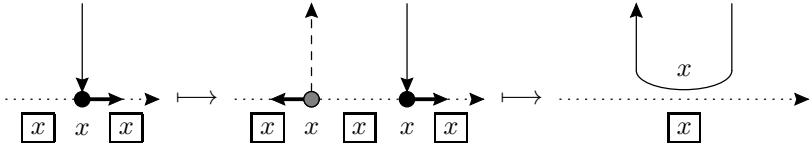
\begin{figure}[ht]
\begin{center}
\unitlength 0.1in
\begin{picture}( 42.1400, 8.0000)(  1.9200, -12.3000)
%
%
\special{pn 8}%
\special{pa  200 1000}%
\special{pa  950 1000}%
\special{dt 0.045}%
\special{sh 1}%
\special{pn 4}%
\special{pa 1000 1000}%
\special{pa  930  970}%
\special{pa  940 1000}%
\special{pa  930 1030}%
\special{pa 1000 1000}%
\special{fp}%
%
\special{pn 8}%
\special{pa  600  500}%
\special{pa  600  960}%
\special{fp}%
\special{sh 1}%
\special{pn 4}%
\special{pa  600  900}%
\special{pa  630  890}%
\special{pa  600  960}%
\special{pa  570  890}%
\special{pa  600  900}%
\special{fp}%
%
\special{pn 8}%
\special{sh 1}%
\special{ar 600 1000 40 40  0.0000000 6.2831853}%
%
\special{pn 20}%
\special{pa 600 1000}%
\special{pa 750 1000}%
\special{fp}%
\special{sh 1}%
\special{pn 4}%
\special{pa 810 1000}%
\special{pa 740 970}%
\special{pa 750 1000}%
\special{pa 740 1030}%
\special{pa 810 1000}%
\special{fp}%
%
\put(3.7000,-11.5000){\makebox(0,0){\fbox{\(x\)}}}%
%
\put(8.3000,-11.5000){\makebox(0,0){\fbox{\(x\)}}}%
%
\put(6.0000,-11.5900){\makebox(0,0){\(x\)}}%
%
%
\special{pn 8}%
\special{pa 1400 1000}%
\special{pa 2700 1000}%
\special{dt 0.045}%
\special{sh 1}%
\special{pn 4}%
\special{pa 2700 1000}%
\special{pa 2630  970}%
\special{pa 2640 1000}%
\special{pa 2630 1030}%
\special{pa 2700 1000}%
\special{fp}%
%
\special{pn 8}%
\special{pa 1800  960}%
\special{pa 1800  500}%
\special{da 0.04}%
\special{sh 1}%
\special{pn 4}%
\special{pa 1800  500}%
\special{pa 1830  570}%
\special{pa 1800  560}%
\special{pa 1770  570}%
\special{pa 1800  500}%
\special{fp}%
%
\special{pn 20}%
\special{pa 1800 1000}%
\special{pa 1650 1000}%
\special{fp}%
\special{sh 1}%
\special{pn 4}%
\special{pa 1590 1000}%
\special{pa 1660 970}%
\special{pa 1650 1000}%
\special{pa 1660 1030}%
\special{pa 1590 1000}%
\special{fp}%
%
\special{pn 8}%
\special{sh 0}%
\special{ar 1800 1000 40 40  0.0000000 6.2831853}%
\special{pn 8}%
\special{sh 0.5}%
\special{ar 1800 1000 40 40  0.0000000 6.2831853}%
%
\special{pn 8}%
\special{pa 2300  500}%
\special{pa 2300  960}%
\special{fp}%
\special{sh 1}%
\special{pn 4}%
\special{pa 2300  900}%
\special{pa 2330  890}%
\special{pa 2300  960}%
\special{pa 2270  890}%
\special{pa 2300  900}%
\special{fp}%
%
\special{pn 20}%
\special{pa 2300 1000}%
\special{pa 2450 1000}%
\special{fp}%
\special{sh 1}%
\special{pn 4}%
\special{pa 2510 1000}%
\special{pa 2440  970}%
\special{pa 2450 1000}%
\special{pa 2440 1030}%
\special{pa 2510 1000}%
\special{fp}%
%
\special{pn 8}%
\special{sh 1}%
\special{ar 2300 1000 40 40  0.0000000 6.2831853}%
%
\put(15.7000,-11.5000){\makebox(0,0){\fbox{\(x\)}}}%
%
\put(18.0000,-11.5000){\makebox(0,0){\(x\)}}%
%
\put(20.5000,-11.5000){\makebox(0,0){\fbox{\(x\)}}}%
%
\put(23.0000,-11.5000){\makebox(0,0){\(x\)}}%
%
\put(25.3000,-11.5000){\makebox(0,0){\fbox{\(x\)}}}%
%
%
\special{pn 8}%
\special{pa 3100 1000}%
\special{pa 4400 1000}%
\special{dt 0.045}%
\special{sh 1}%
\special{pn 4}%
\special{pa 4400 1000}%
\special{pa 4330 970}%
\special{pa 4340 1000}%
\special{pa 4330 1030}%
\special{pa 4400 1000}%
\special{fp}%
%
\special{pn 8}%
\special{pa 3500  850}%
\special{pa 3500  500}%
\special{fp}%
\special{sh 1}%
\special{pn 4}%
\special{pa 3500  500}%
\special{pa 3530  570}%
\special{pa 3500  560}%
\special{pa 3470  570}%
\special{pa 3500  500}%
\special{fp}%
%
\special{pn 8}%
\special{pa 4000  500}%
\special{pa 4000  850}%
\special{fp}%
%
\special{pn 8}%
\special{ar 3750  850 250 100  0.0000000 3.1415926}%
%
\put(37.5000, -8.5000){\makebox(0,0){\(x\)}}%
%
\put(37.5000,-11.5000){\makebox(0,0){\fbox{\(x\)}}}%
%
%
\put(12.0000,-10.0000){\makebox(0,0){\(\longmapsto\)}}%
%
\put(29.0000,-10.0000){\makebox(0,0){\(\longmapsto\)}}%
\end{picture}%
\end{center}
\caption{Modification of 2-diagram on boundary.}
\label{FIG:2-diagram boundary cancellation}
\end{figure}
Fix a quandle \(3\)-cycle \(c \in Z^\qdl_3(Q)\) 
and \(Q\)-shadow coloured \(2\)-diagram \(D\) 
which represents \(\hat{c}\). 
At first, we modify the diagram \(D\) 
so that \(\sigma D\) represents \(\alpha_2\langle D\rangle\) 
as in Figure \ref{FIG:2-diagram twist}. 
Next, 
we will parallelise the arcs of \(D\) 
such that the additional arcs are 
supposed to be in the lowest level and 
to be smoothened near each crossing, 
as drawn in Figure \ref{FIG:2-diagram parallelisation}. 
For simplicity, 
the trivial components to appear are removed there. 

Easily, we can see that the new diagram \(D'\) 
represents \(\alpha_3(\hat{c})\), 
but it still has its boundary. 
Since \(\partial_3(\hat{c})\) is 
a degeneracy \(2\)-chain, 
each additional point on \(\partial D'\) 
lies next to an original point 
with the same colour and with the opposite signature.
Thus, by connecting these pairs with arcs, 
we obtain another diagram \(D''\) with 
trivially coloured boundary, 
which can be capped off by disks. 


\subsection{Construction of shadow diagram classes.}
\label{SEC:construction shadow class}

Let \(L\) be a link 
and \(D\) a diagram of \(L\) on \(\sphere^2\). 
As mentioned in \S \ref{SEC:knot quandle}, 
the diagram \(D\) is canonically coloured by \(Q(L)\), 
thus \(D\) is freely \(Q(L)\)-shadow colourable 
by Corollary \ref{COR:sphere 2-diagram}.
Denote by \(D_\alpha\) a \(Q(L)\)-shadow coloured diagram 
obtained from \(D\) by colouring the base-region 
with \(\alpha \in Q(L)\). 
The rack \(3\)-chain \(\langle D_\alpha \rangle\) is a cycle,
for it is represented by a diagram on a closed manifold.
Thus, we have a rack homology class 
\([D_\alpha] \in H^\rck_3(Q(L))\) from \(D_\alpha\),
called a \textbf{shadow diagram class} of \(D\).
We also have a quandle homology class \([L_\alpha] 
\in H^\qdl_3(Q(L))\) as the image of \([D_\alpha]\)
via \(\rho_\ast\): \(H^\rck_3(Q(L)) \to H^\qdl_3(Q(L))\).
In parallel with the relation between 
the diagram class and the fundamental class,
we call it a \textbf{shadow fundamental class} of \(L\).

Now, we will show one of the main theorems:

\begin{theorem}
\label{TH:construction theorem}
Let \(L\) be a non-trivial \(n\)-component link. 
\begin{itemize}
\item[a)] The shadow diagram classes \([D_\alpha]\) 
are non-zero elements of \(H^\rck_3(Q(L))\) 
for each diagram \(D\) of \(L\) and \(\alpha \in Q(L)\). 
\item[b)] There exist linearly independent \(n\) shadow 
fundamental classes of \(L\) in \(H_3^\qdl(Q(L))\). 
\item[c)] The third quandle homology group 
\(H_3^\qdl(Q(L))\) splits into the direct sum 
\begin{quote}
\((\bigoplus \Z [L_i]) 
\oplus \Bigl(H_3^\qdl(Q(L))/(\bigoplus \Z [L_i])\Bigr)\), 
\end{quote}
where \([L_1]\), \ldots, \([L_n]\) are 
distinct shadow fundamental classes of \(L\). 
\end{itemize}
\end{theorem}

\begin{remark}
\label{REM:un-uniqueness}
If \(L\) has more than one component, 
the shadow `fundamental' class is not unique. 
It seems to be slightly confusing, 
but, as far as the shadow cocycle invariants concern, 
all the shadow fundamental classes work in the same way. 
\end{remark}

\begin{proof}
We have already noticed that 
the shifting homomorphism is an operation 
to disregard shadow colours of diagrams. 
Thus, it is clear that \(\sigma_3 \langle D_\alpha \rangle 
= \langle D \rangle\), 
which concludes \(\rho_\ast \sigma_\ast[D_\alpha] 
= \rho_\ast[D] = [L]\).
By Theorem \ref{TH:second homology}, the fundamental class 
\([L]\) is non-zero since \(L\) is non-trivial. 
It completes the proof of (a). \\

If \(\alpha\) and \(\beta \in Q(L)\) are connected, 
\(D_\alpha\) and \(D_\beta\) 
represent rack homologous \(3\)-cycles 
by Lemma \ref{LEM:cobordism 2-diagram}, 
so \([D_\alpha] = [D_\beta]\) holds in \(H^\rck_3(Q(L))\). 

Let \(D\) and \(D'\) be diagrams of \(L\) on \(\sphere^2\). 
We can obtain \(D'\) from \(D\) 
by finitely many Reidemeister deformations. 
Through these deformations, 
a \(Q(L)\)-shadow coloured diagram \(D_\alpha\) 
becomes another \(Q(L)\)-shadow coloured diagram \(D'_\beta\). 
Clearly, \(\beta\) is connected with \(\alpha\). 
Therefore, Lemmas \ref{LEM:cobordism 2-diagram}, 
\ref{LEM:R-II R-III deformation} 
and \ref{LEM:R-I deformation} show that 
\(\langle D'_\alpha \rangle\), \(\langle D'_\beta \rangle\) 
and \(\langle D_\alpha \rangle\) are all quandle homologous. 

Thus we conclude that \([L_\alpha]\) is determined 
independently of the choice of diagrams 
and of the shadow colours of the base-region 
as long as the shadow colours are connected. 
It is well known that the knot quandle 
of \(n\)-component link has \(n\) orbits. 
So, it follows that there exist 
at most \(n\) shadow fundamental classes of \(L\). \\

Let \(\alpha_1\), \ldots, \(\alpha_n\) be elements 
of \(Q(L)\) which are not connected each other. 
To prove that \([D_{\alpha_1}]\), \ldots, \([D_{\alpha_n}]\) 
are linearly independent, 
we use some evaluation maps. 
For a \(2\)-cocycle \(\phi \in Z_\rck^2(Q(L))\) 
and an element \(\beta \in Q(L)\), 
define a rack \(3\)-cocycle \(\tilde{\phi}_\beta\) by 
\begin{quote}
\(\tilde{\phi}_\beta(\alpha, a, b) = \left\{
\begin{tabular}{ll}
\(\phi(a, b)\) 
& if \(\alpha\) and \(\beta\) are in the same orbit, \\
\(0\) & otherwise. 
\end{tabular}\right.\)
\end{quote}
Easy computation shows that \(\tilde{\phi}_\beta\) 
is also a cocycle. 
Denote by \(S\) a linear combination 
\(\sum c_i [D_{\alpha_i}]\). 
Clearly by the definition, we have 
\begin{quote}
\(\langle S, [\tilde{\phi}_{\alpha_i}]\rangle 
= c_i \langle [D], [\phi]\rangle\). 
\end{quote}
Thus, if we can choose a non-trivial cocycle \(\phi\) 
such that \(\langle [D], [\phi]\rangle\) is non-zero, 
then the classes \([D_{\alpha_1}]\), \ldots, \([D_{\alpha_n}]\) 
are linearly independent. 
In fact, Eisermann's results in [E] says 
that there exists such a cocycle \(\phi\) of \(Q(L)\) 
when \(L\) is non-trivial. 

By the definition,
the shadow fundamental class \([L_\alpha]\) 
can be written in the form of \(\rho_\ast[D_\alpha]\) 
for some knot diagram \(D\) of \(L\).
As shown in \S \ref{SEC:homomorphism diagram}, 
\(\alpha_\ast \rho_\ast[D_\alpha]\) can be represented 
by some shadow coloured diagram on a closed surface. 
By observing the operations 
in \S \ref{SEC:homomorphism diagram} precisely, 
we can see that the resulting diagram \(D'_\alpha\) from \(D_\alpha\) 
is also on a sphere \(\sphere^2\).
Moreover, the additional arcs of \(D'_\alpha\) are all 
simple closed curves and in the lowest level of the diagram. 
Since such curves on \(\sphere^2\) can be removed 
via Reidemeister deformations II and III, 
the \(3\)-chain \(\langle D'_\alpha\rangle\) 
is rack homologous to \(\langle D_\alpha\rangle\) 
from Lemma \ref{LEM:R-II R-III deformation}, 
which concludes that \([L_\alpha]\) is equal to one of 
the shadow diagram classes.
Our proof of (b) is completed. \\

Fix \(\alpha \in Q(L)\). 
From the fact above, 
if a diagram \(D\) with the canonical colouring 
represents the fundamental class \([L]\) 
in \(H^\rck_2(Q(L))\), 
then the shadow coloured diagram \(D_\alpha\) 
represents the shadow fundamental class \([L_\alpha]\) 
in \(H^\rck_3(Q(L))\). 
Thus we have two homomorphisms 
\begin{quote}
\(H^\qdl_2(Q(L)) \cong \Z[L] \cong \Z[L_\alpha] \subset H^\rck_3(Q(L))\) 
\end{quote}
and 
\begin{quote}
\(H^\rck_3(Q(L)) \stackrel{\sigma_\ast}{\to} 
H^\rck_2(Q(L)) \stackrel{\rho_\ast}{\to} 
H^\qdl_2(Q(L)) \cong \Z[L]\).
\end{quote}
Easily, we can check diagrammatically 
that the composition of them becomes the identity on \(\Z[L]\).
Thus, with Theorem \ref{TH:homology splitting}, 
\(H^\qdl_3(Q(L))\) is proved to split as in the statement of (c).
\end{proof}

\subsection{Relations to shadow cocycle invariants.}
\label{SEC:shadow cocycle invariant}

Shadow cocycle invariants \(\Phi_\phi\) defined in [CJKS] 
are invariants of links 
computed with a quandle \(3\)-cocycle \(\phi\) 
of a finite quandle \(X\). 
Satoh [S] introduced based shadow cocycle invariants 
\(\Phi^\ast_\phi\) of links, and he proved that, 
if \(X\) is a dihedral quandle \(\Z_p\) for a prime 
odd \(p\), an equation \(\Phi_\phi = |X| \cdot \Phi^\ast_\phi\) 
holds for the two invariants. 
Here, we will prove this equation for a connected quandle \(X\) 
in general, by using the concept of shadow fundamental classes. \\

Let \(L\) be a link and \(D\) a diagram of \(L\).
At first, we will define the \textbf{shadow cocycle invariant}
\(\Phi_\phi(L)\) and the \textbf{based shadow cocycle invariant}
\(\Phi^\ast_\phi(L)\) of \(L\). 
Fix a quandle \(3\)-cocycle \(\phi \in Z_\qdl^3(X; A)\) 
of a finite quandle \(X\), 
where the operation of \(A\) is supposed to be 
written as multiplication. 
When \(C\) is an \(X\)-shadow colouring of \(D\), 
we define the \textbf{Boltzmann weight} \(w(c)\) 
of each crossing \(c\) by 
\begin{quote}
\(w(c) = \phi(C(r^\ini_c), C(u^\ini_c), C(o_c))^{\epsilon_c}\). 
\end{quote}
The symbols above follow in \S \ref{SEC:colouring}.
Then, we define the \textbf{whole weight} \(W(C)\) of \(C\) by
the product \(\prod w(c)\) of weights of all crossings.

Both the unbased and the based shadow cocycle invariants
are of the state-sum type.
They use \(X\)-shadow colourings as states,
but they differ on the colourings which they allow as states.

In the case of shadow cocycle invariants,
all \(X\)-shadow colourings are allowed to be states of \(D\).
Therefore, \(\Phi_\phi(L)\) is the sum \(\sum W(C)\),
where \(C\) ranges all \(X\)-shadow colourings.

On the other hand,
when we are concerned with the based shadow cocycle invariants, 
we fix a point \(p\) of an arc of \(D\) which is not a crossing. 
An \(X\)-shadow colouring is allowed to be a state
when the regions \(r^\ini_p\) and \(r^\ter_p\) 
have as same a colour as \(c_p\) (See Figure \ref{FIG:allowed colouring}).
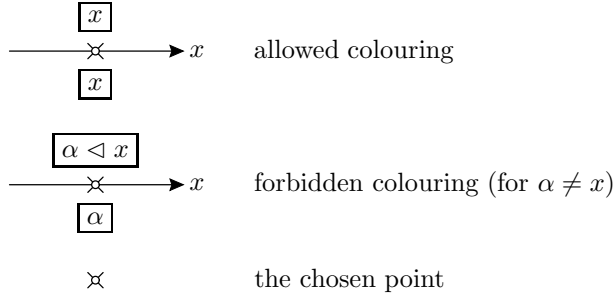
\begin{figure}[ht]
\begin{center}
\unitlength 0.1in
\begin{picture}( 31.7500, 16.1000)(-0.1000, -15.5000)
%
\special{pn 8}%
\special{pa 0000 250}%
\special{pa 0900 250}%
\special{fp}%
\special{sh 1}%
\special{pn 4}%
\special{pa 0910 250}%
\special{pa 0840 220}%
\special{pa 0850 250}%
\special{pa 0840 280}%
\special{pa 0910 250}%
\special{fp}%
%
\special{pn 8}%
\special{pa 410 210}%
\special{pa 490 290}%
\special{fp}%
\special{pa 490 210}%
\special{pa 410 290}%
\special{fp}%
%
\special{pn 8}%
\special{sh 0}%
\special{ar 450 250 20 20  0.0000000 6.2831853}%
%
\put(9.8000,-2.5000){\makebox(0,0){\(x\)}}%
%
\put(4.4900,-4.2500){\makebox(0,0){\fbox{\(x\)}}}%
%
\put(4.4900,-0.7100){\makebox(0,0){\fbox{\(x\)}}}%
%
\put(18.0000,-2.5000){\makebox(0,0){allowed colouring}}%
%
%
\special{pn 8}%
\special{pa 0000 950}%
\special{pa 0900 950}%
\special{fp}%
\special{sh 1}%
\special{pn 4}%
\special{pa 0910 950}%
\special{pa 0840 920}%
\special{pa 0850 950}%
\special{pa 0840 980}%
\special{pa 0910 950}%
\special{fp}%
%
\special{pn 8}%
\special{pa 410 910}%
\special{pa 490 990}%
\special{fp}%
\special{pa 490 910}%
\special{pa 410 990}%
\special{fp}%
%
\special{pn 8}%
\special{sh 0}%
\special{ar 450 950 20 20  0.0000000 6.2831853}%
%
\put(9.8000,-9.5000){\makebox(0,0){\(x\)}}%
%
\put(4.4900,-11.2500){\makebox(0,0){\fbox{\(\alpha\)}}}%
%
\put(4.4900,-7.7100){\makebox(0,0){\fbox{\(\alpha \lhd x\)}}}%
%
\put(22.2500,-9.5000){\makebox(0,0){forbidden colouring %
(for \(\alpha \ne x\))}}%
%
%
\special{pn 8}%
\special{pa 410 1410}%
\special{pa 490 1490}%
\special{fp}%
\special{pa 490 1410}%
\special{pa 410 1490}%
\special{fp}%
%
\special{pn 8}%
\special{sh 0}%
\special{ar 450 1450 20 20  0.0000000 6.2831853}%
%
\put(17.8000,-14.5000){\makebox(0,0){the chosen point}}%
\end{picture}%
\end{center}
\caption{Allowed and forbidden colourings.}
\label{FIG:allowed colouring}
\end{figure}
It is proved in [S] that the sum \(\Phi^\ast_\phi(L)\) of whole weights
of all colourings that satisfy the condition as above 
is invariant of \(L\). \\

\begin{remark}
\label{REM:chosen point}
The invariant \(\Phi^\ast_\phi\) has a name as ``based'', 
for the chosen point \(p\) is called a ``basepoint'' in [S]. 
Since we are necessary to consider the basepoint 
of a manifold where a diagram lies, 
it is confusing to call \(p\) a basepoint. 
Thus, we use the term ``chosen point'' instead of basepoint.
\end{remark}

\begin{theorem}
\label{TH:invariants relation}
If \(X\) is a connected finite quandle,
an equation
\begin{quote}
\(\Phi_\phi(L) = |X| \cdot \Phi^\ast_\phi(L)\) 
\end{quote}
holds between the unbased and 
the based shadow cocycle invariants. 
\end{theorem}

\begin{proof}
Let \(p\) be a chosen point 
of a diagram \(D\) of \(L\) 
and set \(k = |X|\).
Fix an \(X\)-colouring \(C\) of \(D\).
From Corollary \ref{COR:sphere 2-diagram},
when we give a shadow colour \(\alpha\) 
to the region \(r^\ini_p\), 
we have a whole shadow colouring of \(D\).
Thus there are \(k\) shadow colourings 
\(C_1\), \ldots, \(C_k\) extending \(C\).
Clearly, only one of them, say \(C_1\), is 
an allowed colouring, others are not.
Since \(X\) is connected, 
Corollary \ref{COR:diagram class} says 
that \([D_1] = \cdots = [D_k]\), 
where \(D_i\) denotes a shadow coloured 
diagram \(D\) with \(C \cup C_i\).
Therefore we have an equation 
\begin{quote}
\(\sum_i \langle [D_i], \phi\rangle 
= |X| \cdot \langle [D_1], \phi \rangle\).
\end{quote}
This lemma is a direct conclusion of the equation.
\end{proof}

\section{Topological realisation of \(3\)-cycles.}
\label{CHAP:topological realisation}

\subsection{Surgeries on coloured diagrams.}
\label{SEC:diagram surgery}

When the homology theory 
with integral coefficients concerns, 
every \(3\)-cycle \(c\) in \(Z_3^\rck(Q(K))\) 
can be represented by a \(Q(K)\)-shadow coloured diagram \(D\) 
on some closed surface \(M\). 
We prove, under some conditions, 
that \(D\) can be chosen as a diagram on \(\sphere^2\). \\

\begin{theorem}
\label{TH:topological realisation}
Let \(K\) be a prime knot. 
For any quandle \(3\)-cycle \(c \in Z^\qdl_3(Q(K))\), 
there exists a pair of a link \(L\)
and a homomorphism \(f\): \(Q(L) \to Q(K)\)
such that \([c] = f_\ast[L_\sh]\),
where \([L_\sh]\) is one of the shadow fundamental 
classes of \(L\). 
\end{theorem}

\begin{proof}
As already seen in \S \ref{SEC:homomorphism diagram},
there exists a \(Q(K)\)-shadow 
coloured \(2\)-diagram \(D\) 
on a closed surface \(M\) such that 
\(D\) represents \([c]\).
Denote by \(\ast\) the basepoint of \(M\)
and denote by \(a\) the shadow colour of the base-region.
Throughout the proof, closed curves on \(M\) are 
supposed to be based, 
that is, they start from \(\ast\). 
A closed curve \(C\) on \(M\) is called to be
in generic position when \(C\) does not pass any crossings of \(D\)
and it crosses over \(D\) transversely 
where it intersects with \(D\). 
Any closed curve \(C\) on \(M\) can be transformed homotopically 
into a new curve \(C'\) in generic position, 
thus we will omit the notation about genericity. 
Also we notice that the intersection of a curve \(C\) 
with the diagram \(D\) makes \((C, D \cap C)\) 
a \(Q(K)\)-shadow coloured \(1\)-diagram. 
We denote this diagram also by \(C\). 

Suppose that there exists an essential curve \(C\) on \(M\)
such that \(\Pi(C) = e \in \pi(K)\).
By Corollary \ref{COR:boundary 1-diagram} 
and Lemma \ref{LEM:disk 2-diagram}, 
there exists a \(Q(K)\)-shadow coloured \(2\)-diagram 
\(D_1\) on \(\disk^2\) such that \(\partial D_1 = C\). 
Therefore, by cutting \(M\) along \(C\), 
and by attaching \(D_1\) to \(C\) 
and \(- D_1\) to \(- C\), 
as depicted in Figure \ref{FIG:capping off}, 
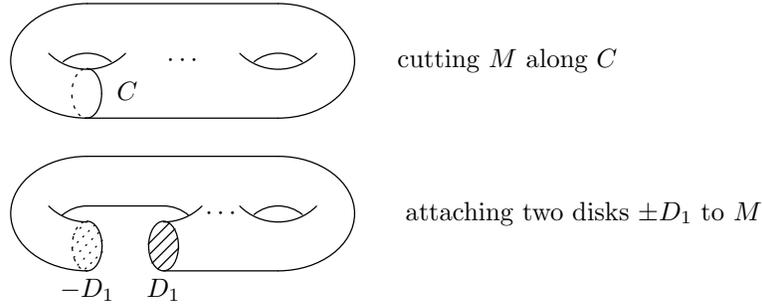
\begin{figure}[ht]
\begin{center}
\unitlength 0.1in%
\begin{picture}(39.5000, 17.0000)(  4.0000,-20.0000)%
%
%
\special{pn 8}%
\special{ar 800 700 400 300  1.5707963 4.7123890}%
\special{pn 8}%
\special{pa 1800 400}%
\special{pa 800 400}%
\special{fp}%
\special{pn 8}%
\special{ar 1800 700 400 300  4.7123890 6.2831853}%
\special{ar 1800 700 400 300  0.0000000 1.5707963}%
\special{pn 8}%
\special{pa 1800 1000}%
\special{pa 800 1000}%
\special{fp}%
\special{pn 8}%
\special{ar 800 460 284 284  0.7853982 2.3561945}%
\special{pn 8}%
\special{ar 800 860 200 200  4.0300724 5.3947055}%
\put(13.0000,-7.0000){\makebox(0,0){\(\cdots\)}}%
\special{pn 8}%
\special{ar 1800 460 284 284  0.7853982 2.3561945}%
\special{pn 8}%
\special{ar 1800 860 200 200  4.0300724 5.3947055}%
\special{pn 8}%
\special{ar 798 870 78 130  4.7123890 6.2831853}%
\special{ar 798 870 78 130  0.0000000 1.5707963}%
\special{pn 8}%
\special{ar 798 870 78 130  1.5707963 1.6861809}%
\special{ar 798 870 78 130  2.0323348 2.1477194}%
\special{ar 798 870 78 130  2.4938732 2.6092579}%
\special{ar 798 870 78 130  2.9554117 3.0707963}%
\special{ar 798 870 78 130  3.4169502 3.5323348}%
\special{ar 798 870 78 130  3.8784886 3.9938732}%
\special{ar 798 870 78 130  4.3400271 4.4554117}%
\put(10.1000,-8.6000){\makebox(0,0){\(C\)}}%
\put(30.0000,-7.0000){\makebox(0,0){cutting \(M\) along \(C\)}}%
%
%
\special{pn 8}%
\special{ar 800 1500 400 300  1.5707963 4.7123890}%
\special{pn 8}%
\special{pa 800 1200}%
\special{pa 1800 1200}%
\special{fp}%
\special{pn 8}%
\special{ar 1800 1500 400 300  4.7123890 6.2831853}%
\special{ar 1800 1500 400 300  0.0000000 1.5707963}%
\special{pn 8}%
\special{pa 1800 1800}%
\special{pa 1200 1800}%
\special{fp}%
\special{pn 8}%
\special{ar 1200 1670 78 130  0.0000000 6.2831853}%
\special{pn 8}%
\special{ar 1200 1260 284 284  0.7853982 1.5707963}%
\special{pn 8}%
\special{ar 1200 1660 200 200  4.7123890 5.4264797}%
\special{pn 8}%
\special{pa 1200 1460}%
\special{pa 800 1460}%
\special{fp}%
\special{pn 8}%
\special{ar 800 1660 200 200  3.9982983 4.7123890}%
\special{pn 8}%
\special{ar 800 1260 284 284  1.5707963 2.3561945}%
\special{pn 8}%
\special{ar 800 1670 78 130  4.7123890 6.2831853}%
\special{ar 800 1670 78 130  0.0000000 1.5707963}%
\special{pn 8}%
\special{ar 800 1670 78 130  1.5707963 1.6861809}%
\special{ar 800 1670 78 130  2.0323348 2.1477194}%
\special{ar 800 1670 78 130  2.4938732 2.6092579}%
\special{ar 800 1670 78 130  2.9554117 3.0707963}%
\special{ar 800 1670 78 130  3.4169502 3.5323348}%
\special{ar 800 1670 78 130  3.8784886 3.9938732}%
\special{ar 800 1670 78 130  4.3400271 4.4554117}%
\put(15.0000,-15.0000){\makebox(0,0){\(\cdots\)}}%
\special{pn 8}%
\special{ar 1800 1260 284 284  0.7853982 2.3561945}%
\special{pn 8}%
\special{ar 1800 1660 200 200  4.0300724 5.3947055}%
\special{pn 8}%
\special{pa 1270 1630}%
\special{pa 1140 1760}%
\special{fp}%
\special{pa 1280 1680}%
\special{pa 1170 1790}%
\special{fp}%
\special{pa 1260 1580}%
\special{pa 1130 1710}%
\special{fp}%
\special{pa 1230 1550}%
\special{pa 1120 1660}%
\special{fp}%
\special{pn 8}%
\special{pa 880 1640}%
\special{pa 750 1770}%
\special{dt 0.045}%
\special{pa 880 1700}%
\special{pa 790 1790}%
\special{dt 0.045}%
\special{pa 860 1600}%
\special{pa 730 1730}%
\special{dt 0.045}%
\special{pa 840 1560}%
\special{pa 720 1680}%
\special{dt 0.045}%
\special{pa 800 1540}%
\special{pa 730 1610}%
\special{dt 0.045}%
\put(12.0000,-19.0000){\makebox(0,0){\(D_1\)}}%
\put(8.0000,-19.0000){\makebox(0,0){\(- D_1\)}}%
\put(34.0000,-15.0000){\makebox(0,0){%
attaching two disks \(\pm D_1\) to \(M\)}}%
%
%
\end{picture}%
\end{center}
\caption{Decreasing the genus of \(M\).}
\label{FIG:capping off}
\end{figure}
we have a new diagram \(D'\) on 
a new surface \(M'\) with genus less than that of \(M\).
It is clear that \(D'\) represents \([D] + [D_1] - [D_1] = [c]\).

After repeating the above argument,
we can suppose that, for any essential curve \(C\) on \(M\),
\(\Pi(C)\) is not trivial. 
Since \(\Pi(C)\) has \(Q(K)\)-shadow colouring, 
Lemma \ref{LEM:shadow 1-diagram} says that \(\Pi(C)\) commutes with \(a\). 
We notice that \(\Pi(C)\) and \(a\) are elements of \(\pi(K)\), 
thus they are represented by some loops 
in the exterior \(E(K)\) of \(K\).
We denote the loops by the same symbols \(\Pi(C)\) and \(a\), 
respectively.
Choose a homotopy \(H\): \(\Pi(C) \cdot a \simeq a \cdot \Pi(C)\).
Clearly, \(H\) is an image of a torus in \(E(K)\), 
via some continuous map. 
Let \(\lambda\) be a longitude of \(K\) such 
that \(a\) and \(\lambda\) gives the peripheral system of \(K\).
By the fact that \(a\) bounds some meridian disk 
and the assumption that \(K\) is prime, 
we concludes that \(\Pi(C)\) is written in the form of
\(\lambda^i a^j\). \\

We can change the word presentation 
of \(\Pi(C)\) with preserving the homology class 
which \(D\) represents. 
By Lemma \ref{LEM:cobordism 1-diagram}, 
for two word presentations of \(\Pi(C)\), 
we have a shadow coloured diagram \(D_2\) on \(\sphere^1 \times I\). 
As drawn in Figure \ref{FIG:uniform presentation}, 
by cutting \(M\) along \(C\) and attaching 
two diagrams \(D_2\) and \(- D_2\), 
we obtain a new curve \(C'\) on a diagram 
where \(\Pi(C') = \Pi(C)\) is presented 
by another word. 
\begin{figure}[ht]
\begin{center}
\unitlength 0.1in%
\begin{picture}( 16.0000, 13.8500)(  7.0000,-18.0000)%
%
%
\special{pn 8}%
\special{pa 700 600}%
\special{pa 1100 600}%
\special{fp}%
\special{pn 8}%
\special{pa 700 1000}%
\special{pa 1100 1000}%
\special{fp}%
\special{pn 8}%
\special{ar 1100 800 100 200  4.7123890 6.2831853}%
\special{ar 1100 800 100 200  0.0000000 1.5707963}%
\special{ar 1100 800 100 200  1.5707963 1.6507963}%
\special{ar 1100 800 100 200  1.8907963 1.9707963}%
\special{ar 1100 800 100 200  2.2107963 2.2907963}%
\special{ar 1100 800 100 200  2.5307963 2.6107963}%
\special{ar 1100 800 100 200  2.8507963 2.9307963}%
\special{ar 1100 800 100 200  3.1707963 3.2507963}%
\special{ar 1100 800 100 200  3.4907963 3.5707963}%
\special{ar 1100 800 100 200  3.8107963 3.8907963}%
\special{ar 1100 800 100 200  4.1307963 4.2107963}%
\special{ar 1100 800 100 200  4.4507963 4.5307963}%
\put(11.0000,-5.0000){\makebox(0,0){\(- C\)}}%
\special{pn 8}%
\special{pa 2300 600}%
\special{pa 1900 600}%
\special{fp}%
\special{pn 8}%
\special{pa 2300 1000}%
\special{pa 1900 1000}%
\special{fp}%
\special{pn 8}%
\special{ar 1900 800 100 200  0.0000000 6.2831853}%
\put(19.0000,-5.0000){\makebox(0,0){\(C\)}}%
%
%
\special{pn 8}%
\special{pa 700 1400}%
\special{pa 1300 1400}%
\special{fp}%
\special{pn 8}%
\special{pa 700 1800}%
\special{pa 1300 1800}%
\special{fp}%
\special{pn 8}%
\special{ar 1100 1600 100 200  4.7123890 6.2831853}%
\special{ar 1100 1600 100 200  0.0000000 1.5707963}%
\special{ar 1100 1600 100 200  1.5707963 1.6507963}%
\special{ar 1100 1600 100 200  1.8907963 1.9707963}%
\special{ar 1100 1600 100 200  2.2107963 2.2907963}%
\special{ar 1100 1600 100 200  2.5307963 2.6107963}%
\special{ar 1100 1600 100 200  2.8507963 2.9307963}%
\special{ar 1100 1600 100 200  3.1707963 3.2507963}%
\special{ar 1100 1600 100 200  3.4907963 3.5707963}%
\special{ar 1100 1600 100 200  3.8107963 3.8907963}%
\special{ar 1100 1600 100 200  4.1307963 4.2107963}%
\special{ar 1100 1600 100 200  4.4507963 4.5307963}%
\special{pn 8}%
\special{pa 1380 1500}%
\special{pa 1190 1690}%
\special{fp}%
\special{pa 1390 1550}%
\special{pa 1140 1800}%
\special{fp}%
\special{pa 1400 1600}%
\special{pa 1200 1800}%
\special{fp}%
\special{pa 1390 1670}%
\special{pa 1260 1800}%
\special{fp}%
\special{pa 1360 1460}%
\special{pa 1200 1620}%
\special{fp}%
\special{pa 1340 1420}%
\special{pa 1200 1560}%
\special{fp}%
\special{pa 1300 1400}%
\special{pa 1190 1510}%
\special{fp}%
\special{pa 1240 1400}%
\special{pa 1170 1470}%
\special{fp}%
\special{pa 1180 1400}%
\special{pa 1150 1430}%
\special{fp}%
\special{pn 8}%
\special{ar 1300 1600 100 200  4.7123890 6.2831853}%
\special{ar 1300 1600 100 200  0.0000000 1.5707963}%
\put(12.0000,-13.0000){\makebox(0,0){\(- D_2\)}}%
\special{pn 8}%
\special{pa 2300 1400}%
\special{pa 1700 1400}%
\special{fp}%
\special{pn 8}%
\special{pa 2300 1800}%
\special{pa 1700 1800}%
\special{fp}%
\special{pn 8}%
\special{ar 1900 1600 100 200  4.7123890 6.2831853}%
\special{ar 1900 1600 100 200  0.0000000 1.5707963}%
\special{pn 8}%
\special{pa 1990 1550}%
\special{pa 1740 1800}%
\special{fp}%
\special{pa 2000 1600}%
\special{pa 1800 1800}%
\special{fp}%
\special{pa 1990 1670}%
\special{pa 1860 1800}%
\special{fp}%
\special{pa 1980 1500}%
\special{pa 1790 1690}%
\special{fp}%
\special{pa 1960 1460}%
\special{pa 1800 1620}%
\special{fp}%
\special{pa 1940 1420}%
\special{pa 1800 1560}%
\special{fp}%
\special{pa 1900 1400}%
\special{pa 1790 1510}%
\special{fp}%
\special{pa 1840 1400}%
\special{pa 1770 1470}%
\special{fp}%
\special{pa 1780 1400}%
\special{pa 1750 1430}%
\special{fp}%
\special{pn 8}%
\special{ar 1700 1600 100 200  0.0000000 6.2831853}%
\put(18.0000,-13.0000){\makebox(0,0){\(D_2\)}}%
%
%
\end{picture}%
\end{center}
\caption{Deforming word presentations.}
\label{FIG:uniform presentation}
\end{figure}
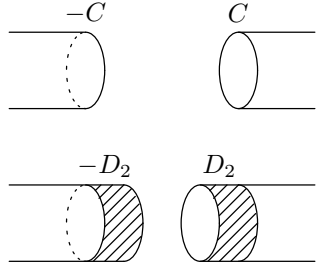
Thus we can suppose that the word presentation 
of \(\lambda\) is uniform. \\

Cut \(M\) into a \(2n\)-gon \(N\)
along \(n\) essential curves \(C_1\), \ldots, \(C_n\).
Obviously, \(\Pi(\partial N)\) is the product of 
\(\Pi(C_1)\), \ldots, \(\Pi(C_n)\) and 
\(\Pi(C_1)^{- 1}\), \ldots, \(\Pi(C_n)^{- 1}\) in some order. 
\begin{figure}[ht]
\begin{center}
\unitlength 0.1in%
\begin{picture}( 44.0000, 14.5000)( 1.8000,-16.0000)%
%
%
\special{pn 8}%
\special{pa 2200 1400}%
\special{pa 2200 1000}%
\special{dt 0.045}%
\special{pn 8}%
\special{pa 2200 1000}%
\special{pa 2000 600}%
\special{fp}%
\special{pn 4}%
\special{sh 1}%
\special{pa 2000 600}%
\special{pa 2012 670}%
\special{pa 2024 648}%
\special{pa 2048 652}%
\special{pa 2000 600}%
\special{fp}%
\special{pn 8}%
\special{pa 2000 600}%
\special{pa 1600 400}%
\special{fp}%
\special{pn 4}%
\special{sh 1}%
\special{pa 1600 400}%
\special{pa 1652 448}%
\special{pa 1648 424}%
\special{pa 1670 412}%
\special{pa 1600 400}%
\special{fp}%
\special{pn 8}%
\special{pa 1600 400}%
\special{pa 1200 400}%
\special{fp}%
\special{pn 4}%
\special{sh 1}%
\special{pa 1200 400}%
\special{pa 1268 420}%
\special{pa 1254 400}%
\special{pa 1268 380}%
\special{pa 1200 400}%
\special{fp}%
\special{pn 8}%
\special{pa 1200 400}%
\special{pa 800 600}%
\special{fp}%
\special{pn 4}%
\special{sh 1}%
\special{pa 800 600}%
\special{pa 870 588}%
\special{pa 848 576}%
\special{pa 852 552}%
\special{pa 800 600}%
\special{fp}%
\special{pn 8}%
\special{pa 800 600}%
\special{pa 600 1000}%
\special{fp}%
\special{pn 4}%
\special{sh 1}%
\special{pa 600 1000}%
\special{pa 648 950}%
\special{pa 624 952}%
\special{pa 612 932}%
\special{pa 600 1000}%
\special{fp}%
\special{pn 8}%
\special{pa 600 1000}%
\special{pa 600 1400}%
\special{dt 0.045}%
\put(15.0000,-3.0000){\makebox(0,0){\(\Pi(C_i)^{\pm 1}\)}}%
\put(7.9000,-4.4000){\makebox(0,0){\(\Pi(C_j)^{\pm 1}\)}}%
\put(4.6000,-7.4000){\makebox(0,0){\(\Pi(C_k)^{\pm 1}\)}}%
\put(14.0000,-11.0000){\makebox(0,0){\(N\)}}%
%
%
\special{pn 8}%
\special{pa 3800 600}%
\special{pa 2600 600}%
\special{fp}%
\special{pn 4}%
\special{sh 1}%
\special{pa 3800 600}%
\special{pa 3730 630}%
\special{pa 3740 600}%
\special{pa 3730 570}%
\special{pa 3800 600}%
\special{fp}%
\put(28.0000,-4.0000){\makebox(0,0){\fbox{\(a\)}}}%
\special{pn 8}%
\special{pa 3000 200}%
\special{pa 3000 600}%
\special{fp}%
\put(32.0000,-4.0000){\makebox(0,0){\(\cdots\)}}%
\special{pn 8}%
\special{pa 3400 200}%
\special{pa 3400 600}%
\special{fp}%
\put(36.0000,-4.0000){\makebox(0,0){\fbox{\(a\)}}}%
\put(32.0000,-7.5000){\makebox(0,0){%
\(\underbrace{\hskip40pt}_{\Pi(C) = \lambda^{\pm 1}}\)}}%
\put(42.0000,-4.0000){\makebox(0,0){\(\lambda^{\pm 1}\)-segment}}%
%
%
\special{pn 8}%
\special{pa 3800 1400}%
\special{pa 2600 1400}%
\special{fp}%
\special{pn 4}%
\special{sh 1}%
\special{pa 3800 1400}%
\special{pa 3730 1430}%
\special{pa 3740 1400}%
\special{pa 3730 1370}%
\special{pa 3800 1400}%
\special{fp}%
\put(27.5000,-12.0000){\makebox(0,0){\fbox{\(a\)}}}%
\special{pn 8}%
\special{pa 2900 1400}%
\special{pa 2900 1000}%
\special{fp}%
\put(30.5000,-12.0000){\makebox(0,0){\(\cdots\)}}%
\special{pn 8}%
\special{pa 3200 1000}%
\special{pa 3200 1400}%
\special{fp}%
\put(33.5000,-12.0000){\makebox(0,0){\fbox{\(a\)}}}%
\special{pn 8}%
\special{pa 3500 1000}%
\special{pa 3500 1400}%
\special{fp}%
\put(36.5000,-12.0000){\makebox(0,0){\fbox{\(a\)}}}%
\put(29.7000,-15.3000){\makebox(0,0){\(a^{\pm 1}\)}}%
\put(32.7000,-15.3000){\makebox(0,0){\(\cdots\)}}%
\put(35.7000,-15.3000){\makebox(0,0){\(a^{\pm 1}\)}}%
\put(42.0500,-12.0000){\makebox(0,0){\(a\)-segment}}%
%
%
\end{picture}%
\end{center}
\caption{Boundary of \(2n\)-gon \(N\).}
\label{FIG:expansion}
\end{figure}
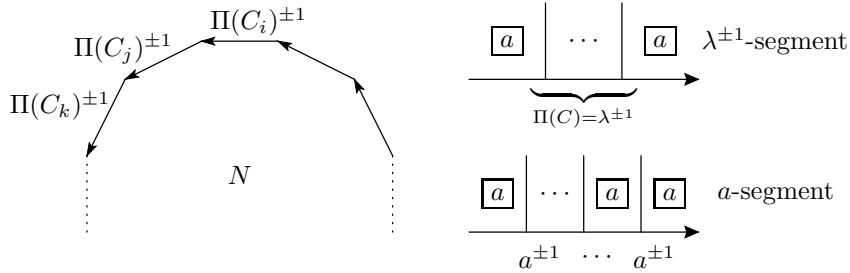
On the boundary \(\partial N\),
there exist the same numbers of \(\lambda\)- and \(\lambda^{- 1}\)-segments, 
and some \(a\)-segments as drawn in Figure \ref{FIG:expansion}. 
We will first remove the \(\lambda^{\pm 1}\)-segments from \(\partial N\).
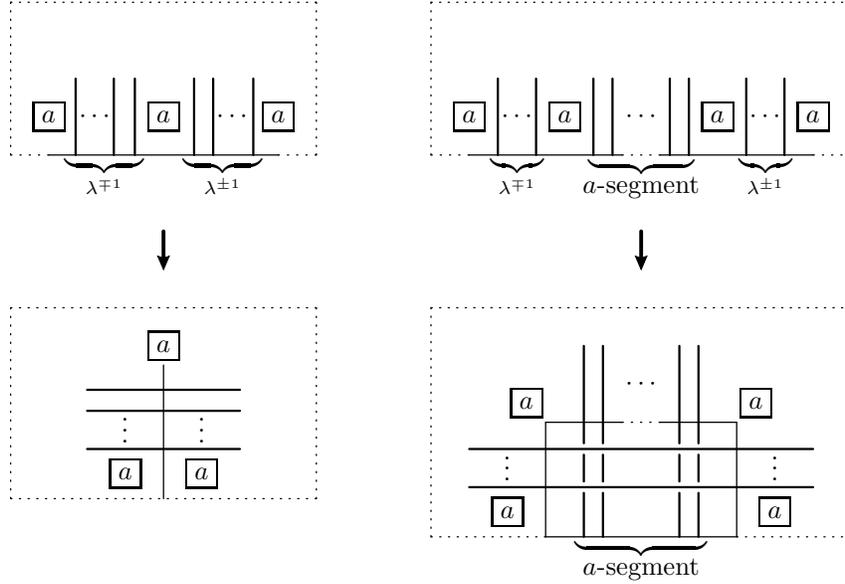
\begin{figure}[ht]
\begin{center}
\unitlength 0.1in%
\begin{picture}( 44.1000, 31.1000)( 5.9000,-34.4000)%
%
%
\special{pn 8}%
\special{pa 600 400}%
\special{pa 600 1200}%
\special{dt 0.045}%
\special{pn 8}%
\special{pa 600 1200}%
\special{pa 800 1200}%
\special{dt 0.045}%
\special{pn 8}%
\special{pa 800 1200}%
\special{pa 2000 1200}%
\special{fp}%
\special{pn 8}%
\special{pa 2000 1200}%
\special{pa 2200 1200}%
\special{dt 0.045}%
\special{pn 8}%
\special{pa 2200 1200}%
\special{pa 2200 400}%
\special{dt 0.045}%
\special{pn 8}%
\special{pa 2200 400}%
\special{pa 600 400}%
\special{dt 0.045}%
\put(8.0000,-10.0000){\makebox(0,0){\fbox{\(a\)}}}%
\special{pn 13}%
\special{pa 940 800}%
\special{pa 940 1200}%
\special{fp}%
\put(10.4000,-10.0000){\makebox(0,0){\(\cdots\)}}%
\special{pn 13}%
\special{pa 1140 800}%
\special{pa 1140 1200}%
\special{fp}%
\special{pn 13}%
\special{pa 1250 800}%
\special{pa 1250 1200}%
\special{fp}%
\put(14.0000,-10.0000){\makebox(0,0){\fbox{\(a\)}}}%
\special{pn 13}%
\special{pa 1560 800}%
\special{pa 1560 1200}%
\special{fp}%
\special{pn 13}%
\special{pa 1660 800}%
\special{pa 1660 1200}%
\special{fp}%
\put(17.6000,-10.0000){\makebox(0,0){\(\cdots\)}}%
\special{pn 13}%
\special{pa 1870 800}%
\special{pa 1870 1200}%
\special{fp}%
\put(20.0000,-10.0000){\makebox(0,0){\fbox{\(a\)}}}%
\put(10.9000,-13.1000){\makebox(0,0){%
\(\underbrace{\hskip30pt}_{\lambda^{\mp 1}}\)}}%
\put(17.1000,-13.1000){\makebox(0,0){%
\(\underbrace{\hskip30pt}_{\lambda^{\pm 1}}\)}}%
\special{pn 20}%
\special{pa 1400 1600}%
\special{pa 1400 1750}%
\special{fp}%
\special{pn 4}%
\special{sh 1}%
\special{pa 1400 1800}%
\special{pa 1430 1730}%
\special{pa 1400 1740}%
\special{pa 1370 1730}%
\special{pa 1400 1800}%
\special{fp}%
%
%
\special{pn 8}%
\special{pa 600 2000}%
\special{pa 600 3000}%
\special{dt 0.045}%
\special{pn 8}%
\special{pa 600 3000}%
\special{pa 2200 3000}%
\special{dt 0.045}%
\special{pn 8}%
\special{pa 2200 3000}%
\special{pa 2200 2000}%
\special{dt 0.045}%
\special{pn 8}%
\special{pa 2200 2000}%
\special{pa 600 2000}%
\special{dt 0.045}%
\special{pn 8}%
\special{pa 1400 2300}%
\special{pa 1400 3000}%
\special{fp}%
\put(14.0000,-22.0000){\makebox(0,0){\fbox{\(a\)}}}%
\special{pn 13}%
\special{pa 1000 2430}%
\special{pa 1800 2430}%
\special{fp}%
\special{pn 13}%
\special{pa 1000 2540}%
\special{pa 1800 2540}%
\special{fp}%
\put(12.0000,-26.0000){\makebox(0,0){\(\vdots\)}}%
\put(16.0000,-26.0000){\makebox(0,0){\(\vdots\)}}%
\special{pn 13}%
\special{pa 1000 2740}%
\special{pa 1800 2740}%
\special{fp}%
\put(12.0000,-28.7000){\makebox(0,0){\fbox{\(a\)}}}%
\put(16.0000,-28.7000){\makebox(0,0){\fbox{\(a\)}}}%
%
%
\special{pn 8}%
\special{pa 2800 400}%
\special{pa 2800 1200}%
\special{dt 0.045}%
\special{pn 8}%
\special{pa 2800 1200}%
\special{pa 3000 1200}%
\special{dt 0.045}%
\special{pn 8}%
\special{pa 3000 1200}%
\special{pa 3800 1200}%
\special{fp}%
\special{pn 8}%
\special{pa 3800 1200}%
\special{pa 4000 1200}%
\special{dt 0.045}%
\special{pn 8}%
\special{pa 4000 1200}%
\special{pa 4800 1200}%
\special{fp}%
\special{pn 8}%
\special{pa 4800 1200}%
\special{pa 5000 1200}%
\special{dt 0.045}%
\special{pn 8}%
\special{pa 5000 1200}%
\special{pa 5000 400}%
\special{dt 0.045}%
\special{pn 8}%
\special{pa 5000 400}%
\special{pa 2800 400}%
\special{dt 0.045}%
\put(30.0000,-10.0000){\makebox(0,0){\fbox{\(a\)}}}%
\special{pn 13}%
\special{pa 3150 800}%
\special{pa 3150 1200}%
\special{fp}%
\put(32.5000,-10.0000){\makebox(0,0){\(\cdots\)}}%
\special{pn 13}%
\special{pa 3350 800}%
\special{pa 3350 1200}%
\special{fp}%
\put(35.0000,-10.0000){\makebox(0,0){\fbox{\(a\)}}}%
\special{pn 13}%
\special{pa 3650 800}%
\special{pa 3650 1200}%
\special{fp}%
\special{pn 13}%
\special{pa 3750 800}%
\special{pa 3750 1200}%
\special{fp}%
\put(39.0000,-10.0000){\makebox(0,0){\(\cdots\)}}%
\special{pn 13}%
\special{pa 4050 800}%
\special{pa 4050 1200}%
\special{fp}%
\special{pn 13}%
\special{pa 4150 800}%
\special{pa 4150 1200}%
\special{fp}%
\put(43.0000,-10.0000){\makebox(0,0){\fbox{\(a\)}}}%
\special{pn 13}%
\special{pa 4450 800}%
\special{pa 4450 1200}%
\special{fp}%
\put(45.5000,-10.0000){\makebox(0,0){\(\cdots\)}}%
\special{pn 13}%
\special{pa 4650 800}%
\special{pa 4650 1200}%
\special{fp}%
\put(48.0000,-10.0000){\makebox(0,0){\fbox{\(a\)}}}%
\put(32.5000,-13.1000){\makebox(0,0){%
\(\underbrace{\hskip20pt}_{\lambda^{\mp 1}}\)}}%
\put(45.5000,-13.1000){\makebox(0,0){%
\(\underbrace{\hskip20pt}_{\lambda^{\pm 1}}\)}}%
\put(39.0000,-12.6000){\makebox(0,0){\(\underbrace{\hskip40pt}\)}}%
\put(39.0000,-13.7000){\makebox(0,0){\(a\)-segment}}%
\special{pn 20}%
\special{pa 3900 1600}%
\special{pa 3900 1750}%
\special{fp}%
\special{pn 4}%
\special{sh 1}%
\special{pa 3900 1800}%
\special{pa 3930 1730}%
\special{pa 3900 1740}%
\special{pa 3870 1730}%
\special{pa 3900 1800}%
\special{fp}%
%
%
\special{pn 8}%
\special{pa 2800 2000}%
\special{pa 2800 3200}%
\special{dt 0.045}%
\special{pn 8}%
\special{pa 2800 3200}%
\special{pa 3400 3200}%
\special{dt 0.045}%
\special{pn 8}%
\special{pa 3400 3200}%
\special{pa 4400 3200}%
\special{fp}%
\special{pn 8}%
\special{pa 4400 3200}%
\special{pa 5000 3200}%
\special{dt 0.045}%
\special{pn 8}%
\special{pa 5000 3200}%
\special{pa 5000 2000}%
\special{dt 0.045}%
\special{pn 8}%
\special{pa 5000 2000}%
\special{pa 2800 2000}%
\special{dt 0.045}%
\put(33.0000,-25.0000){\makebox(0,0){\fbox{\(a\)}}}%
\put(45.0000,-25.0000){\makebox(0,0){\fbox{\(a\)}}}%
\special{pn 8}%
\special{pa 3400 3200}%
\special{pa 3400 2600}%
\special{fp}%
\special{pn 8}%
\special{pa 3400 2600}%
\special{pa 3800 2600}%
\special{fp}%
\special{pn 8}%
\special{pa 3800 2600}%
\special{pa 4000 2600}%
\special{dt 0.045}%
\special{pn 8}%
\special{pa 4000 2600}%
\special{pa 4400 2600}%
\special{fp}%
\special{pn 8}%
\special{pa 4400 2600}%
\special{pa 4400 3200}%
\special{fp}%
\special{pn 13}%
\special{pa 3000 2740}%
\special{pa 4800 2740}%
\special{fp}%
\put(32.0000,-28.0000){\makebox(0,0){\(\vdots\)}}%
\put(46.0000,-28.0000){\makebox(0,0){\(\vdots\)}}%
\special{pn 13}%
\special{pa 3000 2940}%
\special{pa 4800 2940}%
\special{fp}%
\put(32.0000,-30.7000){\makebox(0,0){\fbox{\(a\)}}}%
\put(46.0000,-30.7000){\makebox(0,0){\fbox{\(a\)}}}%
\special{pn 13}%
\special{pa 3600 2200}%
\special{pa 3600 2710}%
\special{fp}%
\special{pn 13}%
\special{pa 3600 2770}%
\special{pa 3600 2910}%
\special{fp}%
\special{pn 13}%
\special{pa 3600 2970}%
\special{pa 3600 3200}%
\special{fp}%
\special{pn 13}%
\special{pa 3700 2200}%
\special{pa 3700 2710}%
\special{fp}%
\special{pn 13}%
\special{pa 3700 2770}%
\special{pa 3700 2910}%
\special{fp}%
\special{pn 13}%
\special{pa 3700 2970}%
\special{pa 3700 3200}%
\special{fp}%
\put(39.0000,-24.0000){\makebox(0,0){\(\cdots\)}}%
\special{pn 13}%
\special{pa 4100 2200}%
\special{pa 4100 2710}%
\special{fp}%
\special{pn 13}%
\special{pa 4100 2770}%
\special{pa 4100 2910}%
\special{fp}%
\special{pn 13}%
\special{pa 4100 2970}%
\special{pa 4100 3200}%
\special{fp}%
\special{pn 13}%
\special{pa 4200 2200}%
\special{pa 4200 2710}%
\special{fp}%
\special{pn 13}%
\special{pa 4200 2770}%
\special{pa 4200 2910}%
\special{fp}%
\special{pn 13}%
\special{pa 4200 2970}%
\special{pa 4200 3200}%
\special{fp}%
\put(39.0000,-32.6000){\makebox(0,0){\(\underbrace{\hskip50pt}\)}}%
\put(39.0000,-33.7000){\makebox(0,0){\(a\)-segment}}%
%
%
\end{picture}%
\end{center}
\caption{Vanishing \(\lambda^{\pm 1}\)-segments.}
\label{FIG:lambda surgery}
\end{figure}
When \(\lambda\)- and \(\lambda^{- 1}\)-segments 
are adjacent, we easily connect them together.
This surgery is shown in the left two of Figure \ref{FIG:lambda surgery}.
If an \(a\)-segment exists between \(\lambda^{\pm 1}\)-segments, 
we attach a checker-board diagram there, 
which is depicted in the right two of the same figure. 
The latter surgery does not preserve the rack homology classes. 
However, since the checker-board diagram represents 
a degeneracy \(3\)-chain, 
the quandle homology classes of the diagrams are invariant 
under this surgery. 

Therefore, by repeating those surgeries, 
we have a new polygon \(N'\) such that 
there are only \(a\)-segments on \(\partial N'\). 
Since \(\Pi(\partial N') = \Pi(\partial N) = e\) holds 
and since all the endpoints of arcs are coloured with \(a\), 
there are the same numbers of points with signature \(+ 1\) 
and with \(- 1\). 
Connecting adjacent points with opposite signatures 
by an arc make the diagram \(N'\) simpler. 
After finitely many times of surgeries as such, 
a diagram on \(\disk^2\) with trivially coloured boundary is obtained. 
Finally, by capping off the diagram, 
we have a diagram \(\tilde{D}\) on \(\sphere^2\) 
such that a \(3\)-cycle \(\langle \tilde{D} \rangle\) 
is quandle homologous to \(c\). 

The diagram \(\tilde{D}\) can be regarded as 
a diagram of a link \(L\). 
Obviously, a \(Q(K)\)-colouring of a link diagram of \(L\) 
gives a homomorphism \(f\): \(Q(L) \to Q(K)\), 
which completes the proof.
\end{proof}

\subsection{Examples of computation.}
\label{SEC:computation}

Since our motivation is 
to determine the third quandle homology group 
of knot quandle \(Q(K)\), 
it is natural to consider whether the quotient 
\(H^\qdl_3(Q(K)) / \Z[K_\sh]\) is trivial or not. 

By Theorem \ref{TH:topological realisation}, 
for a prime knot \(K\), 
it is sufficient to consider the images 
of the shadow fundamental classes of some links 
via homomorphisms between knot quandles. 

Regrettably, we could not found any elements of 
\(H^\qdl_3(Q(K)) / \Z[K_\sh]\). 
Thus, further research is expected. \\

We show two examples of computations. 
Let \(K\) be a knot \(4_1\). 
The homomorphisms which we use are 
due to the tables in [KS]. 
At the first, we compute 
the shadow fundamental class of \(K\). 
\begin{figure}[ht]
\begin{center}
\unitlength 0.1in%
\begin{picture}( 18.0000, 15.0000)( 12.6000,-17.0000)%
\special{pn 8}%
\special{pa 2570 710}%
\special{pa 2570 850}%
\special{fp}%
\special{pn 8}%
\special{pa 2570 850}%
\special{pa 2570 856}%
\special{pa 2570 860}%
\special{pa 2568 866}%
\special{pa 2568 870}%
\special{pa 2566 876}%
\special{pa 2562 880}%
\special{pa 2560 886}%
\special{pa 2556 890}%
\special{pa 2554 896}%
\special{pa 2550 900}%
\special{pa 2544 906}%
\special{pa 2540 910}%
\special{pa 2534 916}%
\special{pa 2530 920}%
\special{pa 2524 926}%
\special{pa 2516 930}%
\special{pa 2510 936}%
\special{pa 2504 940}%
\special{pa 2496 946}%
\special{pa 2488 950}%
\special{pa 2480 956}%
\special{pa 2472 960}%
\special{pa 2462 966}%
\special{pa 2454 970}%
\special{pa 2444 976}%
\special{pa 2434 980}%
\special{pa 2424 986}%
\special{pa 2414 990}%
\special{pa 2404 996}%
\special{pa 2392 1000}%
\special{pa 2382 1006}%
\special{pa 2370 1010}%
\special{pa 2360 1016}%
\special{pa 2348 1020}%
\special{pa 2336 1026}%
\special{pa 2324 1030}%
\special{pa 2312 1036}%
\special{pa 2300 1040}%
\special{pa 2286 1046}%
\special{pa 2274 1050}%
\special{pa 2262 1056}%
\special{pa 2248 1060}%
\special{pa 2236 1066}%
\special{pa 2222 1070}%
\special{pa 2210 1076}%
\special{pa 2196 1080}%
\special{pa 2184 1086}%
\special{pa 2170 1090}%
\special{pa 2158 1096}%
\special{pa 2144 1100}%
\special{pa 2132 1106}%
\special{pa 2118 1110}%
\special{pa 2106 1116}%
\special{pa 2092 1120}%
\special{pa 2080 1126}%
\special{pa 2066 1130}%
\special{pa 2054 1136}%
\special{pa 2042 1140}%
\special{pa 2030 1146}%
\special{pa 2018 1150}%
\special{pa 2006 1156}%
\special{pa 1994 1160}%
\special{pa 1982 1166}%
\special{pa 1970 1170}%
\special{pa 1960 1176}%
\special{pa 1948 1180}%
\special{pa 1938 1186}%
\special{pa 1926 1190}%
\special{pa 1916 1196}%
\special{pa 1906 1200}%
\special{pa 1898 1206}%
\special{pa 1888 1210}%
\special{pa 1878 1216}%
\special{pa 1870 1220}%
\special{pa 1862 1226}%
\special{pa 1854 1230}%
\special{pa 1846 1236}%
\special{pa 1838 1240}%
\special{pa 1830 1246}%
\special{pa 1824 1250}%
\special{pa 1818 1256}%
\special{pa 1812 1260}%
\special{pa 1806 1266}%
\special{pa 1800 1270}%
\special{pa 1796 1276}%
\special{pa 1792 1280}%
\special{pa 1788 1286}%
\special{pa 1784 1290}%
\special{pa 1780 1296}%
\special{pa 1778 1300}%
\special{pa 1776 1306}%
\special{pa 1774 1310}%
\special{pa 1772 1316}%
\special{pa 1772 1320}%
\special{pa 1770 1326}%
\special{pa 1770 1330}%
\special{sp}%
\special{pn 8}%
\special{ar 2570 1330 800 300  2.3007963 3.1415927}%
\special{ar 2570 1330 800 300  1.5707963 1.9415927}%
\special{pn 8}%
\special{ar 2570 1145 450 485  0.0000000 1.5707963}%
\special{ar 2570 1145 450 485  4.7123890 6.2831853}%
\special{pn 8}%
\special{pa 2570 658}%
\special{pa 1820 658}%
\special{fp}%
\special{pn 8}%
\special{ar 1770 1145 450 485  1.5707963 4.5823890}%
\special{pn 8}%
\special{ar 1770 1330 800 300  0.0000000 1.5707963}%
\special{pn 8}%
\special{pa 1770 850}%
\special{pa 1770 856}%
\special{pa 1772 860}%
\special{pa 1772 866}%
\special{pa 1774 870}%
\special{pa 1776 876}%
\special{pa 1778 880}%
\special{pa 1780 886}%
\special{pa 1784 890}%
\special{pa 1788 896}%
\special{pa 1792 900}%
\special{pa 1796 906}%
\special{pa 1800 910}%
\special{pa 1806 916}%
\special{pa 1812 920}%
\special{pa 1818 926}%
\special{pa 1824 930}%
\special{pa 1830 936}%
\special{pa 1838 940}%
\special{pa 1846 946}%
\special{pa 1854 950}%
\special{pa 1862 956}%
\special{pa 1870 960}%
\special{pa 1878 966}%
\special{pa 1888 970}%
\special{pa 1898 976}%
\special{pa 1906 980}%
\special{pa 1916 986}%
\special{pa 1926 990}%
\special{pa 1938 996}%
\special{pa 1948 1000}%
\special{pa 1960 1006}%
\special{pa 1970 1010}%
\special{pa 1982 1016}%
\special{pa 1994 1020}%
\special{pa 2006 1026}%
\special{pa 2018 1030}%
\special{pa 2030 1036}%
\special{pa 2042 1040}%
\special{pa 2054 1046}%
\special{pa 2066 1050}%
\special{pa 2080 1056}%
\special{pa 2092 1060}%
\special{pa 2106 1066}%
\special{sp}%
\special{pa 2236 1116}%
\special{pa 2248 1120}%
\special{pa 2262 1126}%
\special{pa 2274 1130}%
\special{pa 2286 1136}%
\special{pa 2300 1140}%
\special{pa 2312 1146}%
\special{pa 2324 1150}%
\special{pa 2336 1156}%
\special{pa 2348 1160}%
\special{pa 2360 1166}%
\special{pa 2370 1170}%
\special{pa 2382 1176}%
\special{pa 2392 1180}%
\special{pa 2404 1186}%
\special{pa 2414 1190}%
\special{pa 2424 1196}%
\special{pa 2434 1200}%
\special{pa 2444 1206}%
\special{pa 2454 1210}%
\special{pa 2462 1216}%
\special{pa 2472 1220}%
\special{pa 2480 1226}%
\special{pa 2488 1230}%
\special{pa 2496 1236}%
\special{pa 2504 1240}%
\special{pa 2510 1246}%
\special{pa 2516 1250}%
\special{pa 2524 1256}%
\special{pa 2530 1260}%
\special{pa 2534 1266}%
\special{pa 2540 1270}%
\special{pa 2544 1276}%
\special{pa 2550 1280}%
\special{pa 2554 1286}%
\special{pa 2556 1290}%
\special{pa 2560 1296}%
\special{pa 2562 1300}%
\special{pa 2566 1306}%
\special{pa 2568 1310}%
\special{pa 2568 1316}%
\special{pa 2570 1320}%
\special{pa 2570 1326}%
\special{pa 2570 1330}%
\special{sp}%
\special{pn 8}%
\special{pa 1770 850}%
\special{pa 1770 660}%
\special{fp}%
\special{pn 8}%
\special{ar 2170 660 400 360  3.1415927 6.1031853}%
\special{pn 4}%
\special{sh 1}%
\special{pa 2140 300}%
\special{pa 2210 270}%
\special{pa 2200 300}%
\special{pa 2210 330}%
\special{pa 2140 300}%
\special{fp}%
\put(17.1000,-13.4300){\makebox(0,0){\(1\)}}%
\put(25.6000,-3.8700){\makebox(0,0){\(2\)}}%
\put(14.0000,-15.4000){\makebox(0,0){\(3\)}}%
\put(30.1000,-8.0700){\makebox(0,0){\(4\)}}%
\put(16.1000,-10.3000){\makebox(0,0){\fbox{\(1\)}}}%
\put(21.7000,-13.5000){\makebox(0,0){\fbox{\(1\)}}}%
\put(13.5000,-6.8000){\makebox(0,0){\fbox{\(4\)}}}%
\put(27.3000,-10.3000){\makebox(0,0){\fbox{\(4\)}}}%
\put(21.7000,-8.3000){\makebox(0,0){\fbox{\(4 \lhd 1\)}}}%
\put(21.7000,-5.0000){\makebox(0,0){\fbox{\(4 \lhd 2\)}}}%
\end{picture}%
\end{center}
\caption{Shadow coloured diagram of \(4_1\).}
\label{FIG:knot 4-1}
\end{figure}
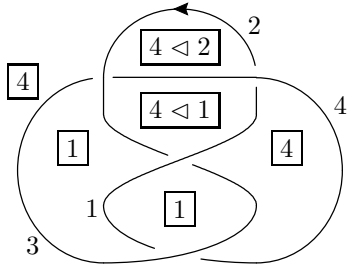
By Figure \ref{FIG:knot 4-1}, we obtain 
\begin{quote}
\([K_\sh] = (4, 1, 4) - (1, 3, 1) 
- (1, 1, 3) + (1, 3, 2) \\
= (4, 1, 4) - (1, 3, 1) + (1, 3, 2)\). 
\end{quote}

\begin{example}
\label{EX:knot 9-37}
Let \(L\) be a knot \(9_{37}\).
We have a surjective homomorphism \(f\) 
from \(Q(L)\) to \(Q(K)\) defined by 
\begin{quote}
\(1 \mapsto 2\), 
\(2 \mapsto 3\), 
\(3 \mapsto 4 \blhd 1\), 
\(4 \mapsto 3\), 
\(5 \mapsto 1\), \\
\(6 \mapsto 4 \lhd 1\), 
\(7 \mapsto 4\), 
\(8 \mapsto 1\), 
\(9 \mapsto 4\).
\end{quote}
The induced shadow colouring of \(L\) is 
shown in Figure \ref{FIG:knot 9-37}. 
We disregard the shadow colours which are 
not used in the calculus below. 
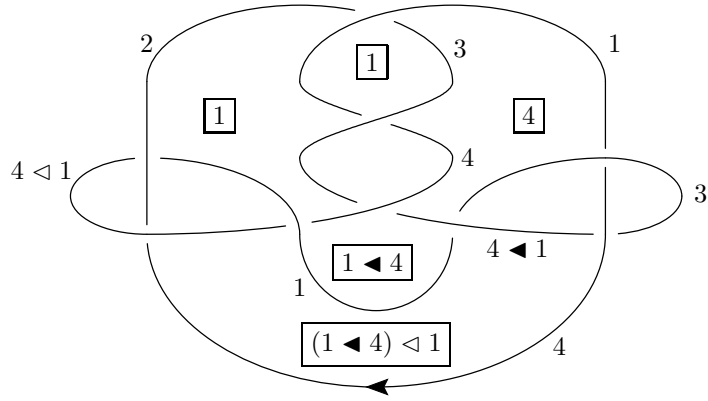
\begin{figure}[ht]
\begin{center}
\unitlength 0.1in%
\begin{picture}( 36.4000, 21.1500)(  4.8000,-22.4500)%
\special{pn 8}%
\special{ar 2800 600 800 400  3.1415927 6.2831853}%
\special{pn 8}%
\special{pa 3600 600}%
\special{pa 3600 950}%
\special{fp}%
\special{pn 8}%
\special{pa 3600 1050}%
\special{pa 3600 1400}%
\special{fp}%
\special{pn 8}%
\special{ar 2400 1400 1200 800  0.0000000 3.0800000}
\special{pn 8}%
\special{pa 1200 1350}%
\special{pa 1200 600}%
\special{fp}%
\special{pn 8}%
\special{ar 2000 600 800 400  3.1415927 5.0800000}%
\special{ar 2000 600 800 400  5.3800000 6.2831853}%
\special{pn 8}%
\special{pa 2800 600}%
\special{pa 2800 606}%
\special{pa 2800 610}%
\special{pa 2798 616}%
\special{pa 2796 620}%
\special{pa 2792 626}%
\special{pa 2790 630}%
\special{pa 2786 636}%
\special{pa 2780 640}%
\special{pa 2776 646}%
\special{pa 2770 650}%
\special{pa 2764 656}%
\special{pa 2756 660}%
\special{pa 2750 666}%
\special{pa 2742 670}%
\special{pa 2734 676}%
\special{pa 2724 680}%
\special{pa 2714 686}%
\special{pa 2704 690}%
\special{pa 2694 696}%
\special{pa 2684 700}%
\special{pa 2672 706}%
\special{pa 2660 710}%
\special{pa 2648 716}%
\special{pa 2636 720}%
\special{pa 2622 726}%
\special{pa 2610 730}%
\special{pa 2596 736}%
\special{pa 2582 740}%
\special{pa 2568 746}%
\special{pa 2554 750}%
\special{pa 2538 756}%
\special{pa 2524 760}%
\special{pa 2510 766}%
\special{pa 2494 770}%
\special{pa 2478 776}%
\special{pa 2464 780}%
\special{pa 2448 786}%
\special{pa 2432 790}%
\special{pa 2416 796}%
\special{pa 2400 800}%
\special{pa 2384 806}%
\special{pa 2370 810}%
\special{pa 2354 816}%
\special{pa 2338 820}%
\special{pa 2322 826}%
\special{pa 2308 830}%
\special{pa 2292 836}%
\special{pa 2276 840}%
\special{pa 2262 846}%
\special{pa 2248 850}%
\special{pa 2234 856}%
\special{pa 2218 860}%
\special{pa 2206 866}%
\special{pa 2192 870}%
\special{pa 2178 876}%
\special{pa 2166 880}%
\special{pa 2152 886}%
\special{pa 2140 890}%
\special{pa 2128 896}%
\special{pa 2118 900}%
\special{pa 2106 906}%
\special{pa 2096 910}%
\special{pa 2086 916}%
\special{pa 2076 920}%
\special{pa 2068 926}%
\special{pa 2060 930}%
\special{pa 2052 936}%
\special{pa 2044 940}%
\special{pa 2038 946}%
\special{pa 2030 950}%
\special{pa 2026 956}%
\special{pa 2020 960}%
\special{pa 2016 966}%
\special{pa 2012 970}%
\special{pa 2008 976}%
\special{pa 2006 980}%
\special{pa 2004 986}%
\special{pa 2002 990}%
\special{pa 2000 996}%
\special{pa 2000 1000}%
\special{sp}%
\special{pn 8}%
\special{ar 3600 1000 1600 400  2.5207963 3.1415927}%
\special{ar 3600 1000 1600 400  1.6107963 2.3207963}%
\special{pn 8}%
\special{ar 3600 1200 400 200  0.0000000 1.4007963}%
\special{ar 3600 1200 400 200  4.7123890 6.2831853}%
\special{pn 8}%
\special{ar 3600 1400 800 400  3.4500000 4.7123890}%
\special{pn 8}%
\special{ar 2400 1400 400 400  0.0500000 3.1415927}%
\special{pn 8}%
\special{ar 1200 1400 800 400  4.8023890 6.2831853}%
\special{pn 8}%
\special{ar 1200 1200 400 200  1.5707963 4.5500000}%
\special{pn 8}%
\special{ar 1200 1000 1600 400  1.1000000 1.5707963}%
\special{ar 1200 1000 1600 400  0.0000000 1.0000000}%
\special{pn 8}%
\special{pa 2800 1000}%
\special{pa 2800 996}%
\special{pa 2800 990}%
\special{pa 2798 986}%
\special{pa 2796 980}%
\special{pa 2792 976}%
\special{pa 2790 970}%
\special{pa 2786 966}%
\special{pa 2780 960}%
\special{pa 2776 956}%
\special{pa 2770 950}%
\special{pa 2764 946}%
\special{pa 2756 940}%
\special{pa 2750 936}%
\special{pa 2742 930}%
\special{pa 2734 926}%
\special{pa 2724 920}%
\special{pa 2714 916}%
\special{pa 2704 910}%
\special{pa 2694 906}%
\special{pa 2684 900}%
\special{pa 2672 896}%
\special{pa 2660 890}%
\special{pa 2648 886}%
\special{pa 2636 880}%
\special{pa 2622 876}%
\special{pa 2610 870}%
\special{pa 2596 866}%
\special{pa 2582 860}%
\special{pa 2568 856}%
\special{pa 2554 850}%
\special{pa 2538 846}%
\special{pa 2524 840}%
\special{pa 2510 836}%
\special{pa 2494 830}%
\special{pa 2478 826}%
\special{sp}%
\special{pa 2322 776}%
\special{pa 2308 770}%
\special{pa 2292 766}%
\special{pa 2276 760}%
\special{pa 2262 756}%
\special{pa 2248 750}%
\special{pa 2234 746}%
\special{pa 2218 740}%
\special{pa 2206 736}%
\special{pa 2192 730}%
\special{pa 2178 726}%
\special{pa 2166 720}%
\special{pa 2152 716}%
\special{pa 2140 710}%
\special{pa 2128 706}%
\special{pa 2118 700}%
\special{pa 2106 696}%
\special{pa 2096 690}%
\special{pa 2086 686}%
\special{pa 2076 680}%
\special{pa 2068 676}%
\special{pa 2060 670}%
\special{pa 2052 666}%
\special{pa 2044 660}%
\special{pa 2038 656}%
\special{pa 2030 650}%
\special{pa 2026 646}%
\special{pa 2020 640}%
\special{pa 2016 636}%
\special{pa 2012 630}%
\special{pa 2008 626}%
\special{pa 2006 620}%
\special{pa 2004 616}%
\special{pa 2002 610}%
\special{pa 2000 606}%
\special{pa 2000 600}%
\special{sp}%
\special{pn 4}%
\special{sh 1}%
\special{pa 2350 2200}%
\special{pa 2460 2240}%
\special{pa 2440 2200}%
\special{pa 2460 2160}%
\special{pa 2350 2200}%
\special{fp}%
\put(36.5000,-4.0000){\makebox(0,0){\(1\)}}%
\put(20.0000,-16.8000){\makebox(0,0){\(1\)}}%
\put(12.0000,-4.0000){\makebox(0,0){\(2\)}}%
\put(28.4000,-4.3000){\makebox(0,0){\(3\)}}%
\put(41.0000,-11.9000){\makebox(0,0){\(3\)}}%
\put(28.8000,-10.0000){\makebox(0,0){\(4\)}}%
\put(33.6000,-19.9000){\makebox(0,0){\(4\)}}%
\put(31.4000,-14.8000){\makebox(0,0){\(4 \blhd 1\)}}%
\put(6.5000,-10.7000){\makebox(0,0){\(4 \lhd 1\)}}%
\put(15.8000,-7.8000){\makebox(0,0){\fbox{\(1\)}}}%
\put(32.0000,-7.8000){\makebox(0,0){\fbox{\(4\)}}}%
\put(23.8000,-5.0000){\makebox(0,0){\fbox{\(1\)}}}%
\put(23.8000,-15.5000){\makebox(0,0){\fbox{\(1 \blhd 4\)}}}%
\put(24.0000,-19.8000){\makebox(0,0){\fbox{\((1 \blhd 4) \lhd 1\)}}}%
\end{picture}%
\end{center}
\caption{Shadow coloured diagram of \(9_{37}\).}
\label{FIG:knot 9-37}
\end{figure}
Now we have the image of \([L_\sh]\) via \(f_\ast\) 
as follows: 
\begin{quote}
\(f_\ast [L_\sh] \\
\hskip15pt = - (1, 3, 1) - (1, 1, 3) 
- (1 \blhd 4, 4 \blhd 1, 4) \\
\hskip26pt + (4, 1, 3) 
+ ((1 \blhd 4) \lhd 1, 4 \blhd 1, 4) 
- (1 \blhd 4, 1, 4 \blhd 1) \\
\hskip26pt + (1, 1, 2) 
+ ((1 \blhd 4) \lhd 1, 4, 4 \lhd 1) 
- (1 \blhd 4, 4, 1) \\
\hskip15pt = - (1, 3, 1) + (4, 1, 4) + (1, 3, 2) \\
\hskip26pt - (4, 1, 4) - (1, 3, 2) - (1 \blhd 4, 4 \blhd 1, 4) \\
\hskip26pt + (4, 1, 3) 
+ ((1 \blhd 4) \lhd 1, 4 \blhd 1, 4) 
- (1 \blhd 4, 1, 4 \blhd 1) \\
\hskip26pt + (1, 1, 2) 
+ ((1 \blhd 4) \lhd 1, 4, 4 \lhd 1) 
- (1 \blhd 4, 4, 1) \\
\hskip15pt = [K_\sh] \\
\hskip26pt - ((1 \blhd 4) \lhd 1, 4, 4) - (1 \blhd 4, 4 \blhd 1, 1) 
- (1 \blhd 4, 1, 4) \\
\hskip26pt + (4, 1, 3) 
+ ((1 \blhd 4) \lhd 1, 4 \blhd 1, 4) 
- (1 \blhd 4, 1, 4 \blhd 1) \\
\hskip26pt + (1, 1, 2) 
+ ((1 \blhd 4) \lhd 1, 4, 4 \lhd 1) 
- (1 \blhd 4, 4, 1) \\
\hskip15pt = [K_\sh] \\
\hskip26pt + \{ - (1 \blhd 4, 4 \blhd 1, 1) 
+ (4, 1, 3) \\
\hskip52pt + ((1 \blhd 4) \lhd 1, 4 \blhd 1, 4) 
- (1 \blhd 4, 1, 4 \blhd 1)\} \\
\hskip26pt + \{ - (1 \blhd 4, 1, 4) 
+ (1, 1, 2) \\
\hskip52pt + ((1 \blhd 4) \lhd 1, 4, 4 \lhd 1) 
- (1 \blhd 4, 4, 1)\} = 3[K_\sh]\).
\end{quote}
\end{example}

\begin{example}
\label{EX:knot 10-59}
Let \(L\) be a knot \(10_{59}\).
A homomorphism \(f\): \(Q(L) \to Q(K)\) defined by
\begin{quote}
\(1 \mapsto 1\), 
\(2 \mapsto 1\), 
\(3 \mapsto 4\), 
\(4 \mapsto 1\), 
\(5 \mapsto 1\), \\
\(6 \mapsto 3\), 
\(7 \mapsto 4 \blhd 1\), 
\(8 \mapsto 4\), 
\(9 \mapsto 3\), 
\(10 \mapsto 2\)
\end{quote}
is surjective. 
Figure \ref{FIG:knot 10-59} shows the induced shadow 
colouring of \(L\). 
In this case, we have \(f_\ast[L_\sh] = - 2 [K_\sh]\).
\begin{figure}[ht]
\begin{center}
\input{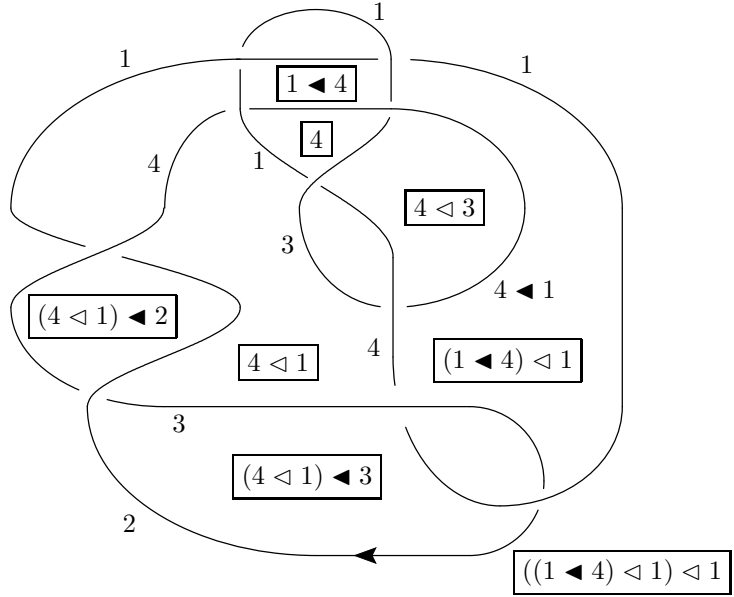}
\end{center}
\caption{Shadow coloured diagram of \(10_{59}\).}
\label{FIG:knot 10-59}
\end{figure}
\end{example}

\section*{Acknowledgement.}

The author would like to express his sincere gratitude
to Prof.~Toshitake Kohno for his invaluable suggestions.
Also he would like to thank Kokoro Tanaka and Ken Kanno 
for their helpful comments.

The author is financially supported by COE of the University of Tokyo.

\section*{References.}

\begin{itemize}
\item[{[BM]}] Burde, G., Murasugi, K.,
Links and Seifert Fiber Spaces,
Duke Math.~J.~\textbf{37} (1970), 89-93.

\item[{[CEGS]}] Carter, J. S., Elhamdadi, M., Gra\~na, M., Saito, M.,
Cocycle Knot Invariants from Quandle Modules 
and Generalized Quandle Cohomology,
Osaka J.~Math. \textbf{42} (2005), 499-541

\item[{[CENS]}] Carter, J. S., Elhamdadi, M., Nikiforou, M. A., Saito, M.,
Extensions of Quandles and Cocycle Knot Invariants,
J.~K.~T.~R.~\textbf{12} (2003), 725-738.

\item[{[CES]}] Carter, J. S., Elhamdadi, M., Saito, M.,
Twisted Quandle Homology Theory and Cocycle Knot Invariants,
Alg.~Geom.~Top.~\textbf{2} (2002), 95-135.

\item[{[CHNS]}] Carter, J. S., Harris, A., Nikiforou, M. A., Saito, M.,
Cocycle Knot Invariants, Quandle Extensions, and Alexander Matrices,
S\=urikaisekikenky\=usho K\=oky\=uroku \textbf{1272} (2002), 12-35.

\item[{[CJKLS]}] Carter, J. S., Jelsovsky, D., Kamada, S., 
Langford, L., Saito, M.,
Quandle Cohomology and State-sum Invariants of Knotted Curves and Surfaces,
Trans. Amer.~Math.~Soc.~\textbf{355} (2003), 3947-3989.

\item[{[CJKS1]}] Carter, J. S., Jelsovsky, D., Kamada, S., Saito, M.,
Quandle Homology Groups, their Betti Numbers, and Virtual Knots,
J.~Pure Appl.~Alg.~\textbf{57} (2001), 135-155.

\item[{[CJKS2]}] Carter, J. S., Jelsovsky, D., Kamada, S., Saito, M.,
Computations of Quandle Cocycle Invariants of Knotted Curves and Surfaces,
Adv.~Math.~\textbf{157} (2001), 36-94.

\item[{[CJKS3]}] Carter, J. S., Jelsovsky, D., Kamada, S., Saito, M.,
Shifting Homomorphisms in Quandle Cohomology 
and Skeins of Cocycle Knot Invariants,
J.~K.~T.~R. \textbf{10} (2001), 579-596.

\item[{[CKS1]}] Carter, J. S., Kamada, S., Saito, M.,
Geometric Interpretations of Quandle Homology,
J.~K.~T.~R.~\textbf{10} (2001), 345-386.

\item[{[CKS2]}] Carter, J. S., Kamada, S., Saito, M.,
Diagrammatic Computations for Quandles and Cocycle Knot Invariants,
Cont.~Math.~\textbf{318}, 51-74.

\item[{[E]}] Eisermann, M.,
Homological Characterization of the Unknot,
J.~Pure Appl. Alg. \textbf{177} (2003), 131-157.

\item[{[FR]}] Fenn, R., Rourke, C.,
Racks and Links in Codimension Two,
J.~K.~T.~R. \textbf{1} (1992), 343-406.

\item[{[FRS]}] Fenn, R., Rourke, C., Sanderson, B.,
Trunks and Classifying Spaces,
Appl. Cat.~Str.~\textbf{3} (1995), 321-356.

\item[{[G]}] Gra\~na, M.,
Quandle Knot Invariants are Quantum Knot Invariants,
J.~K.~T. R. \textbf{11} (2002), 673-681.

\item[{[J]}] Joyce, D.,
A Classifying Invariant of Knots, the Knot Quandle,
J.~Pure Appl. Alg.~\textbf{23} (1982), 37-65.

\item[{[K]}] Kamada, S.,
Knot Invariants derived from Quandles and Racks,
Geom.~Top. Mono.~\textbf{4} (2002), 103-117.

\item[{[KS]}] Kitano, T., Suzuki, M., 
A Partial Order in the Knot Table, 
Exp.~Math.~\textbf{14} (2005), 385-390.

\item[{[LN]}] Litherland, R. A., Nelson, S.,
The Betti Numbers of some Finite Racks,
J.~Pure Appl.~Alg.~\textbf{178} (2003), 187-202.

\item[{[M]}] Murasugi, K.,
On the Center of the Group of a Link,
Proc.~Amer.~Math.~Soc.~\textbf{16} (1965), 1052-1057.

\item[{[S]}] Satoh, S.,
Note on the Shadow Cocycle Invariant of a Knot with a Base Point,
J.~K.~T.~R. \textbf{16} (2007), 959-967.
\end{itemize}

\textsc{Graduate School of Mathematical Sciences,
The University of Tokyo,
3-8-1 Komaba, Meguro-ku, Tokyo 153-8914, Japan.}

\textsl{E-mail address}: \texttt{yasto@ms.u-tokyo.ac.jp}

\end{document}